\documentclass[11pt,leqno,oneside,letterpaper]{amsart}
\usepackage[letterpaper,left=1.75truein, top=1truein]{geometry}
\usepackage{amsmath,amsfonts,amssymb,amscd,amsthm,amsbsy,epsf,graphics}

\usepackage{color}

\usepackage[colorlinks,linkcolor=blue,citecolor=blue]{hyperref}
\usepackage{epsfig}
\usepackage{graphicx}

\newtheorem{thm}{Theorem}[section]
\newtheorem{cor}[thm]{Corollary}
\newtheorem{lemma}[thm]{Lemma}
\newtheorem{prop}[thm]{Proposition}
\newtheorem{conj}[thm]{Conjecture}

\newtheorem{que}[thm]{Question}
\newtheorem{rem}[thm]{Remark}
\newtheorem{claim}[thm]{Claim}

\theoremstyle{definition}

\def\zed{{\mathbb Z}}
\def\D{{\mathcal D}}

\newtheoremstyle{cases}
  {12pt plus 6 pt}%       Space above
  {2pt}%       Space below
  {\bfseries}   %       Body font
  {}%          Indent amount (empty = no indent, \parindent = para indent)
  {\bfseries}% Thm head font
  {.}%         Punctuation after thm head
  {.5em}%      Space after thm head: " " = normal interword space;
  {}%          Thm head spec (can be left empty, meaning `normal')
\theoremstyle{cases}

\numberwithin{subcase}{case} \numberwithin{subsubcase}{subcase}
\numberwithin{equation}{subsection}
\usepackage{numbering}

\def\sfrac#1#2{\kern.1em\raise.5ex\hbox{$#1$}
    \kern-.1em/\kern-.05em\lower.25ex\hbox{$#2$}}
\def\R{{\mathcal R}}

\def\wD{\widetilde{D}}
\def\R{{\mathcal R}}
\def\T{{\mathcal T}}
\def\wT{\widetilde{\T}}
\def\tilt{\tilde t}

\def\que{{\mathbb Q}}
\def\zed{{\mathbb Z}}

\def\zed{{\mathbb Z}}

\def\sfrac#1#2{\kern.1em\raise.5ex\hbox{$#1$}
        \kern-.1em/\kern-.05em\lower.25ex\hbox{$#2$}}

\def\T{{\mathcal T}}

\def\zed{{\mathbb Z}}

\def\G{{\Gamma}}
 \def\d{{\delta}}
 
 \def\e{{\epsilon}}
 
 \def\L{{\Lambda}}

   \def\s{{\sigma}}
 
 \def\a{{\alpha}}
 \def\b{{\beta}}
 \def\p{{\partial}}

 \def\g{{\gamma}}
 \def\D{{\Delta}}

 \def\2{{\mathbb Z_2}}
 
 \def\t{{\tau}}

 \def\sl2{{SL(2,\mathbb C)}}
 
 \def\qed{{\hspace{2mm}{\small $\diamondsuit$}}}

 \def\pf{{\noindent{\bf Proof.\hspace{2mm}}}}

 \def\sl{{{\mbox{\tiny $\L$}}}}

\begin{document}

\title{Dehn fillings of knot manifolds containing essential once-punctured tori}

\author[Steven Boyer]{Steven Boyer}
\thanks{Steven Boyer was partially supported by NSERC grant RGPIN 9446-2008}
\address{D\'epartement de Math\'ematiques, Universit\'e du Qu\'ebec \`a Montr\'eal, 201 avenue du Pr\'esident-Kennedy, Montr\'eal, QC H2X 3Y7.}
\email{boyer.steven@uqam.ca}
\urladdr{http://www.cirget.uqam.ca/boyer/boyer.html}

\author{Cameron McA. Gordon}
\thanks{Cameron Gordon was partially supported by NSF grant DMS-0906276.}
\address{Department of Mathematics, University of Texas at Austin, 1 University Station, Austin, TX 78712, USA.}
\email{gordon@math.utexas.edu}
\urladdr{http://www.ma.utexas.edu/text/webpages/gordon.html}

\author{Xingru Zhang}
\address{Department of Mathematics, University at Buffalo, Buffalo, NY, 14214-3093, USA.}
\email{xinzhang@buffalo.edu}
\urladdr{http://www.math.buffalo.edu/~xinzhang}

\maketitle
\vspace{-.6cm}
\begin{center}
\today
\end{center}

\maketitle

\begin{abstract}
In this paper we study exceptional Dehn fillings on hyperbolic knot
manifolds which contain an essential once-punctured torus.
Let $M$ be such a knot manifold and let $\beta$ be the boundary slope
of such an essential once-punctured torus. We prove that if Dehn filling
$M$ with slope $\alpha$ produces a Seifert fibred manifold,
then $\D(\alpha,\beta)\leq 5$. Furthermore we classify the triples $(M; \alpha,\beta)$
when $\D(\alpha,\beta)\geq 4$. More precisely, when $\D(\alpha,\beta)=5$, then $M$ is
the (unique) manifold $Wh(-3/2)$ obtained by Dehn filling one boundary
component of the Whitehead link exterior with slope $-3/2$,  and $(\alpha, \beta)$ is
the pair of slopes $(-5, 0)$. Further, $\D(\alpha,\beta)=4$ if and only if $(M; \alpha,\beta)$ is the
triple $\displaystyle (Wh(\frac{-2n\pm1}{n}); -4, 0)$ for some integer $n$ with $|n|>1$. Combining this with known results,
we classify all hyperbolic knot manifolds $M$ and pairs of slopes $(\beta, \gamma)$ on $\partial M$ where $\beta$ is the boundary slope of an
essential once-punctured torus in $M$ and $\gamma$ is an exceptional filling slope of distance $4$ or more from $\beta$.
Refined results in the special case of hyperbolic
genus one knot exteriors in $S^3$ are also given.
\end{abstract}

\section{Introduction}
This is the second of four papers in which we investigate the following conjecture of the second named author (see \cite[Conjecture 3.4]{Go2}). Recall that a {\it hyperbolic knot manifold} is a compact, connected, orientable $3$-manifold with torus boundary whose interior admits a complete, finite volume hyperbolic structure.

\begin{conj}\label{conj} {\rm (C. McA. Gordon)} Suppose that $M$ is a hyperbolic knot
manifold and $\alpha, \beta$ are slopes on $\partial M$ such that $M(\alpha)$ is Seifert fibred and $M(\beta)$ toroidal. If $\Delta(\alpha, \beta) > 5$, then $M$ is the figure eight knot exterior.
\end{conj}

Our first result reduces the verification of the conjecture to the case where the Seifert filling is atoroidal.

\begin{thm}\label{reduction2}
Suppose that $M$ is a hyperbolic knot
manifold and $\alpha, \beta$ are slopes on $\partial M$ such that
$M(\alpha)$ is a toroidal Seifert fibred manifold and $M(\beta)$
is toroidal. Then $\Delta (\alpha,\beta) \leq 4$.
Furthermore, if $\Delta (\alpha,\beta) =4 $ then $(M;\alpha,\beta)\cong
(N(-\frac12,-\frac12); -4,0)$ where $N$ is the exterior of the 3-chain link \cite{MP}.
\end{thm}

We have that $N(-\frac12,-\frac12,-4)$ is Seifert fibred with base orbifold
$P^2 (2,3)$, and $N(-\frac12,-\frac12,0)$ contains an incompressible
torus separating $N(-\frac12,-\frac12,0)$ into Seifert fibred manifolds
with base orbifolds $D^2 (2,2)$ and $D^2 (2,3)$.
(See \cite[Table 2]{MP}.)

A {\it small Seifert} manifold is a $3$-manifold which admits a Seifert structure with base orbifold of the form $S^2(a,b,c)$ where $a, b, c \geq 1$. For instance, a closed, atoroidal Seifert manifold is small Seifert.

A small Seifert manifold is a {\it prism manifold} if its base orbifold is $S^2(2,2,n)$ for some $n \geq 2$.

Since the distance between a toroidal filling slope and a reducible filling slope is at most $3$ (\cite{Oh}, \cite{Wu1}), Theorem \ref{reduction2} reduces our analysis of Conjecture \ref{conj} to understanding the case where the Seifert Dehn filling is irreducible and small Seifert.
In an earlier paper \cite{BGZ2} we verified the conjecture in the case where $M$ admits no essential punctured torus of boundary slope $\beta$ which is a fibre or  semi-fibre, or which has fewer than three boundary components; more precisely, we showed that in this case $\Delta(\alpha, \beta) \leq 5$. Here we focus on the case where $M$ admits an essential punctured torus with one boundary component.

Let {\it $Wh$} denote the left-handed Whitehead link exterior (see Figure \ref{bgz4-whitehead}). We parameterise the slopes on a boundary component of {\it $Wh$}  using the standard meridian-longitude coordinates.

\begin{thm} \label{once-punctured}
Let $M$ be a hyperbolic knot manifold and $\alpha$ a slope on $\partial M$ such that $M(\alpha)$ is small Seifert. If $M$ admits an essential, once-punctured torus $F$ of boundary slope $\beta$ then $\Delta(\alpha, \beta) \leq 5$. Further, if $\Delta(\alpha, \beta) > 3$, then $F$ is not a fibre and $\pi_1(M(\alpha))$ is finite. More precisely, \\
$(1)$ if $\Delta(\alpha, \beta) = 4$, then  $(M;\alpha,\beta)\cong (Wh(\frac{-2n\pm1}{n}); -4,0)$
 for some integer $n$ with $|n|>1$ and $M(\alpha)$ has base orbifold $S^2(2,2,|\mp 2n-1|)$, so
 $M(\alpha)$ is a prism manifold; \\
$(2)$ if $\Delta(\alpha, \beta) = 5$, then $(M; \alpha, \beta) \cong (Wh(-3/2); -5, 0)$, and $M(\alpha)$ has base orbifold $S^2(2,3,3)$.
\end{thm}

Baker \cite{Ba} has proven Theorem \ref{once-punctured} in the case that $M(\alpha)$ is a lens space. We provide an alternate proof of his result.

Theorem \ref{once-punctured} is sharp; see the infinite family of examples in \S \ref{d=4} for (1) and \cite[Table A.3]{MP} for (2).  Another family of examples is provided by hyperbolic twist knots. These are genus one knots in the $3$-sphere whose exteriors admit small Seifert filling slopes of distance $1, 2,$ and $3$ from the longitudinal slope. Finally, Baker \cite[Theorem 1.1(IV)]{Ba} has constructed an infinite family of non-fibred hyperbolic knot manifolds which admit a once-punctured essential torus whose boundary slope is of distance $3$ to a lens space filling slope.

Here is an outline of the proof of Theorem \ref{once-punctured}. We begin by showing that the result holds unless, perhaps, $M$ admits an orientation-preserving involution $\tau$ with non-empty branch set $L$ contained in the interior of the quotient $M/\tau$, which is a solid torus. The results of \cite{BGZ2} reduce us to the case that $L$ has a very particular form (see Figure \ref{fig3}). On the other hand, $\tau$ extends to an involution $\tau_\alpha$ of $M(\alpha)$ with branch set $L_\alpha$ contained in the lens space $M(\alpha)/\tau_\alpha$. The fundamental group of $M(\alpha)/\tau_\alpha$ is non-trivial if the distance between $\alpha$ and $\beta$ is at least $3$. Since the involutions on small Seifert manifolds with such quotients are well-understood, we can explicitly describe the branch set $L_\alpha$ of $\tau_\alpha$. Comparing this description with the constraints we have already deduced on $L$ leads to the proof of the theorem.

Recall that an {\it exceptional filling slope} on the boundary of a hyperbolic $3$-manifold is a slope $\gamma$ such that $M(\gamma)$ is not hyperbolic. Geometrisation of $3$-manifolds implies that a slope $\gamma$ is exceptional if and only if $M(\gamma)$ is either reducible, toroidal, or Seifert fibred. Theorem \ref{once-punctured} combines with \cite{Oh}, \cite{Wu1}, \cite{Go1}, \cite{GW}, and Proposition \ref{compresses} to yield  the next result.

\begin{thm} \label{once-punctured-exceptional}
Let $M$ be a hyperbolic knot manifold which admits an essential,
once-punctured torus $F$ of boundary slope $\beta$ and let $\gamma$ be an exceptional filling slope on $\partial M$. \\
$(1)$ $\Delta(\gamma, \beta) \leq 7$. \\
$(2)$ If $\Delta(\gamma, \beta) > 3$, then $M(\gamma)$ is either toroidal or has a finite fundamental group.  \\
$(3)$ If $\Delta(\gamma, \beta) > 3$ and $M(\gamma)$ is toroidal, then either \\
\indent \hspace{.3cm} $(a)$ $\Delta(\gamma, \beta) = 4$ and $(M; \gamma, \beta) \cong (Wh(\delta); -4, 0)$ for some slope $\delta$; or \\
\indent \hspace{.3cm} $(b)$ $\Delta(\gamma, \beta) = 5$ and $(M; \gamma, \beta) \cong (Wh(-4/3); -5, 0)$ or $(Wh(-7/2); -5/2, 0)$; or  \\
\indent \hspace{.3cm} $(c)$ $\Delta(\gamma, \beta) = 7$ and $(M; \gamma, \beta) \cong (Wh(-5/2); -7/2, 0)$. \\
$(4)$ If $\Delta(\gamma, \beta) > 3$ and $\pi_1(M(\gamma))$ is finite, then either \\
\indent \hspace{.3cm} $(a)$ $\Delta(\gamma, \beta) = 4$, $(M;\gamma,\beta)\cong (Wh(\frac{-2n\pm1}{n}); -4,0)$
 for some integer $n$ with $|n|>1$, and $M(\gamma)$ has base orbifold $S^2(2,2,|\mp 2n-1|)$; or
\\
\indent \hspace{.3cm} $(b)$ $\Delta(\gamma, \beta) = 5, (M; \gamma, \beta) \cong (Wh(-3/2); -5, 0)$, and $M(\gamma)$ has base orbifold $S^2(2,3,3)$.
\end{thm}

Next we specialize to the case where $M$ is the exterior of a hyperbolic knot in the $3$-sphere.

\begin{thm} \label{genusones3}
Let $K \subset S^3$ be a hyperbolic knot of genus one with exterior $M_K$ and suppose $p/q$ is an exceptional filling slope on $\partial M_K$. \\
$(1)$ $M_K(0)$ is toroidal but not Seifert. \\
$(2)$ $M_K(p/q)$ is either toroidal or small Seifert with hyperbolic base orbifold.  \\
$(3)$ If $M_K(p/q)$ is small Seifert with hyperbolic base orbifold, then $0 < |p| \leq 3$.    \\
$(4)$ If $M_K(p/q)$ is toroidal, then $|q| = 1$ and $|p| \leq 4$ with equality implying $K$ is a twist knot.
\end{thm}

Here is how the paper is organised. We prove Theorem \ref{reduction2} in \S \ref{sec: reduction2}. In \S \ref{background} we show that there are strong topological constraints on $M$ which must be satisfied if Theorem \ref{once-punctured} doesn't hold. These constraints will be applied later in the paper to construct an involution on $M$. In \S \ref{involutions} we describe the branching set of an orientation-preserving involution on a small Seifert manifold with quotient space a lens space with non-trivial fundamental group. Using this we reduce the proof of Theorem \ref{once-punctured} to five problems involving links in lens spaces in \S \ref{sec: once-punctured},
and a problem in which $\D(\alpha,\beta)=4$ and $M(\a)$ is a prism manifold. These problems are resolved in  \S \ref{m=1 Seifert}, \S \ref{dist7}, \S \ref{lens space delta = 5}, \S \ref{sec6.3},  \S \ref{delta = 5} and \S \ref{prism-section} respectively. The infinite family of examples realising  distance $4$ in Theorem \ref{once-punctured} is constructed in \S \ref{d=4}. Theorems \ref{once-punctured-exceptional} and \ref{genusones3} are dealt with in \S \ref{sec: genus 1}.

\section{The case where $M(\alpha)$ is toroidal} \label{sec: reduction2}

In this section we prove Theorem \ref{reduction2}. Recall from the introduction that $N$ denotes the exterior of the 3-chain link of \cite{MP}. Note that $N(-\frac12,-\frac12)$ is obtained by Dehn filling on $N(-\frac12)$,
which is the exterior of the rational link associated with the rational
number $10/3$.

To prove Theorem \ref{reduction2} we consider all $(M;\alpha,\beta)$ where
$M$ is hyperbolic, $M(\alpha)$ and $M(\beta)$ are toroidal and
$\Delta (\alpha,\beta)\ge 4$.
For $\Delta (\alpha,\beta)\ge 6$ there are only four such $(M;\alpha,\beta)$
\cite{Go1}, and in all four cases neither $M(\alpha)$ nor $M(\beta)$ is Seifert
fibred.

For $\Delta (\alpha,\beta) = 4$ or 5, the triples $(M;\alpha,\beta)$ are
determined in \cite{GW}:
there are 14 hyperbolic manifolds $M_i$, $1\le i\le 14$, each with a pair
of toroidal filling slopes $\alpha_i,\beta_i$ at distance 4 or 5, where
$M_1,M_2,M_3$ and $M_{14}$ have two (torus) boundary components, and the
others, one.
It is shown in \cite{GW} that a hyperbolic manifold $M$ has two toroidal
filling slopes $\alpha$ and $\beta$ at distance 4 or 5 if and only if
$(M;\alpha,\beta) \cong (M_i;\alpha_i,\beta_i)$ for some $1\le i\le 14$,
or $(M;\alpha,\beta) \cong (M_i(\gamma);\alpha_i,\beta_i)$ for $i= 1,2,3$ or
14 and some slope $\gamma$ on the second boundary component of $M_i$.
(We adopt the convention that in the above homeomorphisms either
$\alpha \mapsto \alpha_i$, $\beta\mapsto\beta_i$, or $\alpha\mapsto \beta_i$,
$\beta\mapsto \alpha_i$.)
We prove Theorem \ref{reduction2} by showing that firstly, for $i\ne 1,2,3$
or 14, neither of the toroidal manifolds $M_i (\alpha_i)$ or $M_i(\beta_i)$
is Seifert fibred, secondly, for $i = 1,3$ or 14, there is no hyperbolic
manifold of the form $M_i(\gamma)$ with either $M_i(\gamma)(\alpha_i)$
or $M_i (\gamma)(\beta_i)$ toroidal Seifert fibred, and thirdly, there is
a unique example $(M_2(\gamma);\alpha_2,\beta_2)$ (up to homeomorphism)
where $M_2(\gamma)$ is hyperbolic, $M_2(\gamma)(\alpha_2)$ and
$M_2(\gamma)(\beta_2)$ are toroidal, and one is Seifert fibred; this is the
example described in Theorem \ref{reduction2}.

We first consider the manifolds $M_i$, $6\le i\le 13$.
The toroidal fillings on $M_i$, $M_i(0)$ and $M_i(\beta_i)$, are
described in Lemma~22.2 of \cite{GW}.
We adopt the notation introduced in \cite[p.116]{GW}.

\begin{lemma}\label{lem1}
For $6 \le i\le 13$, $M_i(0)$ is not Seifert fibred.
\end{lemma}

\begin{proof}
$M_i(0)$ is of the form $X(p_1,q_1;p_2,q_2)$; it is the double branched
cover of the tangle $Q_i(0)$, which is of the form $T(p_1,q_1;p_2,q_2)$,
the union of two Montesinos tangles.
Assume the numbering is chosen so that $p_1,q_1$ are not both 2;
(actually this is only an issue when $i=8$).
Then the Seifert fibre $\varphi_1$ of $X(p_1,q_1)$ is unique.
Since $X(p_1,q_1)$ and $X(p_2,q_2)$ are not both twisted $I$-bundles, to
show that $M_i(0)$ is not Seifert fibred it suffices to  show that, in the
gluing of $X(p_1,q_1)$ and $X(p_2,q_2)$, $\varphi_1$ is not identified
with the Seifert fibre $\varphi_2$ of $X(p_2,q_2)$.
(When $i=8$, $p_2 = q_2 =2$ and there are two possible choices for
$\varphi_2$.)
We do this by identifying the image of $\varphi_1$ in the boundary of the
tangle $T(p_1,q_1)$, and then capping off the tangle $T(p_2,q_2)$ with the
corresponding rational tangle; in the double branched cover this corresponds
to doing Dehn filling on $X(p_2,q_2)$ along the slope $\varphi_1$.
If $M_i(0)$ were Seifert fibred then this Dehn filling would be reducible,
and so the corresponding rational tangle filling on $T(p_2,q_2)$ would
give a link that is either composite or split.
One checks that this is not the case.
\end{proof}

\begin{lemma}\label{lem2}
For $6\le i\le 13$, $M_i(\beta_i)$ is not Seifert fibred.
\end{lemma}

\begin{proof}
First note that $M_7 (\beta_7)$ is of the form $X(2,3;2,2)$.
We check that this is not Seifert fibred in the same way as we did
for $M_8(0)$ in Lemma~\ref{lem1}.

When $i\ne 7$, $M_i (\beta_i)$ is the double branched cover of a
2-component link $L_i$;
see \cite[Lemma~22.2]{GW}.
More specifically,
for $i = 6,8,9$ or 12, $L_i$ is a cabled Hopf link $C(p_1,q_1;p_2,q_2)$
with $p_1,p_2 >1$,
for $i=10$ or 11, $L_i$ is the link $C(C;2,1)$ (see \cite[page 116]{GW}), and
for $i=13$, $L_i$ is the 2-string cable of the trefoil shown in
\cite[Figure~22.13(d)]{GW}.
In all cases, $L_i$ is {\em toroidal}, i.e. its exterior contains an
essential torus. Moreover the exterior of $L_i$ is not Seifert fibred.
Therefore if $M_i(\beta_i)$ were Seifert fibred then $L_i$ would be a Montesinos link.
But the only toroidal Montesinos links are (see \cite[Corollary~5]{Oe})
$K(\frac12,\frac12,-\frac12,-\frac12)$,
$K(\frac23,-\frac13,-\frac13)$,
$K(\frac12,-\frac14,-\frac14)$,
and $K(\frac12,-\frac13,-\frac16)$.
One easily checks that no $L_i$ is of this form.
\end{proof}

\begin{lemma}\label{lem3}
$M_4 (\alpha_4)$ and $M_4 (\beta_4)$ are not Seifert fibred.
\end{lemma}

\begin{proof}
$M_4 (\alpha_4)$ and $M_4(\beta_4)$ contain incompressible tori
$\widehat F_a$ and $\widehat F_b$; the corresponding punctured tori
$F_a$ and $F_b$ in $M_4$ have four and two boundary components, respectively.
The intersection of $F_a$ and $F_b$ is described by the intersection graphs
$\Gamma_a \subset \widehat F_a$ and $\Gamma_b\subset \widehat F_b$
depicted in Figures~11.9(a) and (b) of \cite{GW}, respectively.
Note that $\widehat F_a$ separates $M_4 (\alpha_4)$, into $M_B$ and $M_W$,
say, while $\widehat F_b$ is non-separating in $M_4(\beta_4)$.
The faces of the graph $\Gamma_b$ lie alternately in $M_B$ and $M_W$;
we choose the notation so that all the faces of $\Gamma_b$ that lie in
$M_B$ are bigons.

Let $f_1, f_2, f_3$, and $g_1, g_2, g_3$ be the faces of $\Gamma_b$ with edges
$G,H; J,K; A,B;$ and $D,E; K,P,R; A,G,L;$ respectively. Let $h_1, h_2, h_3$ be
the faces of $\Gamma_a$ with edges $E,N; H,E;$ and $B,G,N,R;$ respectively.
(The notation refers to the edges illustrated in Figure 11.9 of [GW].)

For computations in $\pi_1(M_B)$ and $\pi_1(M_W)$ we take as
``base-point'' the rectangle in $\widehat F_a$ shown in Figure~11.9(a)
of \cite{GW}.
Let $s,t$ be the pair of generators of $\pi_1 (\widehat F_a)$ determined
by the downward vertical and rightward horizontal edges of that rectangle,
respectively.
Let $x_1$ and $x_3$ be the elements of $\pi_1(M_B)$ corresponding to the
1-handles $H_{(12)}$ and $H_{(34)}$, in the usual way.
The faces $f_1,f_2$ and $f_3$ give the relations in
$\pi_1 (M_B)$:
\begin{align*}
&x_1^2 t =1\\
&x_3^2 t^{-1} =1\\
&s^{-1}x_3 x_1 =1
\end{align*}
It follows that $M_B$ is Seifert fibred with base orbifold $D^2 (2,2)$, and
that the classes in $\pi_1 (\widehat F_a)$ of the Seifert fibres in the
two Seifert fibrings of $M_B$ are $t$ and $s$.

Let $x_2$ and $x_4$ be the elements of $\pi_1 (M_W)$ corresponding to
$H_{(23)}$ and $H_{(41)}$.
Then the faces $g_1,g_2$ and $g_3$ give the relations in $\pi_1 (M_W)$:
\begin{align*}
&tx_4 x_2 =1\\
&x_2 x_4 t^{-1} x_2 st =1\\
&x_2 x_4^2 t^{-1} =1
\end{align*}

These show that $M_W$ is Seifert fibred with base orbifold $D^2 (2,3)$,
the class of the Seifert fibre in $\pi_1 (\widehat F_a)$ being $st^2$.
Since this is distinct from either of the Seifert fibres of $M_B$,
$M_4 (\alpha_4)$ is not Seifert fibred.

We now consider $M_4 (\beta_4)$.
Let $u,v$ be the pair of generators for $\pi_1 (\widehat F_b)$ given by
the downward vertical and leftward horizontal edges of the rectangle in
Figure~11.9(b) of \cite{GW}.
(We take this rectangle as ``base-point'' for computations in
$\pi_1 (M_4 (\beta_4))$.)
Let $x,y$ be the elements of $\pi_1 (M_4 (\beta_4))$ given by the
1-handles $H_{(12)}$ and $H_{(21)}$.
The faces $h_1,h_2,h_3$ give the relations in $\pi_1(M_4(\beta_4))$:
\begin{align*}
&x(uv) y^{-1} v^{-1} = 1\\
&yv x^{-1} =1\\
& x^{-1} u^{-1} xux^{-1} (vu)^{-1} y =1
\end{align*}
The second relation gives $x= yv$, and the first then gives
$$y^{-1} vy = uv^2$$
The third relation gives
$$(y^{-1} u^{-1} y) u (y^{-1} u^{-1} y) u^{-1}v^{-3} =1$$

Now if $M_4 (\beta_4)$ were Seifert fibred, the non-separating torus
$\widehat F_b$ would be horizontal and so $M_4 (\beta_4)$ would be a
torus bundle over the circle with fibre $\widehat F_b$.
Hence $y^{-1} u^{-1}y$ would belong to $\pi_1 (\widehat F_b)$.
But the last relation above shows that if this is the case then
$$(y^{-1} u^{-1}y)^2 = v^3$$
Since $v^3$ is not a square in $\pi_1 (\widehat F_b)$, this is a
contradiction.
\end{proof}

\begin{lemma}\label{lem4}
$M_5 (\alpha_5)$ and $M_5 (\beta_5)$ are not Seifert fibred.
\end{lemma}

\begin{proof}
This can be proved in a similar fashion to Lemma~\ref{lem3}, using
\cite[Figure~11.10]{GW}.
Another way to establish the result is to note that, according to
\cite[\S6]{L2}, $M_5 \cong N(1,-\frac13)$, the toroidal filling slopes
$\alpha_5,\beta_5$ being $-4$ and 1.
We see that $N(1,-\frac13,-4)$ and $N(1,-\frac13,1)$ are not Seifert fibred
from Tables~4 and 3 of \cite{MP}, respectively.
\end{proof}

We next consider the manifolds $M_1,M_2$ and $M_3$, namely the exteriors
of the Whitehead link, the $10/3$-rational link, and the Whitehead sister
(or $(-2,3,8)$-pretzel) link, respectively.
These are all obtained by Dehn filling on the 3-chain link:
$M_1\cong N(1)$, $M_2 \cong N(-\frac12)$, $M_3 \cong N(-4)$.
Furthermore, their exceptional slopes and toroidal slopes are as
follows (see \cite[Table~A.1]{MP}):
$$\vbox{\offinterlineskip
\halign{\strut
\vrule#&\enspace \hfil$#$\hfil\enspace
&\vrule#&\enspace$#$\hfil\enspace
&\vrule#&\enspace\hfil$#$\hfil\enspace
&\vrule#\cr
\noalign{\hrule}
&&&\hbox{exceptional slopes}&&\hbox{toroidal slopes}&\cr
\noalign{\hrule}
&N(1)&&\infty ,-3,-2,-1,0,1&&-3,1&\cr
\noalign{\hrule}
&N(-\frac12)&&\infty,-4,-3,-2,-1,0&&-4,0&\cr
\noalign{\hrule}
&N(-4)&&\infty,-3,-2,-1,-\frac12,0&&-\frac12,0&\cr
\noalign{\hrule}
}}$$

\begin{lemma}\label{lem5}
In each of the following cases, the manifold $N(\alpha,\beta,\gamma)$
is a toroidal Seifert fibre space if and only if $\gamma$ is one of
the values listed.

$(a)$ $N(1,-3,\gamma) :\hskip9pt $\quad $\gamma = -3,1$; \\
\indent \hspace{5mm} $N(1,1,\gamma) :\hskip17pt $\quad $\gamma = -3, -2,-1,0$.

$(b)$ $N(-\frac12,-4,\gamma) :$\quad $\gamma = -\frac12$;\\
\indent \hspace{5mm} $N(-\frac12,0,\gamma) :\hskip9pt $\quad $\gamma = -\frac72$.

$(c)$ $N(-4,-\frac12,\gamma) : $\quad $\gamma = -\frac12$;\\
\indent \hspace{5mm} $N(-4,0,\gamma) :\hskip9pt $\quad no $\gamma$.

\end{lemma}

\begin{proof}
This follows by inspecting Tables~2, 3 and 4 of \cite{MP}.
We see from these that the only toroidal Seifert fibre spaces
$N(\alpha,\beta,\gamma)$ are

(1) $N(-3,1,1),\ \ N(-3,-\tfrac53,-\tfrac53),\ \
N(-3,-3,t/u)
\text{ where } t/u \ne-1, -1 +\tfrac1m \text{ or }\infty,$ and

(2) $N(0,\frac12 +n ,-\frac92 -n)$,
$N(1,1,n)$ where $|n+1| \le 1$,
$N(-\frac32,-\frac52,0)$, and $N(-4,-\frac12,-\frac12)$.
\end{proof}

Note that the values of $\gamma$ listed in parts (a) and (c) of
Lemma~\ref{lem5} all belong to the set of exceptional slopes of
$N(1)$ and $N(-4)$ respectively.
It follows that for $i=1$ and 3, there is no $\gamma$ such that
$M_i (\gamma)$ is hyperbolic and one of $M_i(\gamma)(\alpha_i)$,
$M_i(\gamma)(\beta_i)$ is toroidal Seifert fibred.

In the case $i=2$, note that by \cite[Proposition 1.5 part (1.4)]{MP},
there is an automorphism of $N(-\frac12)$ inducing homeomorphisms
\begin{align*}
&N(-\tfrac12, -4,-\tfrac12)  \cong N(-\tfrac12, 0,-\tfrac72)\\
\noalign{\vskip6pt}
&N( -\tfrac12, 0,-\tfrac12)  \cong N(-\tfrac12,-4,-\tfrac72)
\end{align*}
Also, we see from \cite[Table 2]{MP} that $N(-\frac12,0,-\frac12)$ is toroidal.
Thus part~(b) of Lemma~\ref{lem5} gives rise to the single example described
in Theorem \ref{reduction2}.

Finally, we take care of $M_{14}$:

\begin{lemma}\label{lem56}
For no slope $\gamma$ on the second boundary component of $M_{14}$
is $M_{14}(\gamma)(\alpha_{14})$ or $M_{14}(\gamma)(\beta_{14})$ toroidal
Seifert fibred.
\end{lemma}

\begin{proof}
In \cite{L1} Lee describes a hyperbolic 3-manifold $Y$ with two torus
boundary components having (homeomorphic) Dehn fillings $Y(0)$ and $Y(4)$
that contain Klein bottles.
In fact $Y(0) \cong Y(4) \cong Q(2,2) \cup Wh$,
where $Q (2,2)$ is the Seifert fibre space with base orbifold $D^2 (2,2)$
and $Wh$ is the exterior of the Whitehead link.
Hence $Y(0) \cong Y(4)$ is toroidal.
It follows from the classification in \cite{GW} of the hyperbolic 3-manifolds
with toroidal fillings at distance~4 that $Y\cong M_{14}$.
(The only other manifolds with two boundary components having toroidal
fillings at distance~4 are $M_1$ and $M_2$, and there the toroidal fillings
are graph manifolds; see e.g. \cite[Table~A.1]{MP}.)
It therefore suffices to show that $M_{14}(\gamma)(\alpha_{14})$ is not
toroidal Seifert fibred for any slope $\gamma$.

The manifold $M = M_{14}(\alpha_{14}) \cong Q(2,2)\cup Wh$ is the double
branched cover of the tangle shown in \cite[Figure~22.14(b)]{GW}.
Thus $M(\gamma) \cong Q(2,2) \cup Wh(\gamma)$.
Hence if  $M(\gamma)$ is toroidal Seifert fibred then $\gamma$ must be an
exceptional slope for $Wh$.
These slopes (with respect to the parametrization in
\cite[Table~A.1]{MP}) are $\infty,-3,-2,-1,0$ and $1$.
Now $Wh(-3)$ and $Wh(1)$ are toroidal non-Seifert,
$Wh(\infty) \cong D^2 \times S^1$, and $Wh(-2)$, $Wh(-1)$ and $Wh(0)$ are
Seifert fibred with base orbifold $D^2 (3,3)$, $D^2(2,4)$ and $D^2(2,3)$,
respectively.
So we need only consider $M(\gamma)$ for $\gamma =\infty,-2,-1$ and $0$;
we do this by examining the corresponding rational tangle filling on the
tangle shown in \cite[Figure~22.14(b)]{GW}.
For $\gamma =\infty$, this yields the pretzel knot
$K(-\frac12,-\frac12,\frac12)$, so $M(\infty)$ is atoroidal.
For $\gamma = -2,-1$ and $0$ we show that the Seifert fibre of $Wh(\gamma)$
does not match the Seifert fibre in either of the two Seifert fibrings
of $Q(2,2)$.
This is straightforward to check, for example using the same approach
as in the proof of Lemma~\ref{lem1}.
\end{proof}

\section{Background Results for the Proof of  Theorem \ref{once-punctured}} \label{background}

We collect various results in this section and the next which will be used throughout this paper and its sequel \cite{BGZ3}. In what follows, $M$ will be a hyperbolic knot manifold and $b_1(M)$ will denote its first Betti number. In this section we assume that $F$ is an essential, punctured torus of slope $\beta$ which is properly embedded in $M$.

For a closed, essential surface $S$ in $M$ we define ${\mathcal{C}}(S)$ to be the set
of slopes $\delta$ on $\partial M$ such that $S$ compresses in $M(\delta)$. A
slope $\eta$ on $\partial M$ is called a {\it singular slope} for $S$ if $\eta
\in {\mathcal{C}}(S)$ and $\Delta(\delta, \eta) \leq 1$ for each $\delta \in \mathcal{C}(S)$. A result of Wu \cite{Wu2} states that if ${\mathcal{C}}(S) \ne \emptyset$,
then there is at least one singular slope for $S$.

\begin{prop} \label{compresses}
Suppose that $M$ admits a non-separating, essential, genus $1$ surface of boundary slope $\beta$ which caps-off to a compressible torus in $M(\beta)$. If $\gamma$ is a slope on $\partial M$ such that $M(\gamma)$ is not hyperbolic, then $\Delta(\gamma, \beta) \leq 3$. If $M(\gamma)$ is an irreducible, atoroidal, small Seifert manifold, then $\Delta(\gamma, \beta) \leq 1$.
\end{prop}

\begin{proof}
By hypothesis $M(\beta)$ admits a non-separating $2$-sphere and so is reducible with first Betti number at least $1$.
In the case that $b_1(M) \geq 2$, there is a closed essential surface $S \subset \hbox{int}(M)$ which is Thurston norm minimizing in $H_2(M)$. By \cite[Corollary]{Ga}, $S$ is essential and Thurston norm minimizing in $H_2(M(\delta))$ for all slopes $\delta \ne \beta$. By \cite[Proposition 5.1]{BGZ1}, $\Delta(\gamma, \beta) \leq 1$ for any slope $\gamma$ such that $M(\gamma)$ is not hyperbolic.  Suppose then that $b_1(M) = 1$ and note that by hypothesis $\beta$ is a strict boundary slope. In this case \cite[Theorem 3.2]{BCSZ2} implies that $\beta$ is a singular slope and so the conclusions of the lemma follow from \cite[Theorem 1.5]{BGZ1}.
\end{proof}

\begin{cor} \label{stays incompressible}
Theorem \ref{once-punctured} holds if $M$ admits a non-separating, essential, genus $1$ surface of boundary slope $\beta$ which caps-off  to a compressible torus in $M(\beta)$.
\qed
\end{cor}

The torus in $M(\beta)$ obtained by capping-off $F$ with a meridional disks will be denoted $\widehat F$. We use $M_F$ to denote the compact manifold obtained by cutting $M$ open along $F$ and $M(\beta)_{\widehat F}$ the manifold obtained by cutting $M(\beta)$ open along $\widehat F$.

\begin{prop} \label{reduction}
Suppose that $M(\alpha)$ is a Seifert fibred manifold and $M(\beta)$ is toroidal. Then $\Delta(\alpha, \beta) \leq 3$ as long as one of the following conditions is satisfied:

\noindent $(a)$ $\alpha$ or $\beta$ is a singular slope of a closed essential surface in $M$.

\noindent $(b)$ $M(\alpha)$ or $M(\beta)$ is reducible.

\noindent $(c)$ $(i)$ $|\partial F| = 1$ and $M_F$ is not a genus $2$ handlebody.

\indent \hspace{.4cm} $(ii)$ $|\partial F| = 2$ and $M_F$ is neither connected nor a union of two genus $2$ handlebodies.

\end{prop}

\begin{proof} If $\alpha$ or $\beta$ is a singular slope of a closed essential surface in $M$, then \cite[Corollary 1.6]{BGZ1} shows that $\Delta(\alpha, \beta) \leq 3$, so we are done in case (a).

Assume next that $M(\gamma)$ is reducible where $\gamma$ is one of $\alpha$ or $\beta$. If $\gamma = \alpha$, then $\Delta(\alpha, \beta) \leq 3$ by \cite{Oh} and \cite{Wu1}. Assume then that $\gamma = \beta$. If $b_1(M) \geq 2$, then $\Delta(\gamma, \beta) \leq 1$ for any exceptional slope $\gamma$ as in the proof of Proposition \ref{compresses}. Assume then that $b_1(M) = 1$. Since $M(\beta)$ is toroidal, it is neither $S^1 \times S^2$ nor a connected sum of lens spaces. Hence \cite[Proposition 6.2]{BGZ1} implies that $\beta$ is a singular slope of a closed essential surface in $M$. Thus we are done by part (a).

Finally consider part (c) of the proposition. If $|\partial F| = 1$, any compression of $\partial M_F$ in $M_F$ yields one or two tori, so as $M$ is hyperbolic it is not hard to see that $M_F$ is a handlebody, contrary to hypothesis. Thus $\partial M_F$ is incompressible in $M_F$, and hence in $M$. Let $S \subset \hbox{int}(M)$ be the inner boundary component of a collar of $\partial M_F$ in $M_F$.
Then $S$ is incompressible in $M$, and by construction there is an annulus $A$ in $M$ with boundary components $\partial_1 A$ and $\partial_2  A$, say, where $A \cap S = \partial_1 A$ and $A \cap \partial M = \partial_2  A$ has slope $\beta$ on $\partial M$. It follows from \cite{Sh} that $S$ is incompressible in $M(\gamma)$ whenever $\Delta(\gamma,\beta) > 1$. Thus $\beta$ is a singular slope for $S$ and so part (a) of this proposition shows $\Delta(\alpha,\beta) \le 3$. Thus (i) holds.

If $|\partial F| = 2$ and $M_F$ is not connected, then $M = X_1 \cup_{F} X_2$ where $\partial X_j$ is a genus $2$ surface for $j = 1, 2$. If $\partial X_j$ compresses in $X_j$ for both $j$, then $X_1$ and $X_2$ are genus $2$ handlebodies as $M$ is hyperbolic. Since this possibility is excluded by our hypotheses, $\partial X_j$ is incompressible in $X_j$ for some $j$. Then it is essential in $M$ but compresses in $M(\beta)$, so as in the previous paragraph, $\beta$ is a singular slope for $\partial X_j$. Thus $\Delta(\alpha, \beta) \leq 3$. This completes the proof.
\end{proof}

Theorem \ref{reduction2} and Propositions \ref{compresses} and \ref{reduction} yield the following corollary.

\begin{cor} \label{smallseifert}
Conjecture \ref{conj} holds as long as it holds when $M(\alpha)$ is an irreducible, atoroidal, small Seifert manifold.
\qed
\end{cor}

Here is a result from \cite{BGZ2}. Recall from \S 6 of that paper that $t_j^+$ is the number of {\it tight components} of $\breve{\Phi}_j^+$.

A $3$-manifold is {\it very small} if its fundamental group does not contain a non-abelian free group.

\begin{prop}  \label{main1}
Suppose that $F$ is a once-punctured essential genus $1$ surface of boundary slope $\beta$ in a hyperbolic knot manifold $M$ which completes to an essential torus in $M(\beta)$ but is not a fibre in $M$. If $M(\alpha)$ is a small Seifert manifold, then
$$\Delta(\alpha, \beta) \leq \left\{ \begin{array}{ll} 6 & \hbox{if $M(\alpha)$ is very small} \\ 8 & \hbox{otherwise}   \end{array}  \right. $$
Moreover if $t_1^+ > 0$, then
$$\Delta(\alpha, \beta) \leq \left\{ \begin{array}{ll} 3 & \hbox{if $M(\alpha)$ is very small} \\ 4 & \hbox{otherwise}   \end{array}  \right. $$
\end{prop}

\begin{rem} \label{t1+ > 0}
{\rm  When $t_1^+ = 0$, $M(\beta)_{\widehat F}$
 is Seifert with base orbifold an annulus with one cone point \cite[Lemma 7.9]{BGZ2}. }
\end{rem}

\begin{proof}[Proof of Proposition \ref{main1}] The first inequality is the conclusion of \cite[Proposition 13.2]{BGZ2}. To deduce the second we use the notation and results of \cite{BGZ2}.

Suppose next that $t_1^+ > 0$. Since $t_1^+$ is even and the number of boundary components $F$ is bounded below by $\frac{1}{2} t_1^+$, we have $t_1^+ = 2$. Proposition 13.1 of \cite{BGZ2} then shows that $\Delta(\alpha, \beta) \leq 4$. Suppose that $M(\alpha)$ is very small.
The first paragraph of the proof of \cite[Proposition 13.1]{BGZ2} shows that $\Delta(\alpha, \beta) \leq 3$ if $\overline{\Gamma}_S$ has a vertex of valency $3$ or less while the second shows the same inequality holds if it doesn't. This completes the proposition's proof.
\end{proof}

\section{Involutions on small Seifert manifolds} \label{involutions}

We collect several results about involutions on small Seifert manifolds in this section.

\begin{lemma} \label{invariant}
Let $W$ be a small Seifert manifold and $\tau$ an orientation-preserving involution on $W$ with non-empty fixed point set. Then there is a $\tau$-invariant Seifert structure on $W$ with base orbifold of the form $S^2(a,b,c)$ where $1 \leq a \leq b \leq c$.
\end{lemma}

\begin{proof} If $W$ is a lens space, the result follows from \cite{HR}.
Assume then that this isn't the case and fix a Seifert structure on $W$ with base orbifold $S^2(a,b,c)$ where
$a \leq b \leq c$. The assumption that $\pi_1(W)$ is not cyclic implies
that $a \geq 2$ and $a, b, c$ are determined by $W$.

Let $L \subset W/\tau$ be the branch set of $\tau$. The orbifold theorem implies that the orbifold $W/ \tau$ is geometric and since $L$ is a link, $W/\tau$ admits a Seifert structure with a $2$-dimensional base orbifold \cite{Du}. Thus $W$ admits a $\tau$-invariant Seifert structure. We claim that we can assume this structure has base orbifold $S^2(a,b,c)$. If $b \ne 2$, all Seifert structures on $W$ have this form, so assume $a = b = 2 \leq c$. If the base orbifold of the $\tau$-invariant structure is not $S^2(a,b,c)$, it must be $P^2(d)$ for some integer $d \geq 1$. When $d > 1$, there is a unique singular fibre $\phi$ in this structure, and it must be invariant under $\tau$. Then $\tau$ leaves the exterior $E$ of this fibre invariant, which is a twisted $I$-bundle over the Klein bottle. By assumption, $\tau$ leaves the Seifert  structure on $E$ with base orbifold a M\"{o}bius band invariant. There is exactly one other Seifert structure on $E$, up to isotopy, and its base orbifold is $D^2(2,2)$. Moreover, there is at least one such structure which is $\tau|E$-invariant. This structure can be extended across a fibred neighbourhood of $\phi$ in a $\tau$-invariant fashion yielding the desired $\tau$-invariant structure on $W$.

The argument is similar if $d = 1$, for $\tau$ induces an involution of the base orbifold $P^2$ of $W$, and since any self-map of $P^2$ has a fixed point, there is a $\tau$-invariant fibre $\phi$ in $W$. Now proceed as in the case $d > 1$.
\end{proof}

For our next three results we let $W$ denote a small Seifert manifold and $\tau$ an orientation-preserving involution on $W$ with non-empty fixed point set such that the quotient $W/\tau$ is a lens space $L(\bar p, \bar q) \not \cong S^3$. We use $L_\tau$ to denote the branch set of $\tau$ in $L(\bar p, \bar q)$.

Fix a $\tau$-invariant Seifert structure on $W$ with base orbifold of the form $S^2(a,b,c)$ where $1 \leq a \leq b \leq c$ (Lemma \ref{invariant}) and let $\bar \tau$ be the involution of $S^2(a,b,c)$ (possibly the identity) induced by $\tau$.

Since the $\tau$-invariant Seifert structure on $W$ has an orientable base orbifold, its fibres can be coherently oriented.

Hodgson and Rubinstein have classified orientation-preserving involutions on lens spaces with non-empty fixed point sets. In particular, their work yields the following result.

\begin{lemma} {\rm (\cite[\S 4.7]{HR})} \label{cyclic case}
Suppose that $W$ is the lens space $L(p,q)$ and $W / \tau = L(\bar p, \bar q) \not \cong S^3$.

$(1)$ If $p$ is odd, then $L_\tau$ is connected and is either
\vspace{-.35cm}
\begin{itemize}

\item[(a)] the core of a solid torus of a genus one Heegaard splitting of $L(\bar p, \bar q)$;

\item[(b)] the boundary of a M\"{o}bius band spine of a Heegaard solid torus of $L(\bar p, \bar q)$;

\end{itemize}
\vspace{-.35cm}
$(2)$ If $p$ is even, then $L_\tau$ has two components and is either
\vspace{-.35cm}
\begin{itemize}

\item[(a)] the union of the cores of the two solid tori of a genus one Heegaard splitting of $L(\bar p, \bar q)$;

\item[(b)] the boundary of an annular spine of a Heegaard solid torus of $L(\bar p, \bar q)$.
\qed

\end{itemize}
\vspace{-.35cm}
\end{lemma}

Next we suppose that $W$ is not a lens space. In this case $2 \leq a \leq b \leq c$.

\begin{lemma} \label{+-quotient}
Suppose that $W$ is not a lens space and that $\tau$ preserves the orientations of the Seifert fibres of $W$. Then there is an induced Seifert structure on $W/\tau$ such that $L_\tau$ is a union of at most three Seifert fibres where at least one of the fibres is regular. Further, $\bar \tau$ is either the identity or has two fixed points and

$(1)$ if $\bar \tau$ is the identity then $a = 2$, $|L_\tau|$ is the number of cone points of $S^2(a,b,c)$ of even order, and the components of $L_\tau$ which are regular fibres correspond to the cone points of order $2$.

$(2)$ if $\bar \tau$ is not the identity then $L_\tau$ has at most two components. Exactly one of its components is a regular fibre.

\end{lemma}

\begin{proof} The hypotheses imply that there is an induced Seifert structure on $L(\bar p, \bar q)$ whose fibres are the images of the fibres of $W$. Since $W$ has three exceptional fibres, $\bar \tau$ fixes precisely one or three cone points. In the latter case, $\bar \tau$ is the identity.

Suppose first that $\bar \tau$ is the identity on $S^2(a,b,c)$. Since $\tau$ has a $1$-dimensional fixed point set, $\tau$ rotates the regular fibres of $W$ by $\pi$. Its fixed point set is the union of the fibres of even multiplicity and therefore $L_\tau$ is a union of Seifert fibres. The reader will verify that if a fibre of $W$ has multiplicity $k$, then its image in $L(\bar p, \bar q)$ has multiplicity $\bar k = \frac{k}{\gcd(k,2)}$. Hence as $L(\bar p, \bar q)$ has at most two exceptional fibres, $a = 2$.

Suppose next that $\bar \tau$ fixes precisely one cone point of $S^2(a,b,c)$. In this case its fixed point set consists of this cone point and a regular point. Thus the fixed point set of $\tau$ is contained in a union of two fibres, so $L_\tau$ has at most two components.  The reader will verify that each exceptional fibre of $W$ is sent to an exceptional fibre of $L(\bar p, \bar q)$, two of them to the same fibre. Thus the $\tau$-invariant regular fibre of $W$ is sent to a regular fibre of $L(\bar p, \bar q)$. It follows that this fibre lies in the fixed-point set of $\tau$ and therefore $L_\tau$ contains a regular fibre of $L(\bar p, \bar q)$.
\end{proof}

\begin{lemma} \label{--quotient}
Suppose that $W$ is not a lens space and that $\tau$ reverses the orientations of the Seifert fibres of $W$.
If $W/\tau = L(\bar p, \bar q) \not \cong S^3$, then

\noindent $(1)$ $W$ has base orbifold $S^2(\bar p, \bar p, m)$ where $m \geq 2$ and the Seifert
invariants of the exceptional fibres of order $\bar p$ are the same. Hence if
 $W$ is not a prism manifold, $\bar p \ne 2$.

\noindent $(2)$ There is an integer $n$ coprime with $m$ such that $L_\tau$ is isotopic
to the closure $K(m/n)$ of an $m/n$ rational tangle in a genus $1$ Heegaard
solid torus of $W / \tau$ as depicted in Figure \ref{bgz4-fig24}.
In particular,
$$|L_\tau| = \left\{ \begin{array}{ll}
1 & \hbox{ if } n \hbox{ is odd} \\
2 & \hbox{ if } n \hbox{ is even}
\end{array} \right. $$

\end{lemma}

\begin{figure}[!ht]
\centerline{\includegraphics{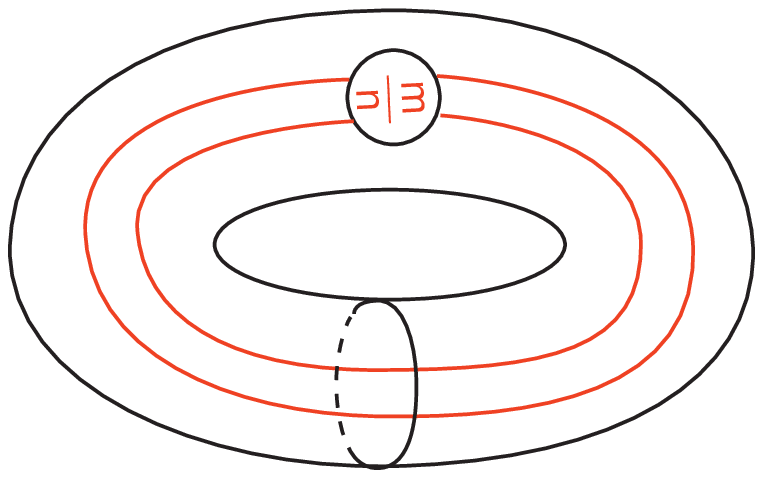}} \caption{}\label{bgz4-fig24}
\end{figure}

\begin{proof} The fixed point set of $\bar \tau$ is non-empty so as it reverses orientation, it is reflection in an equator  of $S^2(a,b,c)$. This equator cannot contain all three cone points as otherwise $\tau$ would be the Montesinos involution on $W$ and therefore $L(\bar p, \bar q)$ would be $S^3$. Thus it contains exactly one cone point and $\bar \tau$ permutes the other two. It follows that up to relabeling, $(a,b,c) = (r,r,m)$ for some integers $r, m \geq 2$. Further, $S^2(r,r,m)/ \bar \tau = D^2(r; m)$, where $D^2(r; m)$ is the $2$-orbifold with underlying space a $2$-disk and singular set consisting of a cone point of order $r$, a corner-reflector point $x$ of order $m$, and a reflection line $\partial D^2 \setminus \{x\}$. Therefore $L(\bar p, \bar q) = W / \tau \cong L(r, t)$ for some integer $t$. Thus $r = \bar p$, which proves part (1).

A Montesinos-type analysis of the quotient of the $\tau$-invariant solid torus given by the inverse image in $W$ of a small annular neighbourhood of $\hbox{Fix}(\bar \tau)$ in $S^2(\bar p, \bar p, m)$ shows that the branch set of this quotient is of the form described in part (2). It is well known that this branch set has one component if $n$ is odd and two otherwise, so part (2) holds.
\end{proof}

\section{Beginning of the proof of Theorem \ref{once-punctured}} \label{sec: once-punctured}

\subsection{Assumptions} \label{assumptions 1}

We assume throughout the rest of the paper that $M$ is a hyperbolic knot manifold containing an essential once-punctured torus $F$ of boundary slope $\beta$ which caps off to an essential torus in $M(\beta)$ (cf. Corollary \ref{stays incompressible}) and that $M(\alpha)$ is an atoroidal, irreducible, small Seifert manifold (cf. Corollary \ref{smallseifert}). We assume as well that $\Delta(\alpha, \beta) > 3$, and (therefore) $M_F$ is a genus $2$ handlebody by Proposition \ref{reduction}.

We will show that under these assumptions, $\Delta(\alpha, \beta) \leq 5$, $F$ is not a fibre, $\pi_1(M(\alpha))$ is finite non-cyclic, and

(a)  if $\Delta(\alpha, \beta) = 4$,  $(M;\alpha,\beta)\cong (Wh(\frac{-2n\pm1}{n}); -4,0)$
 for some integer $n$ with $|n|>1$ and $M(\alpha)$ has base orbifold $S^2(2,2,|\mp 2n-1|)$;

(b) if $\Delta(\alpha, \beta) = 5$, then $(M; \alpha, \beta) \cong (Wh(-3/2); -5, 0)$ and $M(\alpha)$ has base orbifold $S^2(2,3,3)$.

\begin{figure}[!ht]
\centerline{\includegraphics{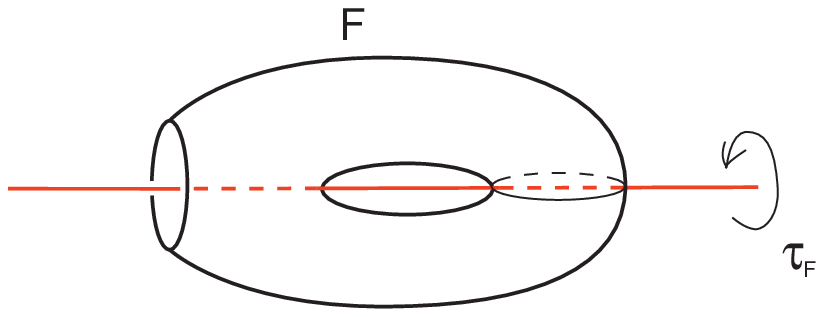}} \caption{ }\label{bgz4-fig0}
\end{figure}

\subsection{An involution on $M$} \label{involution 1}
\noindent There is an involution $\tau_F$ on $F$ with exactly three fixed points whose
action on $\partial F$ is rotation by $\pi$. See Figure
\ref{bgz4-fig0}. Thus $F / \tau_F$ is the $2$-orbifold $D^2(2,2,2)$. Let $N \cong F \times I$
be a small neighbourhood of $F$ in $M$ and extend $\tau_F$ to an involution
$\tau_N$ in the obvious way. Then $\tau_N| F \times \partial I$ extends to a
hyperelliptic involution of $\partial M_F$. Since $M_F$ is a genus $2$ handlebody, the latter extends to an
involution $\tau_{M_F}$ of $M_F$. Piecing together $\tau_N$ and $\tau_{M_F}$  we
obtain an orientation-preserving involution $\tau:M \to M$ with non-empty
$1$-dimensional fixed point set $\widetilde L \subset \hbox{int}(M)$. Further,
$V := M/\tau$ is a solid torus containing the branch set $L$ of $\tau$. By construction, this is a hyperbolic link which
intersects some meridional disk of $V$ transversely and in three points. When $F$ is a fibre in $M$, $L$ is braided in $V$.

Note that $L$
cannot intersect any meridional disk in one point as $M$ is
$\partial$-irreducible.

The slopes on $\partial M$ can be identified with $\pm$-classes of primitive elements of $H_1(\partial M)$.
In particular we assume $\alpha, \beta \in H_1(\partial M)$. Let $\mu$ be any dual slope to $\beta$. This means that $1 = \Delta(\mu, \beta) = |\mu \cdot \beta|$. Hence $\{\mu, \beta\}$ form a basis for $H_1(\partial M)$. Write
\begin{equation}\label{alpha 1}
\text{\em $\alpha = p \mu + q \beta$}
\end{equation}
where $p, q$ are coprime. After possibly changing the signs of $\mu$ and $\beta$ we may assume that
\begin{equation}\label{p 1}
\text{\em $p = \Delta(\alpha, \beta)$}
\end{equation}
Without loss of generality we may suppose that $p \geq 1$.
The map $M \to V$ is a double cover when restricted to $\partial M$. It sends $\beta$ to a slope $\bar \beta$, a meridian of $V$, and sends $\mu$ to $\bar \mu$, a longitude of $V$.

For each slope $\gamma$ on $\partial M$, $\tau$ extends to an involution $\tau_\gamma: M(\gamma) \to M(\gamma)$. Moreover, if $\widetilde U_\gamma$ denotes the filling torus in $M(\gamma)$ and $\widetilde K_\gamma$ its core, then
\begin{equation}\label{fix 1}
\text{\em $\hbox{Fix}(\tau_\gamma) = \left\{ \begin{array}{ll}
\widetilde L & \hbox{ if } \Delta(\gamma, \beta) \hbox{ is odd} \\
\widetilde L \cup \widetilde K_\gamma & \hbox{ if } \Delta(\gamma, \beta)  \hbox{ is even}
\end{array} \right.$}
\end{equation}

It is clear that $\widetilde U_\gamma/\tau_\gamma$ is a solid torus $U_\gamma$.
Denote its core $\widetilde K_\gamma / \tau_\gamma$ by $K_\gamma$. Thus
$M(\gamma) / \tau_\gamma = V \cup_{\bar \g} U_\gamma$ is a lens space. Indeed,
if $\gamma = r\mu + s \beta$, then under the double cover $\partial M \to
\partial V$ we have $\gamma \mapsto r \bar \mu + 2 s \bar \beta$. Let $\bar
\gamma = \frac{1}{\gcd(2,r)}(r \bar \mu + 2 s\bar \beta)$ denote the associated
slope and $L_\gamma$ the branch set in $M(\gamma) / \tau_\gamma$. Then
$$(M(\gamma) / \tau_\gamma, L_\gamma) = (V(\bar \gamma), L_\g) \cong \left\{ \begin{array}{ll}
(L(r, 2s), L) & \hbox{ if } r  \hbox{ is odd} \\
(L(\frac{r}{2}, s), L \cup K_\gamma)  & \hbox{ if } r  \hbox{ is even}
\end{array} \right.$$
We are interested in the case $\gamma = \alpha$. Set
\begin{equation}\label{pbar 1}
\text{\em $\bar p = p/ \gcd(p,2)$ \;\;\;\;\;  \hbox{ and } \;\;\;\;\;
$\bar q = 2q/\gcd(p,2)$}
\end{equation}
so that $\bar \alpha = \bar p \bar \mu + \bar q \bar
\beta$ and
$$M(\alpha) / \tau_\alpha \cong L(\bar p, \bar q)$$
From \ref{fix 1} we see that
\begin{equation}\label{comps 1}
\text{\em $|L_\alpha| = \left\{ \begin{array}{ll}
|L| & \hbox{ if } p  \hbox{ is odd} \\
|L| + 1 & \hbox{ if } p \hbox{ is even}
\end{array} \right. $}
\end{equation}

Fix a $\tau_\alpha$-invariant Seifert structure on $M(\alpha)$ with base orbifold $S^2(a,b,c)$ where $1 \leq a \leq b \leq c$ (Lemma \ref{invariant}).

Let $\bar \tau_\alpha$ be the involution of $S^2(a,b,c)$ (possibly the identity) induced by $\tau_\alpha$.

\begin{lemma} \label{+-quotient 2}
Suppose that assumptions \ref{assumptions 1} hold. Suppose as well that $M(\alpha)$ is not a lens space and that $\tau_\alpha$ preserves the orientations of the Seifert fibres of $M(\alpha)$. Then there is a Seifert structure on $L(\bar p, \bar q)$ in which $L_\alpha$ is a union of at most three fibres, at least one of which is regular. Further, $L_\alpha = L$ so that $p = \Delta(\alpha, \beta)$ is odd.
\end{lemma}

\begin{proof} Lemma \ref{+-quotient} shows that $L$ is a union of fibres in the induced Seifert structure on $L(\bar p, \bar q)$ and at least one of these fibres is regular. This implies that $K_\alpha \not \subset L_\alpha$ as otherwise $L = L_\alpha \setminus K_\alpha$ would not be a hyperbolic link in $V$. Thus $L = L_\alpha$ so $p$ is odd by \ref{comps 1}.
\end{proof}

\begin{lemma} \label{--quotient 2}
Suppose that assumptions \ref{assumptions 1} hold. Suppose as well that $M(\alpha)$ is not a lens space and that $\tau_\alpha$ reverses the orientations of the Seifert fibres of $M(\alpha)$. Then

\noindent $(1)$ $M(\alpha)$ has base orbifold $S^2(\bar p, \bar p, m)$ where $m \geq 2$ and the Seifert
invariants of the exceptional fibres of order $\bar p$ are the same. Hence if
 $M(\alpha)$ is not a prism manifold, $\Delta(\alpha, \beta) \ne 4$.

\noindent $(2)$ There is an integer $n$ coprime with $m$ such that $L_\alpha$ is isotopic
to the closure $K(m/n)$ of an $m/n$ rational tangle in a genus $1$ Heegaard
solid torus of $M(\alpha) / \tau_\alpha$ as depicted in Figure \ref{bgz4-fig24}.
In particular,
$$|L_\alpha| = \left\{ \begin{array}{ll}
1 & \hbox{ if } n \hbox{ is odd} \\
2 & \hbox{ if } n \hbox{ is even}
\end{array} \right. $$

\noindent $(3)$ $|L| = 1$, $m$ is odd, and $n \equiv p$ (mod $2$).
\end{lemma}

\begin{proof} Parts (1) and (2) follow from Lemma \ref{--quotient}.

In order to prove part (3), suppose that $|L| = 2$. Then part (2) shows that $L = L_\alpha$. In particular, $p$ is odd (\ref{comps 1}). Consideration of the form of $L_\alpha$ (cf. Figure \ref{bgz4-fig24}) shows that its two components are isotopic to one another. But since $L$ is transverse to a meridian disk of $V$ and intersects it in three points, the generator $\gamma$ of $H_1(V(\bar \alpha)) \cong \mathbb Z / \bar p$ carried by the core of $V$ satisfies $\gamma = \pm 2 \gamma$. Hence $\bar p = 3$. But $p$ is odd so $\Delta(\alpha, \beta) = p = \bar p = 3$, contrary to our hypotheses. Thus $|L| = 1$.

Next suppose that $m$ is even. Then $L_\alpha = K(m/n)$  is connected, so $L = L_\alpha$ and $p$ is odd, and $L$ is homotopically trivial in $L(\bar p, \bar q)$. But $L$ intersects a meridian disk of the Heegaard torus $V \subset L(\bar p, \bar q)$ transversely and in three points, so the only way it can be null homotopic is for $3 = \bar p$. Since $p$ is odd, $p = 3$, which contradicts our hypotheses. Thus $m$ is odd.

By (2), $|L_\alpha| \equiv n$ (mod $2$). Since $|L| = 1$ by (3), Identity \ref{comps 1} shows that $|L_\alpha| \equiv p$ (mod $2$).
\end{proof}

\subsection{Constraints on the branch set $L$}

Here we deduce strong constraints on the form of the branch set $L$ in $V$.

\begin{lemma} \label{cyclic-cover}
Suppose that assumptions \ref{assumptions 1} hold and that $\tau_\alpha$ reverses the orientation of the Seifert fibres of $M(\alpha)$. Let $k \geq 1$ be an integer dividing $\bar p$ and consider the $k$-fold cyclic cover $S^2(\frac{\bar p}{k}, \frac{\bar p}{k}, m, m, \ldots , m) \to S^2(\bar p, \bar p, m)$
obtained by the $k$-fold unwrapping of $S^2(\bar p, \bar p, m)$ about the two cone points labeled $\bar p$. Let $\widetilde {M(\alpha)}_k \to M(\alpha)$ be the associated $k$-fold cyclic cover where $\widetilde {M(\alpha)}_k$ is Seifert with base orbifold $S^2(\frac{\bar p}{k}, \frac{\bar p}{k}, m, m, \ldots , m)$ and the inclusion of a regular fibre of $M(\alpha)$ lifts to $\widetilde {M(\alpha)}_k$. Define $\widetilde{M}_k \to M$ to be the cover obtained by restricting $\widetilde {M(\alpha)}_k \to M(\alpha)$ to $M$. Then

\noindent $(1)$ $\partial \widetilde M_k$ is connected and $F$ lifts to $\widetilde M_k$. In particular, $\beta$ lifts to a slope $\widetilde \beta$ on $\partial \widetilde M_k$.

\noindent $(2)$ $\alpha$ lifts to a slope $\widetilde \alpha$ on $\partial \widetilde M_k$ such that $\widetilde {M(\alpha)}_k = \widetilde M_k(\widetilde \alpha)$. Further, $\Delta(\widetilde \alpha, \widetilde \beta) = \frac{p}{k}$.

\noindent $(3)$ $\widetilde \alpha$ is the singular slope of a closed essential surface in $\widetilde M_k$ if $S^2(\frac{\bar p}{k}, \frac{\bar p}{k}, m, m, \ldots , m)$ is hyperbolic with at least four cone points. If this is the case, $p/k \leq 3$.
\end{lemma}

\begin{proof} The cover $S^2(\frac{\bar p}{k}, \frac{\bar p}{k}, m, m, \ldots , m) \to S^2(\bar p, \bar p, m)$ is determined by the homomorphism $\varphi: H_1(S^2(\bar p , \bar p, m)) = \langle x,y : \bar p x = \bar p y = m(x+y) = 0 \rangle \to \mathbb Z/k$ where $\varphi(x) \equiv - \varphi(y) \equiv 1 \hbox{ (mod $k$)}$.

First note that the homomorphism $H_1(M(\alpha)) \to H_1(V(\bar \alpha)) \cong \mathbb Z / \bar p$ kills any class carried by a regular Seifert fibre of $M(\alpha)$ (i.e. there are regular fibres with image an interval). Thus it factors through a homomorphism $\psi: H_1(S^2(\bar p, \bar p, m)) \to H_1(V(\bar \alpha))$. Since $\tau_\alpha$ preserves the fibre of multiplicity $m$ in $M(\alpha)$, but reverses its orientation, $(\bar \tau_\alpha)_* (x + y) = -(x + y)$. Thus $2(x+y)$ is sent to zero in $H_1(V(\bar \alpha))$ while $x$ is sent to a generator. Since $m$ is odd and  $m(x+y) = 0$, $x+y \mapsto 0 \in H_1(V(\bar \alpha))$. It follows that $\varphi$ factors as $H_1(S^2(\bar p, \bar p, m)) \stackrel{\psi}{\longrightarrow} H_1(V(\bar \alpha))  \stackrel{\cong}{\longrightarrow} \mathbb Z / \bar p \to \mathbb Z / k$. Since $H_1(F)$ lies in the kernel of $H_1(M) \to H_1(V)$ while $\mu$ is sent to a generator of $H_1(V)$, we conclude that $\partial \widetilde M_k$ is connected and $F$ lifts to $\widetilde M_k$. This proves (1).

For (2), note that by construction, there is a basis $\{\widetilde \mu, \widetilde \beta\}$ of $H_1(\partial \widetilde M_k)$ where $\widetilde \mu$ is sent to $k \mu$ and $\widetilde \beta$ is sent to $\beta$ in $H_1(\partial M)$. Then $\alpha = p \mu + q \beta$ lifts to $(\frac{p}{k}) \widetilde \mu + q \widetilde \beta$. Clearly $\Delta(\widetilde \alpha, \widetilde \beta) = \frac{p}{k}$.

Part (3) is a consequence of \cite[Theorems 1.5 and 1.7]{BGZ1}.
\end{proof}

\begin{lemma} \label{holds}
Suppose that assumptions \ref{assumptions 1} hold. Then $M$ is not a once-punctured torus bundle. In particular, Theorem \ref{once-punctured} holds when $F$ is a fibre.
\end{lemma}

\pf We assume that $M$ is a once-punctured torus bundle in order to  obtain a contradiction.

There is a $3$-braid $\sigma$ whose closure in $V$ is $L$. Altering $\sigma$ by conjugation in
$B_3 = \langle \sigma_1, \sigma_2 : \sigma_1 \sigma_2 \sigma_1 = \sigma_2 \sigma_1 \sigma_2 \rangle$
leaves its closure invariant. (Here $\sigma_1, \sigma_2$ are the standard generators of $B_3$.) There is an isomorphism $B_3 \cong \langle a, b : a^3 = b^2 \rangle$ where $a = \sigma_1 \sigma_2$ and $b = \sigma_1 \sigma_2 \sigma_1$. The center  of $B_3$ is generated by $a^3$ with $B_3 / \langle a^3 \rangle \cong \mathbb Z/2 * \mathbb Z / 3$. We will use $\bar \sigma$ to denote the image of a braid $\sigma$ in $B_3 / \langle a^3 \rangle$. Thus $\bar a$ has order $3$ and $\bar b$ has order $2$. In particular,
$$\bar \sigma_1 = \bar a^{-1} \bar b$$
\vspace{-.7cm}
$$\bar \sigma_2 = \bar b \bar a^2$$

The inverse image $\widehat L$ of $L \subset V \subset L(\bar p, \bar q)$ under the universal cover $S^3 \to L(\bar p, \bar q)$ is the closure the braid $\sigma^{\bar p} a^{-3 \bar q}$.

\begin{claim} \label{nontrivial}
{\it $\widehat L$ is not the trivial knot.}
\end{claim}

\begin{proof}[Proof of Claim \ref{nontrivial}]
If $\widehat L$ is trivial then $\sigma^{\bar p} a^{-3 \bar q}$ is conjugate to $\sigma_1 \sigma_2, \sigma_1^{-1} \sigma_2^{-1}$, or $\sigma_1 \sigma_2^{-1}$ (\cite[Classification Theorem, page 27]{BiMe}). The first two cases can be ruled out since they would imply that the exterior of $\widehat L$ in the inverse image of $V$ in $S^3$ is not hyperbolic. On the other hand, in the third case we have $\bar \sigma^{\bar p} = \bar \sigma_1 \bar \sigma_2^{-1} = \bar a^2 \bar b \bar a \bar b \in B_3 / \langle a^3 \rangle \cong \mathbb Z/2 * \mathbb Z / 3$. But this is impossible since $\bar a^2 \bar b \bar a \bar b$ is not a proper power.
\end{proof}

\begin{claim} \label{+}
{\it $\tau_\alpha$ preserves the orientation of the Seifert fibres of $M(\alpha)$. In particular, $\widehat L$ is a union of fibres in some Seifert structure on $S^3$ and $p$ is odd.}
\end{claim}

\begin{proof}[Proof of Claim \ref{+}] Suppose otherwise and consider the $\bar p$-fold cyclic cover $\widetilde M_{\bar p} \to M$ constructed in Lemma \ref{cyclic-cover}. The base orbifold $S^2(m, m, \ldots , m)$ of $\widetilde{M(\alpha)}$ has $\bar p$ cone points, each of order $m \geq 3$ by Lemma \ref{--quotient 2}(3). If $\bar p \geq 4$, Lemma \ref{cyclic-cover} (3) implies that $\widetilde M_{\bar p}$ contains a closed essential surface, contrary to \cite{CJR} or \cite{FH}. Hence $\bar p$ is $2$ or $3$ and therefore as $p > 3$, $p$ is $4$ or $6$. Identity \ref{comps 1} then combines with parts (2) and (3) of Lemma \ref{--quotient 2} to show that $|L_\alpha| = 2$ and $m$ is odd. It follows that each component of $L_\alpha$ is isotopic to the core of a genus one Heegaard solid torus in $L(\bar p, \bar q)$ (cf. Figure \ref{bgz4-fig24}). In particular this is true of $L = L_\alpha \setminus K_\alpha$. It follows that $\widehat L$ is a trivial knot, contrary to the conclusion of Claim \ref{nontrivial}. Thus $\tau_\alpha$ preserves the orientation of the Seifert fibres of $M(\alpha)$. The remaining conclusions are a consequence of Lemma \ref{+-quotient 2}.
\end{proof}

Claim \ref{+} implies that $\bar p = p$ and $\bar q = 2q$.

Since $L$ is a hyperbolic link in $V$, $\widehat L$ is a hyperbolic link in the inverse image of $V$ in $S^3$. Thus the Schreier normal form for  $\sigma^{p} a^{-6 q}$ is generic (cf. \cite[Theorem 5.2]{FKP}). On the other hand, by Claim \ref{+}, $\widehat L$ is not a hyperbolic link in $S^3$ so \cite[Theorem 5.5]{FKP} implies that $\sigma^{p} a^{-6 q}$ is conjugate in $B_3$ to a braid of the form $\sigma_1^c \sigma_2^d$ where $c, d \in \mathbb Z \setminus \{0\}$. We must have $\min\{|c|, |d|\} = 1$ as otherwise $\widehat L$ would be a connected sum of non-trivial torus links, contrary to the conclusion of Claim \ref{+}. Thus $\sigma^{p} a^{-6 q}$ is conjugate to $\sigma_1^c \sigma_2^\epsilon$ for some $\epsilon \in \{\pm 1\}$ and non-zero $c$. The following claim completes the proof of Lemma \ref{holds}.

\begin{claim} \label{2}
{\it If $p > 3$, $\sigma^{p} a^{-6 q}$ is not conjugate to $\sigma_1^c \sigma_2^\epsilon$ for any $\epsilon \in \{\pm 1\}$. }
\end{claim}

\begin{proof}[Proof of Claim \ref{2}]
Suppose that $\sigma^{p} a^{-6  q}$ is conjugate to $\sigma_1^c \sigma_2^\epsilon$ for some $\epsilon \in \{\pm 1\}$. Projecting into $B_3 / \langle a^3 \rangle$ shows that $\bar \sigma_1^c \bar \sigma_2^\epsilon$ is a $p^{th}$-power in that group. The latter condition is invariant under conjugation and taking inverse, so without loss of generality we can suppose that $\epsilon = 1$. Now
$$\bar \sigma_1^c \bar \sigma_2 = (\bar a^{-1} \bar b)^c (\bar b \bar a^{-1}) = \left\{
\begin{array}{ll}
(\bar b \bar a)^{|c|} (\bar b \bar a^{-1}) & \hbox{if } c \leq 0 \\
\bar a & \hbox{if } c = 1 \\
\bar a^{-1} \bar b \bar a  & \hbox{if } c = 2 \\
(\bar a^{-1} \bar b) \bar a^{-1} (\bar a^{-1} \bar b)^{-1}   & \hbox{if } c = 3 \\
 (\bar a^{-1} \bar b \bar a)(\bar a \bar b) (\bar a^{-1} \bar b)^{c-4} (\bar a^{-1} \bar b \bar a)^{-1}  & \hbox{if } c > 3
  \end{array} \right. $$
Consideration of the normal form for elements of $\mathbb Z/2 * \mathbb Z / 3$ shows that the only values of $c$ which give proper powers in $B_3 / \langle a^3 \rangle$ are $c = 1, 2$, or $3$.

Say $c = 1$ or $3$. Then up to conjugation, $\bar \sigma^p = \bar a^{\pm 1}$ and therefore $\bar \sigma = \bar a ^{\pm 1}$. Hence $\sigma = a^{3k \pm 1}$ for some integer $k$. But then it is easy to see that $L$ is boundary-parallel in $V$, contrary to the fact that $V \setminus L$ is hyperbolic.

Next suppose that $c = 2$. Then $\bar \sigma^p = \bar b$ up to conjugation and therefore the same is true of $\bar \sigma$. As $a^3 = b^2$, $\sigma = b^{2n + 1}$ for some integer $n$. Then $L \subset \hbox{int}(V)$ has two components. One is a core curve $K_0$ of $V$ while the other is isotopic in $V \setminus K_0$ into $\partial V$. It follows that there is an essential annulus properly embedded in the exterior of $L$ in $\hbox{int}(V)$. But this contradicts the fact that $L$ is a hyperbolic link in $V$.
\end{proof}
\vspace{-.5cm} \hfill \qed (of Lemma \ref{holds})

Recall that $t_1^+$ is the number of tight components of $\breve{\Phi}_1^+$ (cf. \cite[\S 6]{BGZ2}).

\begin{lemma} \label{secondreduction}
Suppose that assumptions \ref{assumptions 1} hold. Then $t_1^+ = 0$. In particular, $M(\beta)_{\widehat F}$ is Seifert with base
orbifold of the form $A(a)$ where $A$ is an annulus and $a \geq 2$.
\end{lemma}

\begin{proof} Lemma \ref{holds} implies that $F$ is not a fibre and so Proposition \ref{main1} and Remark \ref{t1+ > 0} show that the lemma holds as long as either $M(\alpha)$ is very small or $\Delta(\alpha, \beta) > 4$. Assume then that $M(\alpha)$ is not very small and that $\Delta(\alpha, \beta) = 4$. The latter equality combines with Lemma \ref{+-quotient 2} to show that $\tau_\alpha$ reverses the orientations of the fibres of $M(\alpha)$. But then Lemma \ref{--quotient 2}(1) implies that $M(\alpha)$ is a prism manifold, contradicting our assumption that $M(\alpha)$ is not very small. Thus the lemma holds.
\end{proof}

\begin{lemma} \label{types}
Suppose that assumptions \ref{assumptions 1} hold. Then there are coprime integers
$a \geq 2$ and $b$ as well as a $3$-braid $\sigma$ such that $L$ is isotopic
to the link depicted in Figure \ref{fig3}.
\end{lemma}

\begin{figure}[!ht]
\centerline{\includegraphics{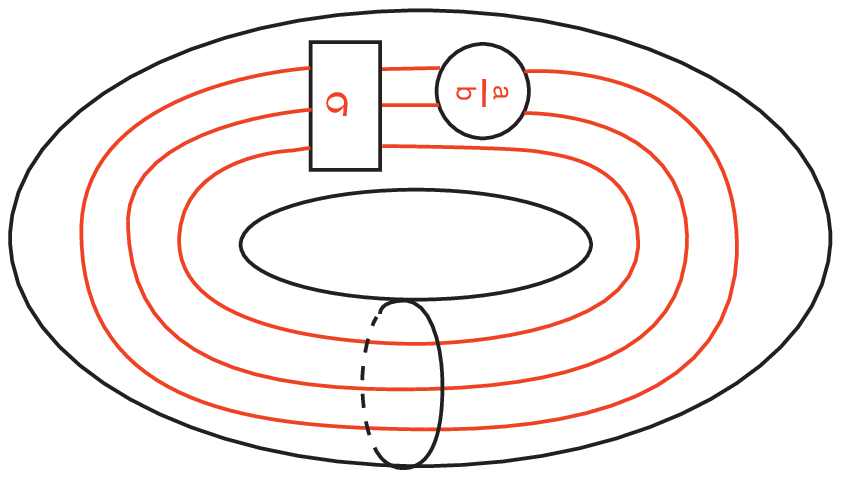}} \caption{ }\label{fig3}
\end{figure}

\begin{proof} By Lemma \ref{secondreduction}, $M(\beta)_{\widehat F}$ is Seifert with base
orbifold of the form $A(a)$ where $A$ is an annulus and $a \geq 2$.
Consider the involution $\widehat \tau: M(\beta)_{\widehat F} \to
M(\beta)_{\widehat F}$ induced by $\tau_\beta$.
Note that $M(\beta)_{\widehat F}/ \widehat \tau = V(\bar \beta)_{\widehat F/
\widehat\tau} \cong (S^2 \times S^1)_{S^2 \times \{x\}} \cong S^2 \times I$. Now
$M(\beta)_{\widehat F}$ has a unique Seifert structure which we can suppose is
$\widehat \tau$-invariant. Let $\overline{\widehat \tau}$ be the induced
involution on $A(a)$. Note $\overline{\widehat \tau}$ cannot preserve
orientation as otherwise $M(\beta)_{\widehat F}/ \widehat \tau \cong S^2 \times
I$ would admit a Seifert structure. Thus it reverses orientation and since it
fixes the cone point and leaves each boundary component invariant, it must be
reflection along a pair of disjoint properly embedded arcs, each of which runs
from one boundary component to the other. The quotient $A(a) /
\overline{\widehat \tau}$ is a disk whose boundary contains two disjoint,
compact arcs, each a reflector arc, one of which contains the $\mathbb Z/ a$ cone point.
It follows that the branch set in $M(\beta)_{\widehat F}/ \widehat \tau \cong
S^2 \times I$ consists of a $2$-braid and an $\frac{a}{b}$-rational tangle running from
one end to the other which are separated by a properly embedded vertical annulus. See
Figure \ref{fig1}.

\begin{figure}[!ht]
\centerline{\includegraphics{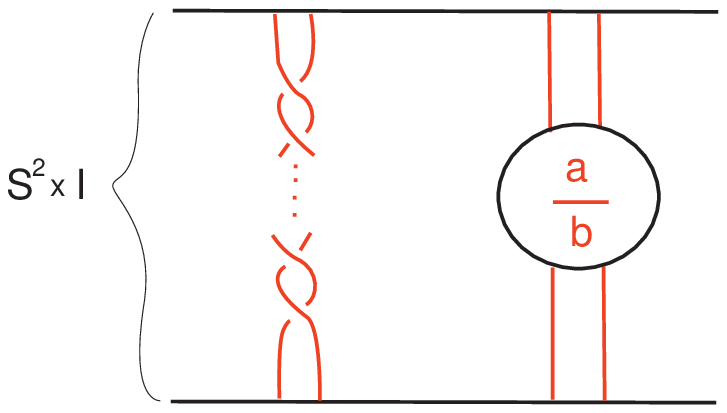}} \caption{ }\label{fig1}
\end{figure}

We claim that $K_\beta \cap M(\beta)_{\widehat F}$ is a component of the 2-braid. To see this, first note that by Lemma \ref{secondreduction}, $\breve{\Phi}_1^+$ has no tight components. Next we refer the reader to the final paragraph of the proof of \cite[Lemma 7.9]{BGZ2}. It is shown there that $M_{F} = X^+$ is obtained by attaching a solid torus $V$ to the product of an interval $I$ and a once-punctured annulus $A_*$ where $V \cap (A_* \times I)$  is a pair of annuli which have winding number $a$ in $V$ and components of $\partial A_* \times I$ in  $A_* \times I$. This decomposition is invariant under the restriction of $\hat \tau$ to $M_F$ and it is easy to see that the quotient of $V$ contains the $\frac{a}{b}$-rational tangle. Since $(\partial M)_{\partial F} \subset A_* \times I$ is disjoint from $V$, it follows that $K_\beta \cap M(\beta)_{\widehat F}$ is a component of the 2-braid. Thus $L \cap M_F/ \tau$ is
as depicted in Figure \ref{fig2}, where $\d$ is a $3$-braid. It follows that there is a $3$-braid $\sigma$ such that $L$ is as depicted
in Figure \ref{fig3}.
\end{proof}

\begin{figure}[!ht]
\centerline{\includegraphics{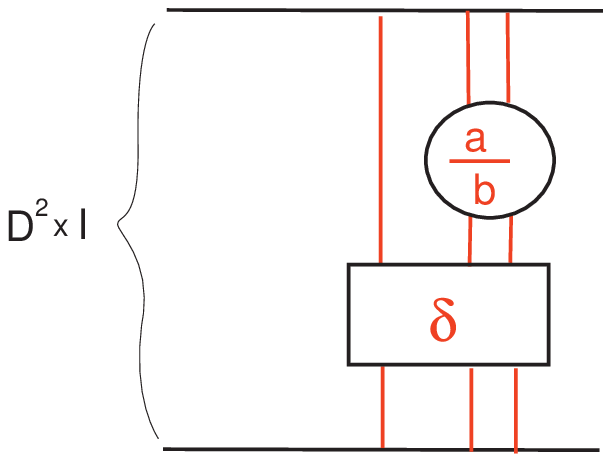}} \caption{ }\label{fig2}
\end{figure}

\subsection{The lens space case} \label{lens space case}

The methods of this paper can be used to give a new proof of Ken Baker's theorem: {\it if $M$ contains a once-punctured essential genus $1$ surface of boundary slope $\beta$ and $M(\alpha)$ is a lens space, then $\Delta(\alpha, \beta) \leq 3$} \cite{Ba}. We begin the proof here and complete it in \S \ref{lens space delta = 5}.

\begin{lemma} \label{once-punctured cyclic case}
Suppose that assumptions \ref{assumptions 1} hold. If $\pi_1(M(\alpha))$ is cyclic, then $p = 5$, $F$ is not a fibre, and $L_\alpha$ is either the core of a solid torus of a genus one Heegaard splitting of $L(5, 2q)$ or the boundary of a M\"{o}bius band spine of a Heegaard solid torus of $L(5, 2q)$.
\end{lemma}

\begin{proof} We know that $F$ is not a fibre (Lemma \ref{holds}), so $p = \Delta(\alpha, \beta) \leq 6$ by Proposition \ref{main1}.
As $\Delta(\alpha, \beta) = p \geq 4$, $M(\alpha) / \tau_\alpha \cong L(\bar p, \bar q)$ is not $S^3$.
Hence by Lemma \ref{cyclic case}, $L_\alpha$ is a union of Seifert fibres of some Seifert fibring of $L(\bar p, \bar q)$. Since $L$ is hyperbolic in $V$, $K_\alpha$ cannot be contained in $L_\alpha$. Thus $p$ is odd by \ref{fix 1}, so $p = \bar p = 5$, $\bar q = 2q$, and $L = L_\alpha$.
Lemma \ref{cyclic case}(1) then shows that $L_\alpha$ is either the core of a solid torus of a genus one Heegaard splitting of $L(5, 2q)$ or the boundary of a M\"{o}bius band spine of a Heegaard solid torus of $L(5, 2q)$.
\end{proof}

\begin{rem}
{\rm We can complete the proof of Baker's result mentioned above at this point by invoking a theorem of Sangyop Lee \cite{L3} which states that the distance between a toroidal filling slope and a lens space filling slope is at most $4$. Nevertheless, we give an independent proof that $\Delta(\alpha, \beta) \ne 5$ (and so $\Delta(\alpha, \beta) \leq 3$) in \S \ref{lens space delta = 5} below.}
\end{rem}

\subsection{Reduction of the proof of Theorem \ref{once-punctured}}

In this section we reduce the proof of Theorem \ref{once-punctured} to several problems concerning links. These will be solved in the subsequent sections of the paper. We begin with a slight sharpening of our upper bound for $\Delta(\alpha, \beta)$.

\begin{lemma} \label{at most 7}
If assumptions \ref{assumptions 1} hold, then $\Delta(\alpha, \beta) < 8$.
\end{lemma}

\begin{proof}  By Lemma \ref{holds}, $F$ is not a fibre in $M$. Hence $\Delta(\alpha, \beta) \leq 8$ by Proposition \ref{main1} (or \cite{LM}). Suppose that $\Delta(\alpha, \beta) = 8$. Then $M(\alpha)$ is not very small by Proposition \ref{main1}. Further, Proposition \ref{reduction} implies that $M_F$ is a genus two handlebody, so we can construct an involution $\tau$ as above. Then Lemma \ref{+-quotient 2} implies that $\tau_\alpha$ reverses the orientations of the Seifert fibres of $M(\alpha)$. Parts (1) and (3) of Lemma \ref{--quotient 2} imply that $M(\alpha)$ has a Seifert structure with base orbifold $S^2(4,4, m)$ where $m \geq 3$ is odd. Let $\widetilde M_2 \to M$ be the $2$-fold cover constructed in Lemma \ref{cyclic-cover}. By part (2) of that lemma, $\widetilde M_2 (\widetilde \alpha)$ is Seifert with base orbifold $S^2(4, 4, m, m)$. But then Lemma \ref{cyclic-cover} (3) implies $4 = \frac{8}{2} \leq 3$, which is false. Thus $\Delta(\alpha, \beta) \ne 8$.
\end{proof}

\begin{lemma} \label{delta = 4 case}
Suppose that assumptions \ref{assumptions 1} hold and that $\Delta(\alpha, \beta) = 4$. Then $M(\alpha)$ is a prism manifold.
\end{lemma}

\begin{proof} Since $\Delta(\alpha, \beta)$ is even, $M(\alpha)$ is not a lens space (Lemma \ref{once-punctured cyclic case}) and so Lemma \ref{+-quotient 2} implies that $\tau_\alpha$  reverses the orientations of the fibres of $M(\alpha)$. Lemma \ref{--quotient 2}(1) now implies that $M(\alpha)$ is a prism manifold.
\end{proof}

Given the last two lemmas, to complete the proof of Theorem \ref{once-punctured} under  assumptions \ref{assumptions 1}, we must consider the possibility that $\Delta(\alpha, \beta) \in \{5,6,7\}$ besides the case when $\D(\alpha,\beta)=4$
and $M(\a)$ is a prism manifold. We do this by comparing the constraints  obtained above on the branch sets $L$ and $L_\alpha$:
\begin{itemize}

\item $L$ lies in $V$ as depicted in Figure \ref{fig3} (Lemma \ref{types});

\vspace{.3cm} \item when $M(\alpha)$ is not a lens space and $\tau_\alpha$ preserves the orientation of the Seifert fibres of $M(\alpha)$, then $\Delta(\alpha, \beta)$ is odd and $L_\alpha$ is the union of at most three fibres of some Seifert structure on $L(\bar p, \bar q)$ (Lemma \ref{+-quotient 2});

\vspace{.3cm} \item when $M(\alpha)$ is not a lens space and $\tau_\alpha$ reverses the orientation of the Seifert fibres of $M(\alpha)$, then  $L_\alpha$ lies in some Heegaard solid torus of $L(\bar p, \bar q)$ as depicted in Figure \ref{bgz4-fig24} (Lemma \ref{--quotient});

\vspace{.3cm} \item when $M(\alpha)$ is a lens space, then $\Delta(\alpha, \beta) = 5$ and $L_\alpha$ is either the core of a Heegaard solid torus of $L(5, 2q)$ or the boundary of a M\"{o}bius band spine of a Heegaard solid torus of $L(5, 2q)$
 (Lemma \ref{once-punctured cyclic case}).

\end{itemize}
The proof of Theorem \ref{once-punctured} therefore reduces to proving the following claims.
\begin{enumerate}

\item  If $\tau_\alpha$ preserves the orientation of the Seifert fibres and $M(\alpha)$ is not a lens space, then $\Delta(\alpha, \beta) = 5$ and $(M; \alpha, \beta)$ is homeomorphic to $(Wh(-3/2); -5, 0)$.

\vspace{.3cm} \item The links contained in the universal cover $S^3$ of $L(7, \bar q)$ which are depicted in Figure \ref{7ab} and Figure \ref{7mn} are not equivalent when $\Delta(\alpha, \beta) = 7$, $|L| = 1$, $m$ is odd, and $n \equiv 1$ (mod 2).

\vspace{.3cm} \item the link depicted in Figure \ref{fig3} considered as lying in a Heegaard solid torus in $L(5, 2q)$ is not isotopic to either the core of a Heegaard solid torus or the boundary of a M\"{o}bius band spine of a Heegaard solid torus.

\vspace{.3cm} \item The links contained in a Heegaard solid torus in $L(3, \bar q)$ depicted in Figure \ref{bgz4-fig24} and Figure \ref{fig3} are not equivalent.

\vspace{.3cm} \item The links contained in the universal cover $S^3$ of $L(5, \bar q)$ which are depicted in Figure \ref{bgz4-5ab} and Figure  \ref{bgz4-5mn} are not equivalent in the universal cover $S^3$ of $L(5, \bar q)$ when $\Delta(\alpha, \beta) = 5$, $|L| = 1$, $m$ is odd, and $n \equiv 1$ (mod 2).

\vspace{.3cm} \item  $\D(\alpha,\beta)=4$ and $M(\a)$ is a prism manifold if and only if
$(M; \alpha,\beta)\cong (Wh(\frac{-2n\pm1}{n});-4,0)$ for some integer $n$ with $|n|>1$.

\end{enumerate}
These will be proved in \S \ref{m=1 Seifert}, \S \ref{dist7}, \S \ref{lens space delta = 5}, \S \ref{sec6.3},  \S \ref{delta = 5} and \S \ref{prism-section} respectively.

\section{The case that $\tau_\alpha$ preserves the orientation of the Seifert fibres, $M(\alpha)$ is not a lens space, and $\Delta(\alpha, \beta) \in \{5, 7\}$.} \label{m=1 Seifert}

In this section we suppose that assumptions \ref{assumptions 1} hold and show that if $\tau_\alpha$ preserves the orientation of the Seifert fibres, $M(\alpha)$ is not a lens space, and $\Delta(\alpha, \beta) \in \{5, 7\}$, then $\Delta(\alpha, \beta) = 5$ and $(M; \alpha, \beta)$ is homeomorphic to $(Wh(-3/2); -5, 0)$.

By hypothesis, $M(\alpha)$ is small Seifert with exactly three singular fibres. It is not a prism manifold
by \cite{L2} and so has a unique Seifert structure. Recall that $M(\alpha)/\tau_\alpha = V(\bar \alpha)$ is the lens space
 $L(\bar p, \bar q) = L(p, 2q)$ and the branch set of $\tau_\alpha$ in $L(p, 2q)$ is a link denoted by  $L_\alpha$. As $p$ is odd,  $L_\alpha=L$ (cf. \ref{comps 1}).

Suppose that $L_\alpha$ is a Seifert link with respect to the  induced Seifert
fibration on $L(p, 2q) = M(\alpha)/\tau_\alpha$. We need to show that
$p=5$ and $(M; \alpha, \beta)$ is homeomorphic to $(Wh(-3/2); -5, 0)$.

By Lemma \ref{+-quotient 2}, at least one component of $L$  is a regular fibre of $L(p, 2q)$.
Let $K$ be such a component and denote by $X$ the exterior of $L$ in $L(p, 2q)$.
Then $X$ has the induced Seifert fibration with $|\partial X|=|L|$ boundary components, each a torus.
Let $T_K$ be the component of $\partial X$ corresponding to the knot  $K$.

\begin{lemma}\label{unique annulus for K}
There is an essential separating  vertical annulus $(A,\partial A)\subset (X, T_K)$
which cuts $X$ into two components $X_1$ and $X_2$ such that  each $X_i$  is
either a torus cross interval or a fibred solid torus
whose core is a singular fibre of $X$ of order larger than $2$.
\end{lemma}

\begin{proof} The lemma follows from Lemma \ref{+-quotient} and its  proof.
Let $\bar \tau_\alpha$ be the induced map on the orbifold $S^2(a,b,c)$ of $M(\alpha)$
where each of $a, b, c$ is  $\geq 2$.
Then $\bar \tau_\alpha$ is either the identity or an involution with two
fixed points. Let $\sigma_1,\sigma_2,\sigma_3$ denote the singular fibres of $M(\alpha)$ and let their orders be
$a,b,c$ respectively.

First assume that $\bar \tau_\alpha$ is the identity map. Then Lemma \ref{+-quotient}(a) implies that at least one of $a,b,c$, say $a$, is $2$ and the fixed point set of $\tau_\alpha$ in $M(\alpha)$ is the union of those $\sigma_i$ with even orders.
In particular $\sigma_1$ belongs to the fixed point set of $\tau_\alpha$
and its image in $L(p, 2q)$  is a regular fibre.
Note that if $\sigma_2$, respectively $\sigma_3$, does not belong to the fixed point
set of $\tau_\alpha$, then $b$, respectively $c$, is odd, and the image of $\sigma_2$,
respectively $\sigma_3$, in $L(p, 2q)$
is a fibre of $L(p, 2q)$ of order  $b$, respectively $c$.
Hence the sum of $|\partial X|=|L|$ and the number of the singular fibres of $X$
equals $3$. Since the surface underlying the base orbifold of $X$ is planar, the lemma follows in this case.

Next  assume that $\bar \tau_\alpha$ is an involution.
Then two of the singular fibres of $M(\alpha)$, say $\sigma_1$ and $\sigma_2$,  have the same order $a = b$. Both are mapped to a common  singular fibre in $L(p, 2q)$  of  order $a$.
Since $M(\alpha)$ is not a prism manifold, $a = b > 2$.

By Lemma \ref{+-quotient}(b), the fixed point set of $\tau_\alpha$ in $M(\alpha)$ consists of a regular fibre
and possibly the remaining singular fibre $\sigma_3$.
If $\sigma_3$ does not belong to $\hbox{Fix}(\tau_\alpha)$, then its image in $L(p, 2q)$  is a singular fibre  of order
 $2c \geq 4$ and therefore the sum $|\partial X|=|L|$ and the number of the singular fibres of $X$
again equals $3$. As in the previous case, the lemma follows from this.
\end{proof}

Recall that $K_\alpha$ is the core circle of the filling solid torus in $V(\bar \alpha) = L(p, 2q)$.
The exterior $Y$ of $K_\alpha$ in $X$ is also the exterior of $L$  in $V$ and so is hyperbolic.
Let $T_V = \partial V \subset \partial Y$.

The solid torus $V$ has a meridian disk $D$ which intersects $L$ in three points
such that $P=D\cap Y$ is an essential thrice-punctured disk in $Y$.
Let $d_V = \partial P \cap T_V$ and let $c_1,c_2,c_3$
 be the three components of $\partial P$ contained in $\partial Y\setminus T_V$.
Note  $d_V$ has the slope $\bar \b$ in $T_V$, and each $c_i$ is a meridian curve
of some component of $L$.

Among all annuli satisfying the conditions of Lemma \ref{unique annulus for K},
we choose one, denoted  $A$, which intersects $T_V$ in the minimal number of components.
Since $Y$ is hyperbolic, $A\cap T_V$ is non-empty.
The surface $Q=A\cap Y$ is essential in $Y$.
Since  $A$ is separating in $X$, $\partial Q\cap T_V$ consists of an even number, say $n$,
of simple  essential loops in $T_V$ of slope $\bar \alpha$.
Let $a_1,a_2$ be the two components of $\partial Q$ in $T_K$,  and let $b_1,...,b_{n}$ be
the components of $\partial Q$ in $T_V$ numbered so that they occur successively around $d_V$.
Each $a_i$ is a Seifert fibre of $X$, and each $b_j$ has slope $\bar \alpha$ on $T_V$. If $c_j$ is a meridian curve of $K$, then the distance between $c_j$ and $a_i$  is $1$ since $K$ is a regular fibre of $L(p,2q)$.

Now define the labeled intersection graphs $\Gamma_P$ and $\Gamma_Q$  as usual.
We may consider $d_V$, $c_1,c_2,c_3$, $a_1,a_2$, $b_1,...,b_{n}$ as
the boundaries of the fat vertices of these graphs.
Each $b_i$, $i=1,...,n$,  has valency $p = \Delta(\bar \alpha,\bar \beta)=\Delta(\alpha,\beta)$, and  the valency of $d_V$
is $np$.
Note that the valency of $a_1$ is equal to the valency of $a_2$ and is equal to
the number of $c_i$'s which are meridians of $K$. Further, the valency of $c_i$
is either $2$ or $0$ depending on whether $c_i$ is a meridian curve of $K$ or not.

We call the edges in $\Gamma_Q$ connecting some $b_i$ to some $b_j$
{\it B-edges}, and call the edges in $\Gamma_P$ connecting  $d_V$ to itself  {\it $D$-edges}.
Similarly we define $A$-edges, $C$-edges, $AB$-edges, and $CD$-edges.
Note that an arc in $P\cap Q$ is a $B$-edge in $\Gamma_Q$ if and only if it is a $D$-edge in $\Gamma_P$, is an $A$-edge in $\Gamma_Q$ if and only if it is an $C$-edge in $\Gamma_P$,
and is an $AB$-edge in $\Gamma_Q$ if and only if it is a $CD$-edge in $\Gamma_P$.

Every $D$-edge is positive, so by the parity rule, every $B$-edge is negative.
By construction, no $D$-edge in $\Gamma_P$ is boundary parallel
in $P$. Thus there are at most three different $D$-edges in the reduced graph $\overline{\G}_P$ (cf. Figure \ref{d-edges}).

\begin{figure}[!ht]
\centerline{\includegraphics{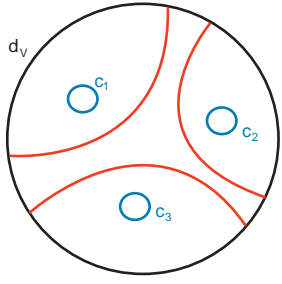}} \caption{The maximal possible $D$-edges in
$\overline{\G}_P$ }\label{d-edges}
\end{figure}

\begin{lemma}\label{S cycle with D-edges} There can be no
$S$-cycle in $\Gamma_P$ consisting of $D$-edges.
\end{lemma}

\begin{proof} Suppose otherwise that
 $\{e_1, e_2\}$ is an $S$-cycle in $\Gamma_P$ consisting of $D$-edges with
label pair $\{j,j+1\}$.
We may assume that  the bigon face $E$ between $e_1$ and $e_2$
lies on the $X_1$-side of $A$.

Let $H$ be the portion of the filling solid torus of $L(p, 2q)$ lying in $X_1$ which contains
$\hat b_j$ and $\hat b_{j+1}$.
In $\Gamma_Q$, $e_1\cup b_j \cup e_2\cup b_{j+1}$ cannot be contained in a disk region $D_*$ of $A$ as otherwise a regular neighborhood of $D_*\cup E\cup H$ in $X_1$ would be a punctured  projective space.
Thus $e_1\cup b_j \cup e_2\cup b_{j+1}$ contains a core circle of $A$ (cf. Figure \ref{S-cycle}).

\begin{figure}[!ht]
\centerline{\includegraphics{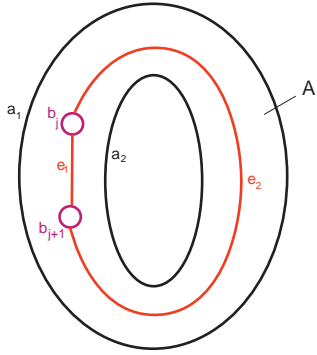}} \caption{The corresponding
cycle $\{e_1, e_2\}$ in
$\G_Q$}\label{S-cycle}
\end{figure}

Let $U$ be a regular neighborhood of $E\cup H\cup A$ in $X_1$. Then $U$ is a solid torus and the frontier
of $U$ in $X_1$ is an annulus $(A', \partial A')\subset (X, T_K)$ for which $\partial A'$ is parallel to $\partial A$ in $T_K$ and which intersects $T_V$ in $n - 2$ components. By construction, $A'$ is inessential in $X_1$ and therefore
$X_1$ cannot be a torus cross interval.
It follows that $X_1$ is a fibred solid torus of $X$.
Since $A'$ has winding number $2$ in the solid torus $U$, the singular fibre of $X_1$
has order $2$, contrary to Lemma \ref{unique annulus for K}. Thus the lemma holds.
\end{proof}

Note that $\G_P$ has at most six $CD$-edges and thus
 $\G_P$ has at least $(np-6)/2$ $D$-edges, so there is a family of at least $(np-6)/6$
mutually parallel $D$-edges.
By Lemma \ref{S cycle with D-edges} we have $(np-6)/6\leq n/2$. Hence $n\leq 6/(p-3)$
and therefore $p=5$ and $n=2$. If $\G_P$ has a $C$-edge, it would have only one family of parallel
$D$-edges, and this family would have at least three edges, contrary to the fact that no two $D$-edges can be parallel in
$\Gamma_P$ by Lemma \ref{S cycle with D-edges}.
Also, $\G_P$ has at least four $CD$-edges as otherwise there would be four $D$-edges
two of which would form an $S$-cycle.
Thus $\G_P$ has either six or four $CD$-edges.

We first consider the case when there are exactly four $CD$-edges.
In this case we have three $D$-edges in $\G_P$, no two of which can be parallel.
Hence $\G_P$ may be assumed
 to be as illustrated in Figure \ref{db not parallel}, i.e. $c_1$ and $c_2$ are contained in $T_K$ and $c_3$ is contained $\partial X\setminus T_K$.
 Thus $|L|=|\partial X|=2$ and we may assume that $X_1$ is a solid torus and $X_2$ is a
 torus cross interval.
In particular $c_3$ is contained in $X_2$.

\begin{figure}[!ht]
\centerline{\includegraphics{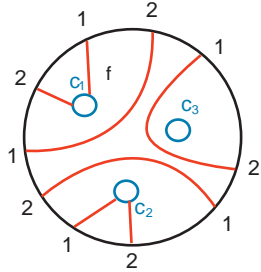}} \caption{
$\G_P$ when $\Delta(\alpha, \beta)=5$, $n=2$ and $4$ $CD$-edges.}\label{db not parallel}
\end{figure}

Consider the face $f$ given in Figure \ref{db not parallel}.
From the figure we see that  $f$ and $c_3$ are on the same side of $A$ (since $A$ is separating
 in $X$) and thus
$f$ is contained in $X_2$.
Let $T_*$ be the component of $\partial X_2$ containing $A$ and $H$ that part of filling solid torus of $L(p, 2q)$ contained
in $X_2$. We use $\partial_0 H$ to denote $\partial H \cap T_V$. It is evident that the boundary $\partial f$ of $f$
is contained in $T_*\cup \partial_0  H$.
Also note that $\partial f\cap T_*$ cannot be contained in
a disk in $T_*$ as otherwise $X_2$ would contain a projective space as a summand.
Thus $\partial f\cap T_*$ is contained in an annulus $A_*$ of $T_*$.
A regular neighborhood $W$ of $H \cup f\cup T_*$ in $X_2$ is a Seifert fibred space whose base orbifold is an annulus with a cone point of order $2$. Since $X_2$ is a torus cross interval, the frontier of $W$ in $X_2$ is an incompressible torus in $X_2$.
But this torus cannot be parallel to $T_*$ in $X_2$, contradicting the fact that
 $X_2$ is a torus cross interval. Thus the case when there are exactly four $CD$-edges does not arise.

We now know that $\G_P$ must have  six $CD$-edges.
Hence  there are exactly two $D$-edges in $\G_P$ and they are not parallel. It follows that $\G_P$ is as illustrated in
Figure \ref{no-c-edges} (1) or (2). (Without
loss of generality, we may assume that the labels around $d_V$ are as shown in these figures and that the vertices $c_1$, $c_2$ and $c_3$
are numbered  as given there.)
Therefore  $L=K$ and both $X_1$ and $X_2$ are solid tori.

We are going to show that  part (1) of Figure \ref{no-c-edges} cannot arise and that in case of part (2) of Figure \ref{no-c-edges}
 the dual graph $\G_Q$ may be assumed to be as shown in part (6) of
 Figure \ref{dual-of-f}.

 \begin{figure}[!ht]
\centerline{\includegraphics{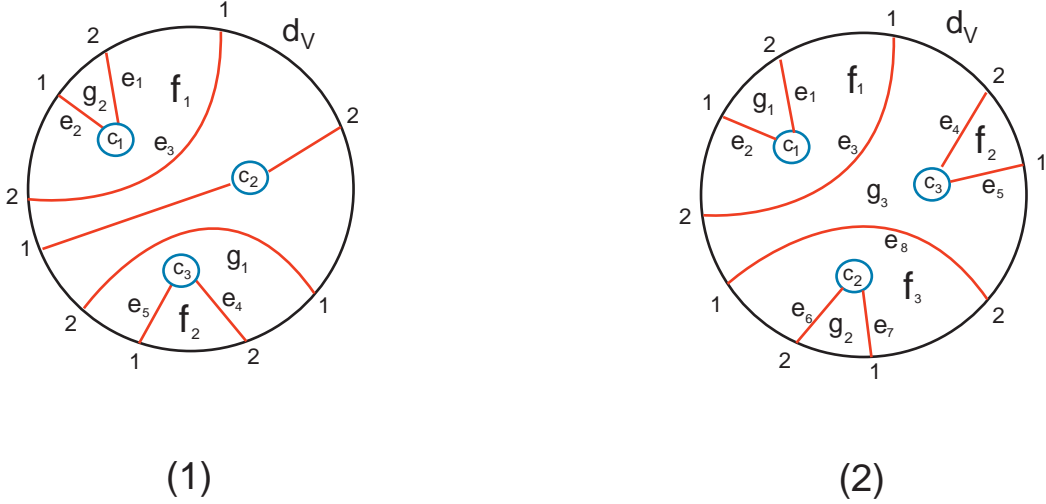}} \caption{$\G_P$ when  $p=5$, $n=2$ with $6$ $CD$-edges}
\label{no-c-edges}
\end{figure}

\begin{lemma}\label{part (1) impossible} The graph $\G_P$ cannot be as shown in part (1) of
Figure \ref{no-c-edges}.
\end{lemma}

\begin{proof}
Suppose otherwise that  $\G_P$ is given by part (1) of Figure \ref{no-c-edges}.
Since $A$ is a separating annulus, the faces $f_1, f_2$  of $\G_P$ lie
on the same side of $A$, say in $X_1$, and the faces $g_1, g_2$ lie in
$X_2$.

Let $H$ be the part of the filling solid torus of $L(p, 2q)$ contained
in $X_1$ and set $\partial_0 H = \partial H \cap T_V$.
The boundary edges of $f_1$ consist of two $CD$-edges $e_1$, $e_2$ and one $D$-edge $e_3$.
Without loss of generality, we may assume that the label of the edge  $e_1$ at the vertex $c_1$ is $2$.
In $\G_Q$, the boundary edges of $f_1$ may be assumed to be as illustrated
in part (1) of Figure \ref{dual-of-f}.
Note that the boundary $\partial f_1$ of $f_1$, including the corners,
lies in $\partial X_1 \cup  \partial_0  H$.
Further, $\partial f_1\cap \partial X_1$ is contained in an annulus $A_*$ of $\partial X_1$ whose slope has
distance $1$ from that of $\partial A$. Note as well that $\partial f_1\cap (\partial X_1\setminus A)$ is an essential arc in the annulus
$(\partial X_1 \setminus A)$.
A regular neighborhood $U$ of $H \cup f_1 \cup A_*$ in $X_1$ is a solid torus whose frontier in $X_1$ is an annulus $A_\#$
of winding number $2$ in $U$.
Thus $A_\#$ must be parallel to $\partial X_1\setminus A_*$ through  $X_1\setminus U$.
It follows that the fundamental group of $X_1$ is carried by $U$ and thus
has presentation
$$<x,t; x^2t=1>$$
where we take a fat base point in $A$ containing $b_1\cup b_2\cup (\partial f_1 \cap A)\cup (\mbox{all $AB$-edges})$,
$x$ is a based loop formed by a cocore arc of $\partial_0 H$, and $t$ is a based loop formed by a cocore arc of $\partial X_1\setminus A$.

Now consider the face $f_2$.
We claim that the label of the edge $e_4$ at the vertex  $c_3$ cannot be $2$.
Otherwise in $\G_Q$, the boundary edges of $f_2$, $e_4$ and $e_5$ would be as depicted in part (2) or part (3) of Figure \ref{dual-of-f}. In either case, the face  $f_2$ would add the relation
 $xts=1$ to the presentation for $\pi_1(X_1)$ above, where $s$ is the element represented by a core circle of the annulus $A$. Thus the fundamental group of the solid torus $X_1$ would be generated by $s=x$.
 But $s$ can be considered as a regular fibre of
 $X$. So the singular fiber of $X_1$ would have order one, which contradicts Lemma \ref{unique annulus for K}.

 Thus the label of $e_4$ at $c_3$ is $1$. It follows that in $\G_Q$, the edges $e_4$ and $e_5$
 are as shown in part (4) of Figure \ref{dual-of-f}, and the face $f_2$ adds the relation
 $xt^{-1}s=1$ to the presentation for $\pi_1(X_1)$, where $s$ is the element represented by a core circle of the annulus $A$.
Therefore $s=x^{-3}$.
Since $s$ can be considered as a regular fibre of
 $X$ and   $x$ can be considered as  a core circle of the solid torus $X_1$, the singular fibre in $X_1$ has order $3$.

With the same argument, we see that the existence of  the faces $g_1$ and $g_2$ in part (1) of Figure \ref{no-c-edges}
implies that the singular fiber in $X_2$ has order $3$.
Hence the two singular fibers of  $X$ both have order $3$ which implies that the order of the lens space $L(p, 2q)$ is divisible by $3$.
But the lens space has order $p = \Delta(\alpha, \beta) = 5$, yielding a contradiction.
So part (1) of Figure \ref{no-c-edges} cannot arise.
\hfill \end{proof}

 \begin{figure}[!ht]
\centerline{\includegraphics{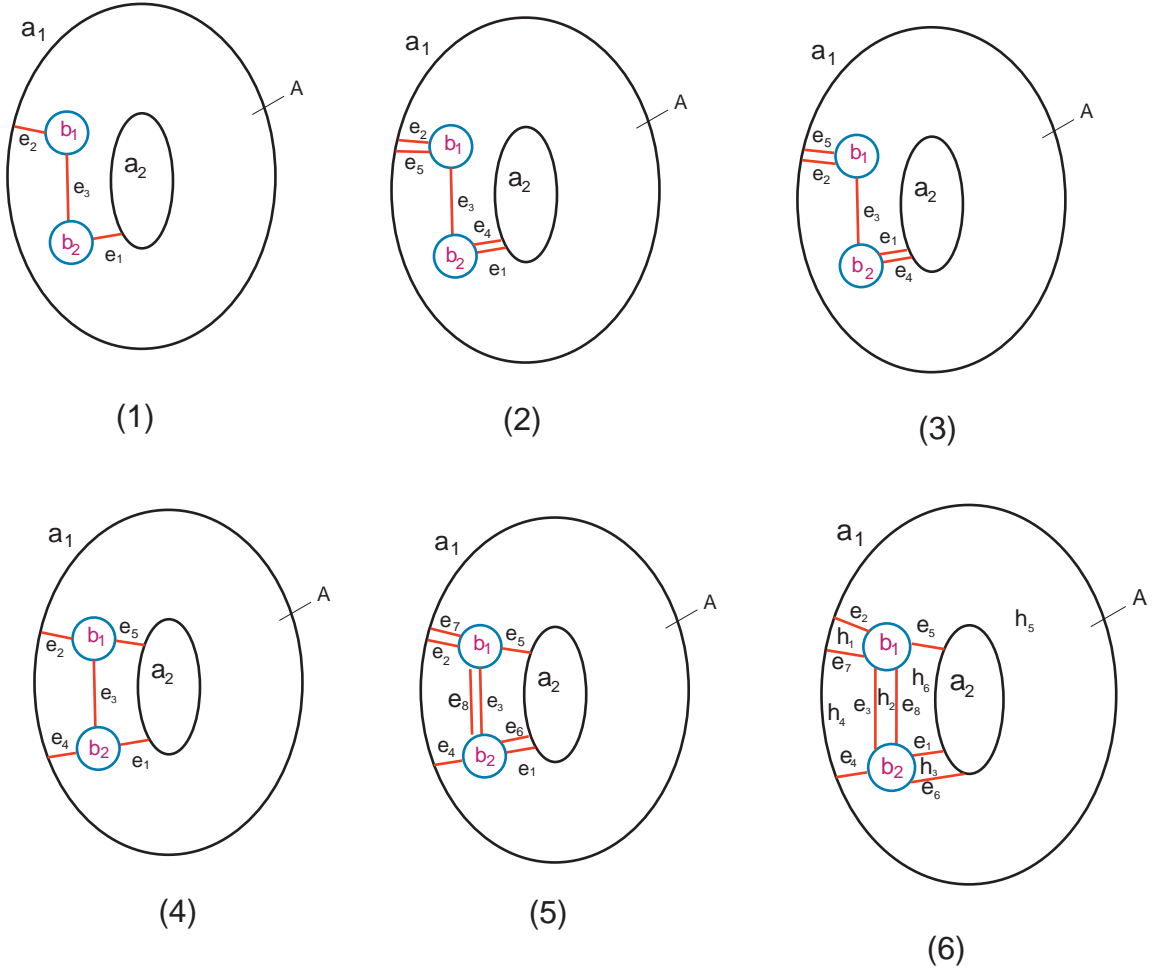}} \caption{About the graph $\G_Q$
}\label{dual-of-f}
\end{figure}

So $\G_P$ must be as shown in part (2) of Figure \ref{no-c-edges}.
Note that the faces $f_1,f_2,f_3$ lie on the same side of $A$, say $X_1$, and the faces $g_1,g_2,g_3$ in $X_2$.
Arguing similarly as the proof of Lemma \ref{part (1) impossible},  we see that in the dual graph $\G_Q$ the edges
$e_1, e_2, e_3, e_4$ and $e_5$
 may be  assumed to be as shown in part (4) of Figure \ref{dual-of-f}.

We now consider the face $g_3$. Note that $\partial g_3$ must be  contained in an annulus $A'$ of $\partial X_2$
whose slope has distance $1$ from that of $\partial A$
and that $\partial g_3 \cap (\partial X_2\setminus A)$ is an essential arc in the annulus
$(\partial X_2\setminus A)$.  Thus $e_8$ is
parallel to $e_3$ in $\G_Q$. By combining this with the argument given in Lemma \ref{part (1) impossible}
we see that the graph $\G_Q$ must be as depicted in part (5) or part (6)
of Figure \ref{dual-of-f}.

\begin{lemma} \label{9(5) impossible}
Figure \ref{dual-of-f}(5) is impossible.
\end{lemma}

\begin{proof}
In Figure \ref{no-c-edges}(2), let $p_0, p_1, p_2, p_3, p_4$ be the points labeled $1$ on $d_V$, in cyclic order around $d_V$. These are points of intersection of $b_1$ with $d_V$ on the torus $T_V$. It follows that the corresponding points appear around $b_1$ in the order $p_0, p_d, p_{2d}, p_{3d}, p_{4d}$, for some $d$ coprime to $\Delta = \Delta(\alpha, \beta) = 5$. The point $p_i$ is the endpoint of an edge $e_{j(i)}$. Then, denoting $p_i$ by the label $j(i)$ of the corresponding edge, the cyclic order of the $p_i$'s around $d_V$ in Figure \ref{no-c-edges}(2) is 28753. In the graph $\Gamma_Q$ in Figure \ref{dual-of-f}(5), the order of the corresponding points is $82753$. Since these cyclic orderings are not related in the manner described above, $\Gamma_Q$ cannot be as illustrated in Figure \ref{dual-of-f}(5).
\end{proof}

\begin{rem} \label{9(6)}
{\rm In Figure \ref{dual-of-f}(6) the order is $27385$, which is of the required form, with $d = 2$. }
\end{rem}

So far we have shown that $p=\D(\alpha,\beta)=5$ and the graphs $\G_P$ and $\G_Q$ must be as shown in part (2) of Figure \ref{no-c-edges}
and part (6) of Figure \ref{dual-of-f} respectively.
In the rest of this section we are going to show that these conditions determine the
triple $(M, \alpha,\beta)$ uniquely up to homomorphism, and thus it must be the triple
$(Wh(-3/2); -5, 0)$.

The surface $Q$ separates $Y$ into $Y_1$ and $Y_2$, say, where $Y_i \subset X_i$, $i = 1, 2$. Let $N$ be a regular neighbourhood of $T_V \cup T_K \cup P \cup Q$ in $Y$, and let $\partial_0 N = \partial N \setminus (T_V \cup T_K)$. Then $\partial_0N = \partial_1 N \cup \partial_2 N_2$ where $\partial_i N \subset Y_i$, $i = 1, 2$.

\begin{lemma}
For $i  = 1$ and $2$, $\partial_i N$ has two components, each a $2$-sphere.
\end{lemma}

\begin{proof}
By Remark \ref{9(6)}, the curves $d_V, b_1, b_2$ on the torus $T_V$ are as shown in Figure \ref{bgz4-torus-Tv}.
They decompose $T_V$ into rectangles $R_1, \ldots, R_5, S_1, \ldots , S_5$, where the $R_i$'s lie in $Y_1$ and the $S_i$'s in $Y_2$. In Figure \ref{bgz4-torus-Tv} a point of intersection of $b_1 \cup b_2$ with $d_V$ is labeled with the edge of which it is an endpoint. Similarly, the curves $a_1, a_2, c_1, c_2, c_3$ decompose the torus $T_K$ into rectangles $T_1, T_2, T_3, U_1, U_2, U_3$, where the $T_j$'s lie in $Y_1$ and the $U_j$'s in $Y_2$. See Figure \ref{bgz4-torus-Tk}.

\begin{figure}[!ht]
\centerline{\includegraphics{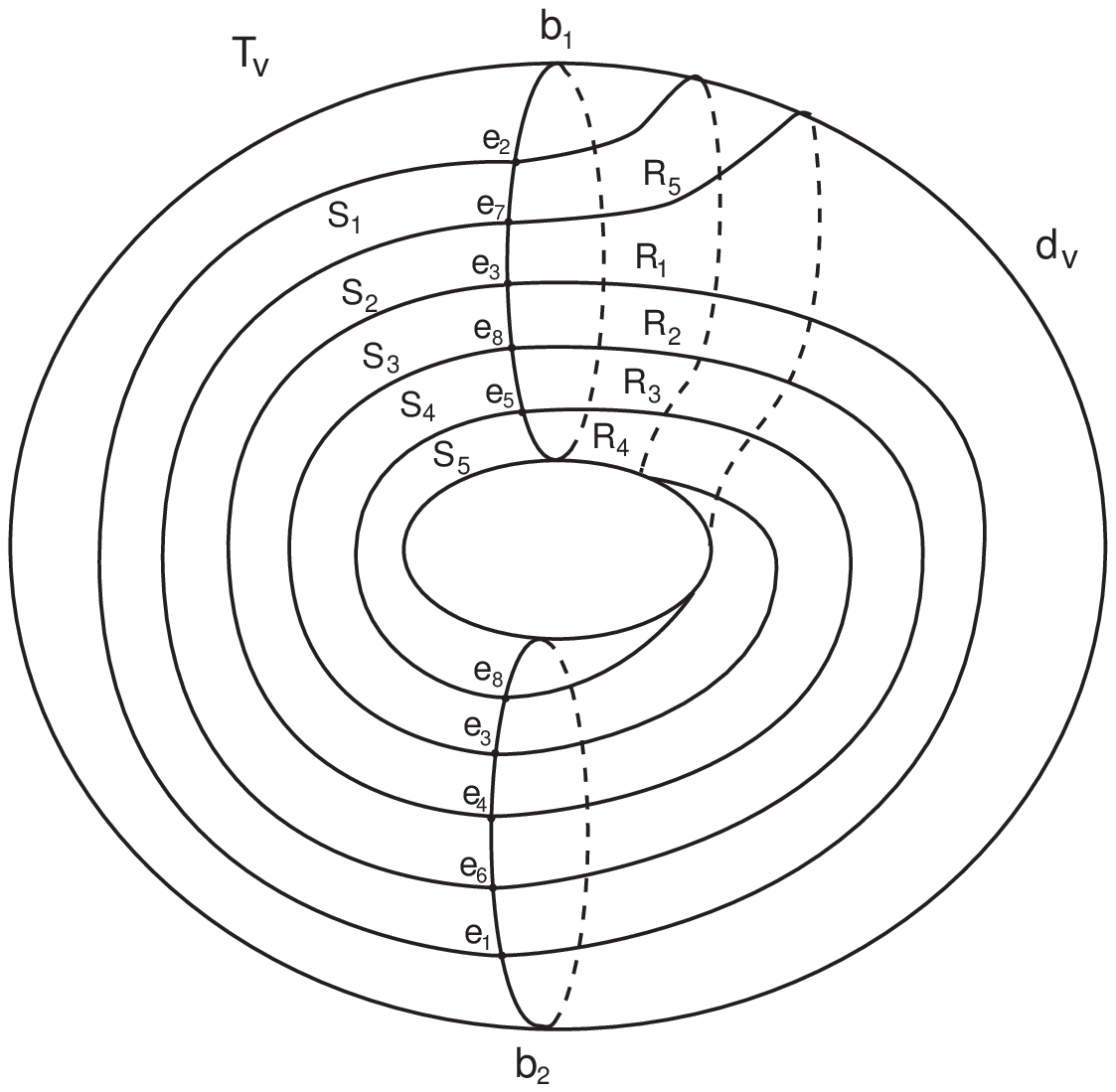}}\caption{} \label{bgz4-torus-Tv}
\end{figure}

\begin{figure}[!ht]
\centerline{\includegraphics{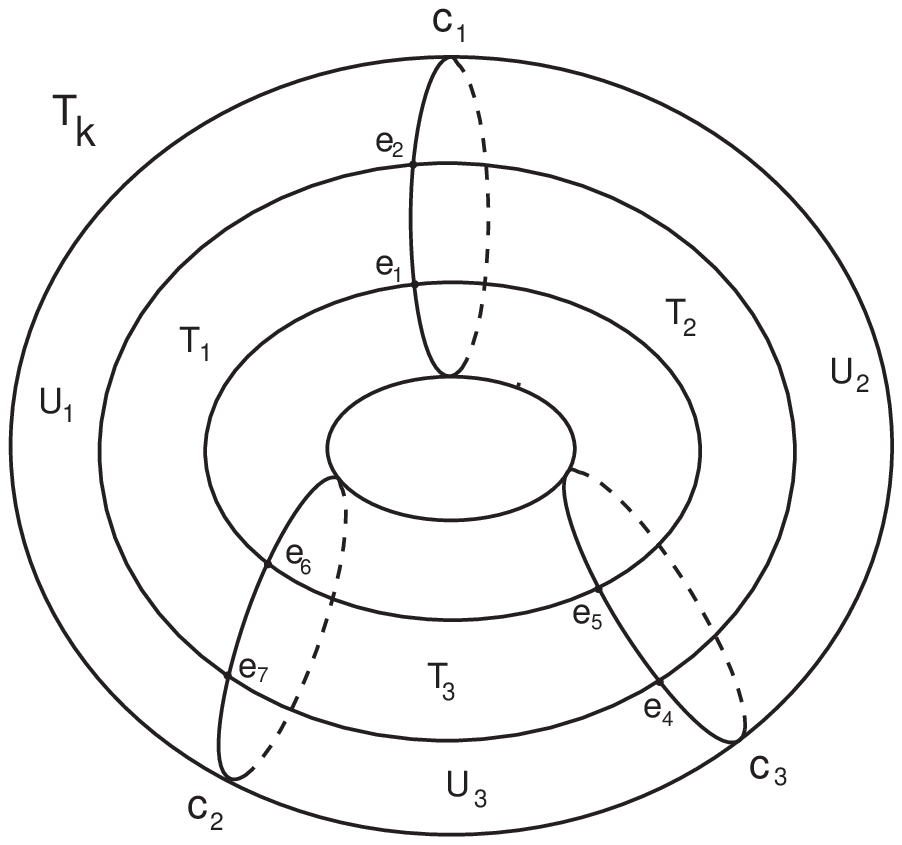}} \caption{} \label{bgz4-torus-Tk}
\end{figure}

The faces of the graph $\Gamma_P$ are $f_1, f_2, f_3, g_1, g_2, g_3$, where the $f_i$'s lie in $Y_1$ and the $g_i$'s lie in $Y_2$; see Figure \ref{no-c-edges} (2). Let the faces of $\Gamma_Q$ be $h_1, \ldots , h_6$, as shown in Figure \ref{dual-of-f}(6).

The regular neighbourhood $N$ is the union of product neighbourhoods $T_V \times [0,1], T_K \times [0,1], P \times [-1, 1]$ and $Q \times [-1, 1]$, in the obvious way, where $T_V = T_V \times \{0\}, T_K = T_K \times \{0\}, P = P \times \{0\}$, and $Q = Q \times \{0\}$.
Corresponding to $R_i$ is a $2$-cell contained in $(T_V \times \{1\}) \cap \partial_0 N$, which we continue to denote by $R_i$; similarly for $S_i, T_j$ and $U_j$. A face $f_i$ of $\Gamma_P$ gives rise to two $2$-cells $f_i^+ \subset (P \times \{1\}) \cap \partial_0 N$ and $f_i^- \subset (P \times \{-1\}) \cap \partial_0 N$, and similarly for the $g_i$'s and the faces $h_k$ of $\Gamma_Q$. Since $h_k^+ \hbox{ (say) } \subset \partial_1 N$ and $h_k^- \subset \partial_2 N$, there will be no confusion in denoting $h_k^{\pm}$ by $h_k$.

By carefully examining the identifications between these various $2$-cells one sees that $\partial_1 N$ has two components $\Sigma_1$ and $\Sigma_1'$, and $\partial_2 N$ has two components $\Sigma_2$ and $\Sigma_2'$, composed of the following $2$-cells:
\begin{description}

\vspace{-.2cm} \item[$\Sigma_1$] $f_1^+, f_3^-, h_1, h_2, h_3, R_2, R_5, T_1$

\vspace{.2cm} \item[$\Sigma_1'$] $f_1^-, f_2^+, f_2^-, f_3^+, h_4, h_5, h_6, R_1, R_3, R_4, T_2, T_3$

\vspace{.2cm} \item[$\Sigma_2$] $g_1^+, g_2^-, h_1, h_3, S_1, U_1$

\vspace{.2cm} \item[$\Sigma_2'$] $g_1^-, g_2^+, g_3^+, g_3^-, h_2, h_4, h_5, h_6, S_5, S_2, S_3, S_4, U_2, U_3$

\end{description}

The precise patterns of identification are shown in Figures \ref{bgz4-sphere1}, \ref{bgz4-sphere1'}, \ref{bgz4-sphere2} and \ref{bgz4-sphere2'}, respectively. In particular, $\Sigma_1, \Sigma_1', \Sigma_2, \Sigma_2'$ are $2$-spheres.
\end{proof}

\begin{figure}[!ht]
\centerline{\includegraphics{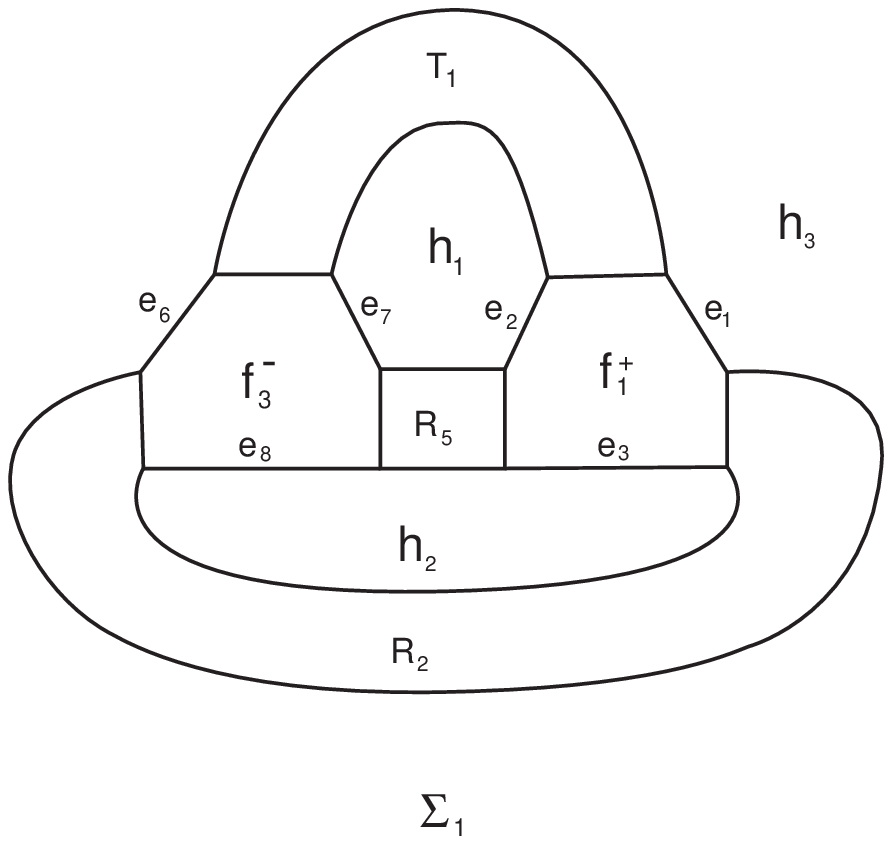}}  \caption{} \label{bgz4-sphere1}
\end{figure}

\begin{figure}[!ht]
\centerline{\includegraphics{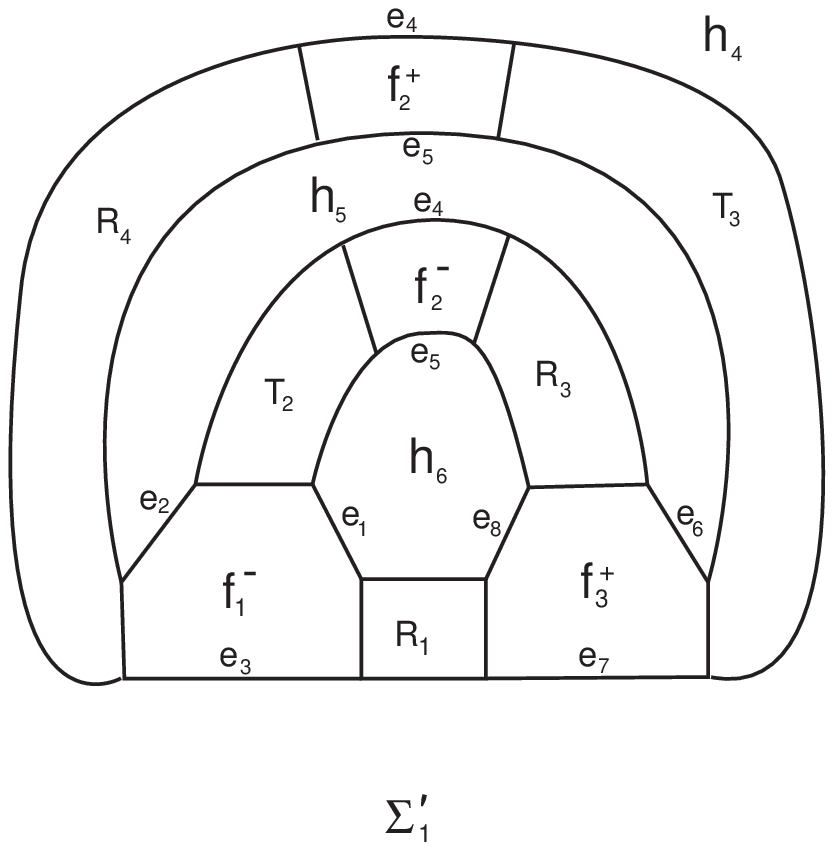}}\caption{} \label{bgz4-sphere1'}
\end{figure}

\begin{figure}[!ht]
\centerline{\includegraphics{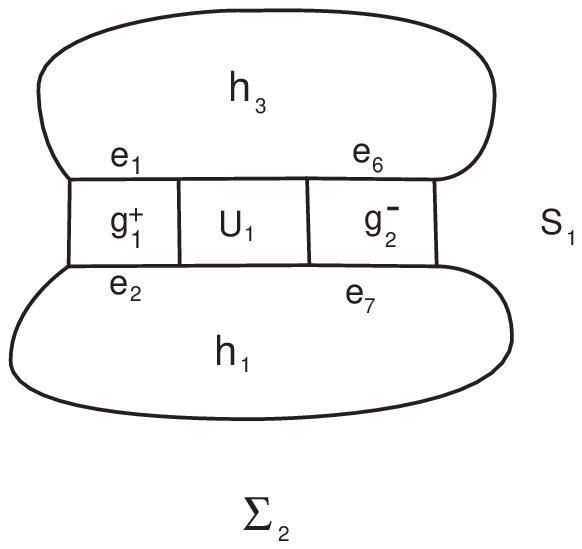}} \caption{} \label{bgz4-sphere2}
\end{figure}

\begin{figure}[!ht]
\centerline{\includegraphics{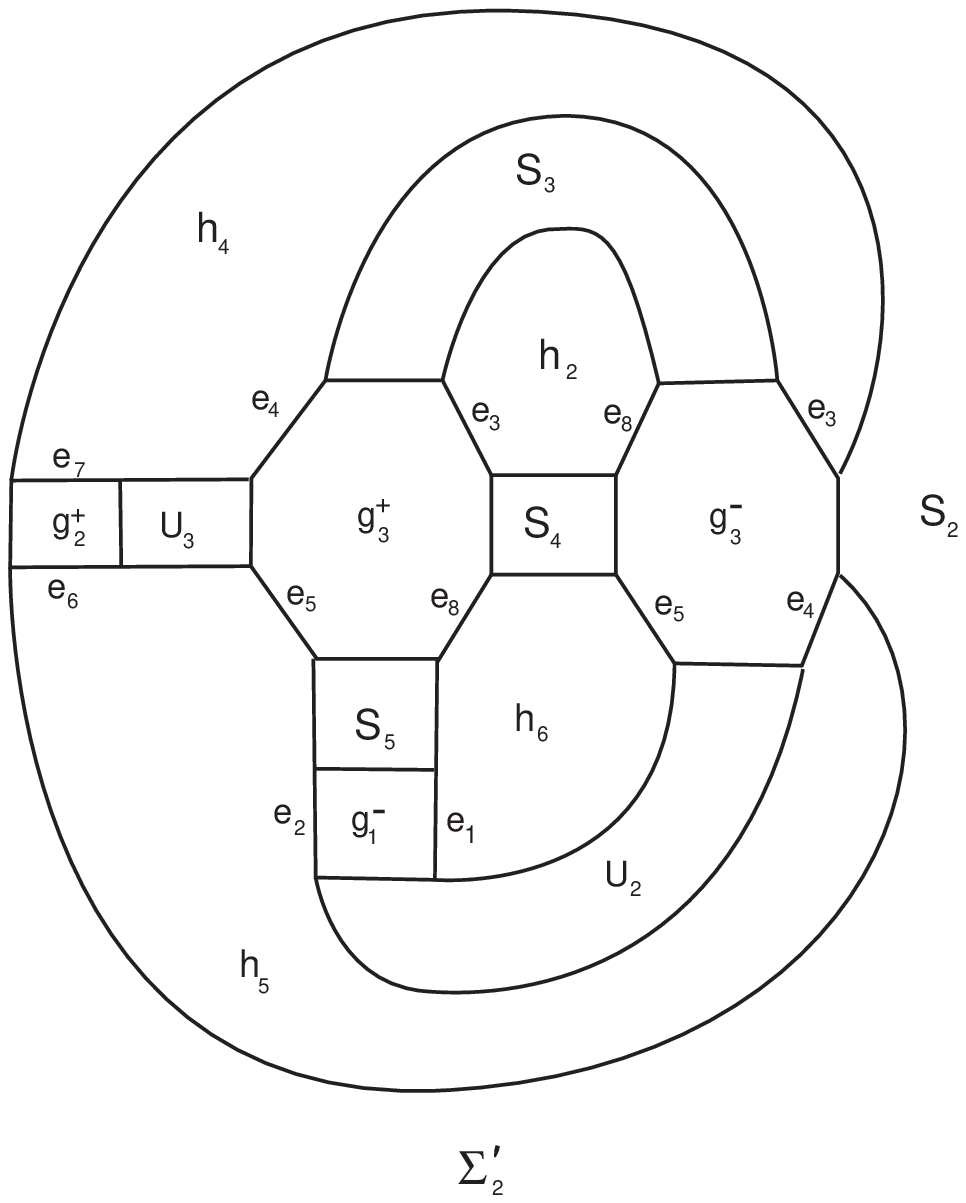}}\caption{}  \label{bgz4-sphere2'}
\end{figure}

\begin{rem}
{\rm One can see that $\Sigma_1, \Sigma_1', \Sigma_2, \Sigma_2'$ are $2$-spheres without completely determining the identification patterns of their constituent $2$-cells, by means of the following Euler characteristic computation.

First note that
$$\chi(P \cup Q) = \chi(P) + \chi(Q) - \chi(P \cap Q) = (-2) + (-2) -8 = -12$$
Also, $(P \cup Q) \cap T_V$ consists of three circles, meeting in a total number of $10$ points. So $\chi((P \cup Q) \cap T_V) = -10$. Similarly $\chi((P \cup Q) \cap T_K) = -6$. Therefore
$$\chi(N) = \chi((P \cup Q) \cup (T_V \cup T_K)) = (-12) + 0 -((-10) + (-6)) = 4$$
Hence $\chi(\partial N) = 8$.

Now one can easily check that each of $\partial_1 N$ and $\partial_2 N$ has at most two components. Hence each must have exactly two components, both $2$-spheres. }
\end{rem}

\begin{proof}[Proof that $(M; \alpha, \beta)$ is homeomorphic to $(Wh(-3/2); -5, 0)$]
Since $Y$ is irreducible the components of $\partial_0 N$ bound $3$-balls in $Y$. Hence the triple $(Y; P, Q)$ is uniquely determined up to homeomorphism, by Figures \ref{no-c-edges} (2) and \ref{dual-of-f}(6). Since the curves $c_j$ are meridians of $L$, the pair $(V, L)$, together with the slopes $\bar \alpha, \bar \beta$, is uniquely determined. Passing to the double branched cover, we have that $(M; \alpha, \beta)$ is uniquely determined.

In \cite[Table A3]{MP} it is shown that $-5$-filling on the hyperbolic manifold $Wh(-3/2)$ is Seifert fibred with base orbifold $S^2(2,3,3)$, while $0$-filling gives a manifold containing a non-separating torus. In fact, it is easy to see that $Wh(-3/2)$ contains an essential once-punctured torus with boundary slope $0$. Hence $(M; \alpha, \beta) \cong (Wh(-3/2); -5, 0)$.
\end{proof}

\section{The case $\Delta(\alpha, \beta) = 7$ and the involution $\tau_\alpha$ reverses the orientations of the Seifert fibres of $M(\alpha)$}\label{dist7}

In this section we suppose that assumptions \ref{assumptions 1} hold and show that it is impossible for $\Delta(\alpha, \beta)$ to be $7$ and for $\tau_\alpha$ to reverse the orientations of the Seifert fibres of $M(\alpha)$. We assume otherwise in order to obtain a contradiction.

A {\em tangle\/} will be a pair $\T = (R,t)$, where $R$ is $S^3$ minus
the interiors of a disjoint union of 3-balls, and $t$ is a properly
embedded 1-manifold.
Let $\wT = (X,\tilt\,)$ be the double branched cover of $\T$.
In our examples each boundary component $S$ of $R$ will meet $t$ in either
4 or 6 points, and hence the corresponding boundary component $\widetilde S$ of $X$
is either a torus or a surface of genus~2, respectively.

An {\em essential disk\/} in $\T$ is a properly embedded disk $D$ in $R$
such that either
\begin{itemize}
\item[(i)] {\em $D\cap t =\emptyset$ and $\partial D$ does not bound a disk in
$\partial R \setminus t$}; or
\item[(ii)] {\em $D$ meets $t$ transversely in a single point and
$\partial D$ does not bound a disk in $\partial R$ containing a single
point of $t$}.
\end{itemize}
It follows from the $\zed /2$-equivariant Disk Theorem (\cite{GLi}, \cite{KT},
\cite{YM}) that $X$ contains an essential disk $\wD$, i.e., a properly embedded
disk such that $\partial \wD$ is essential in $\partial X$, if and only if $\T$
contains an essential disk $D$.

If $S$ is a boundary component of $R$ such that $|S\cap t| = 4$, a
{\em marking\/} of $S$ is a specific identification of $(S,S\cap t)$ with
$(S^2 ,\{NE,NW,SW,SE\})$.
We can then attach a rational tangle $\R (\g)$ to $\T$ along $S$ with
respect to this marking, where $\g\in \que \cup \{\sfrac10\}$.

By Lemma \ref{--quotient} (1), $M(\alpha)=M(7/q)$ has base orbifold $S^2(7, 7,
m)$ for some odd integer $m\geq 3$. As in Lemma \ref{cyclic-cover}, let $\widetilde
M_7$ be the $7$-fold cyclic cover of $M$; then $\partial \widetilde M_7$ is a single
torus, and both $\alpha$ and $\beta$ lift to slopes $\tilde \a$ and $\tilde \b$ in $\p
\widetilde M_7$, i.e. $\widetilde M_7(\tilde \alpha)$ is a $7$-fold cyclic cover of
$M(\alpha)$ and $\widetilde M_7(\tilde \beta)$ is a $7$-fold cyclic cover of $M(\beta)$.
Furthermore the involution $\t$ on $M$ lifts to an involution $\tilde \t$ on
$\widetilde M_7$  and  $\tilde V=\widetilde M_7/ \tilde \t$ is a $7$-fold cyclic
cover of $M/\t=V$. So $\tilde V$ is a solid torus. The involution $\tilde\t$
extends to an involution $\tilde \tau_{\tilde \a}$ on $\widetilde M_7(\tilde \alpha)$
such that $\widetilde M_7(\tilde \alpha)/\tilde \tau_{\tilde \a}=S^3$ is the $7$-fold
cyclic cover of the lens space $M(\alpha)/ \tau_{\a}=L(7, 2q)$. Let $L_7$ be the
inverse image of $L$ in $S^3$.
  Then by Lemma \ref{types},  $L_7$ is as shown
  in Figure \ref{7ab} where the box with an integer $r$ in it stands for $r$
 full horizontal  twists, and
  by Lemma \ref{--quotient} (2),
  $L_7$ is also as
 shown   in Figure \ref{7mn}, where the box with an integer $r'$ in it
 stands for $r'$ full horizontal twists. Since $p = 7$, $n$ is odd by Lemma \ref{--quotient 2}(3). Hence from
Figure \ref{7mn} we see that $L_7$ is a single knot.
So to get a contradiction, we just need to show that the two knots $K$ and $K'$ shown in
Figures \ref{7ab} and \ref{7mn}  respectively are inequivalent.

\begin{figure}[!ht]
\centerline{\includegraphics{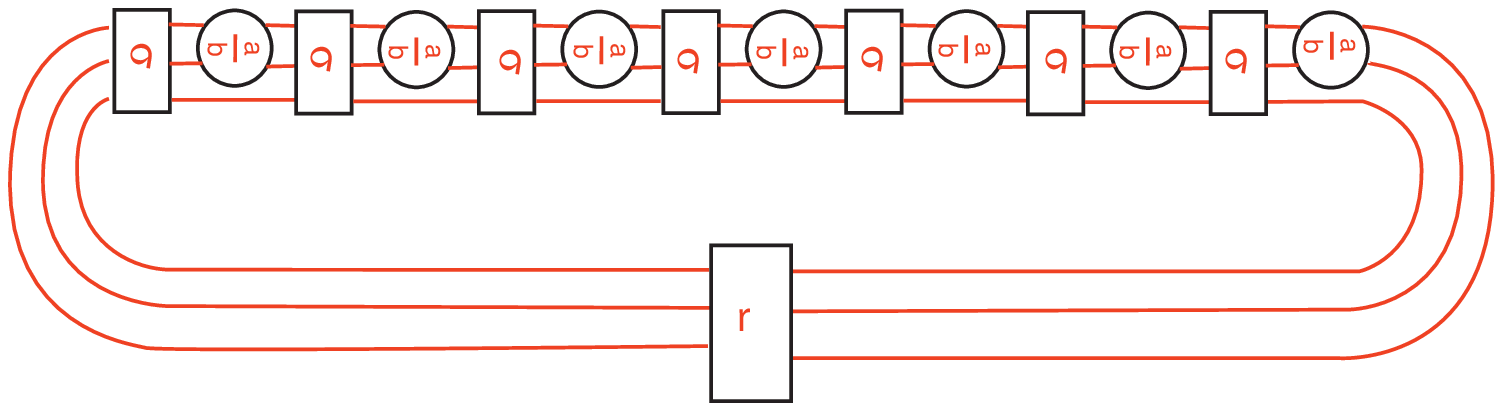}} \caption{ }\label{7ab}
\end{figure}

\begin{figure}[!ht]
\centerline{\includegraphics{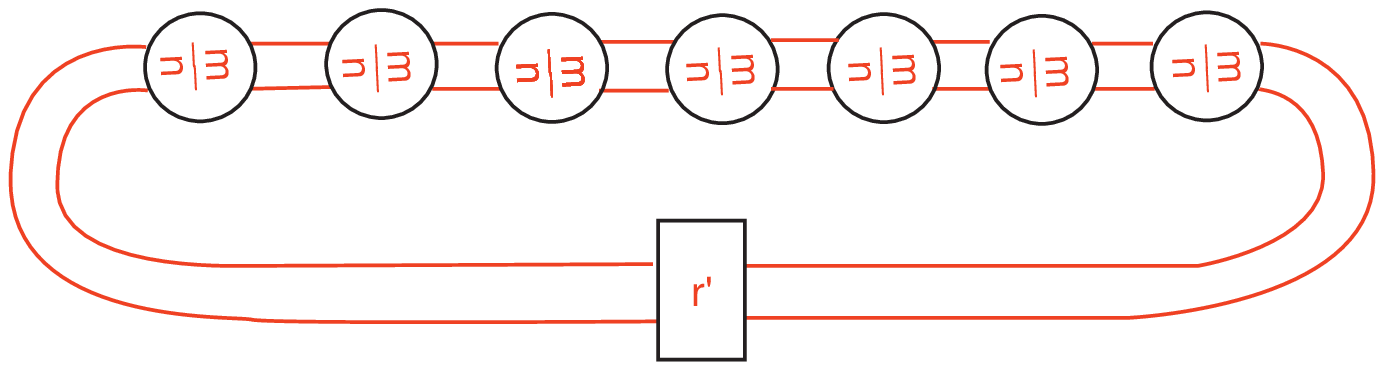}} \caption{ }\label{7mn}
\end{figure}

\begin{thm}\label{thm1}
The knots $K$ and $K'$ are inequivalent.
\end{thm}

Let $W, W'$ be the double cover of $S^3$ branched over  $K,K'$, respectively.
We shall show that $W$ and $W'$ are not homeomorphic. Note that $W'$ is a
Seifert fibred  manifold with  base orbifold $S^2 (m,m,m,m,m,m,m)$. We will examine
$W$ and show that it cannot be such Seifert manifold.

\begin{figure}[!ht]
\centerline{\includegraphics{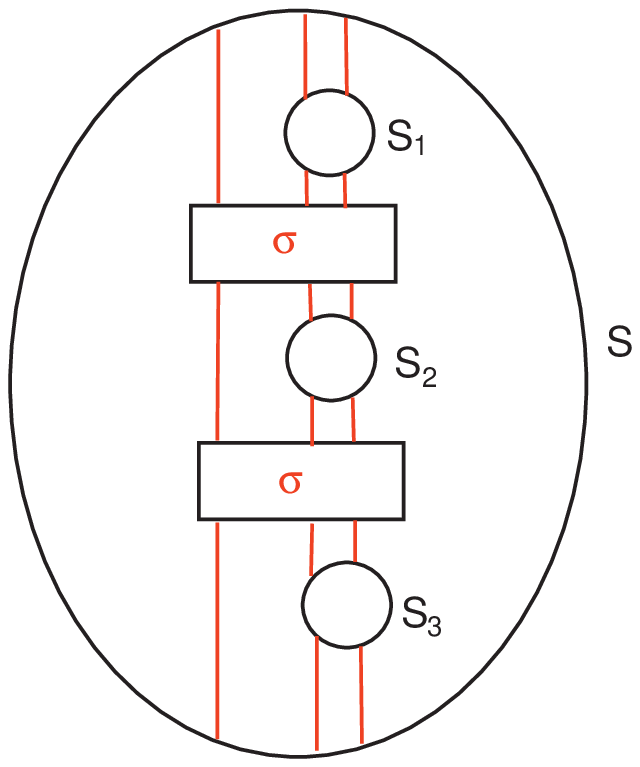}} \caption{ }\label{bgz4-fig4}
\end{figure}

Let $\T = (R,t)$ be the tangle shown in Figure \ref{bgz4-fig4}. Let the boundary
components of $R$ be $S,S_1,S_2,S_3$ as shown. Note that $|t\cap S| =6$ and
$|t\cap S_i| =4$, $i=1,2,3$. Let $X$ be the double branched cover of $\T$. Then
$\partial X =G\ \raise.3ex\hbox{$\scriptstyle\coprod$}\ \coprod_{i=1}^3 T_i$,
where $G$ is the double branched cover of $(S,S\cap t)$ and $T_i$ the double
branched cover of $(S_i,S_i\cap t)$, $i=1,2,3$; thus $G$ has genus~2 and the
$T_i$ are tori.

\begin{rem} \label{rem1}
{\rm The permutation induced by $\s$ takes 1 to 2 or 3, since $K$ is connected.}
\end{rem}

\begin{prop}\label{prop2}
$X(\sfrac{a}b, \sfrac{a}b, \sfrac{a}b)$ is either \\
$(1)$ boundary-irreducible; or \\
$(2)$ the boundary connected sum of two copies of a Seifert fibred manifold with
base orbifold $D^2 (a,d)$, $d>1$; or \\
$(3)$ a handlebody of genus $2$.
\end{prop}

We prove Proposition~\ref{prop2} by successively filling along $T_1$, $T_3$
and $T_2$.

\begin{lemma}\label{lem3'}
$G$ is incompressible in $X$.
\end{lemma}

\begin{proof}
Because of Remark \ref{rem1}  above, the arrangement of the components of $t$
with respect to the boundary components of $R$ is as illustrated schematically
in Figure \ref{bgz4-fig5}. It follows easily that $\T= (R,t)$ cannot contain any
essential disk $D$ with $\partial D\subset S$.
\end{proof}

\begin{figure}[!ht]
\centerline{\includegraphics{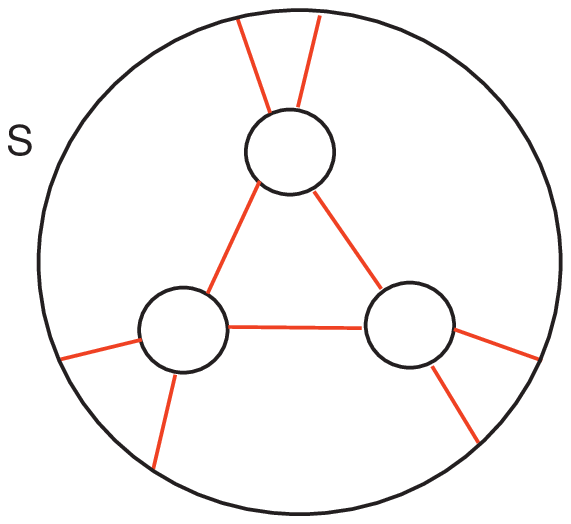}} \caption{ }\label{bgz4-fig5}
\end{figure}

In the sequel, a ``$*$" will indicate that the corresponding boundary component is left unfilled.

\begin{lemma}\label{lem4'}
$G$ is incompressible in $X(\sfrac{a}b, *, *)$.
\end{lemma}

\begin{proof}
There is an essential annulus $A_1 \subset R$, disjoint from $t$, with one
boundary component in $S$ and the other having slope $\sfrac01$ on $S_1$; see
Figure \ref{bgz4-fig6}. A component of the inverse image of $A_1$ in $X$ is an
essential annulus with one boundary component on $G$ and the other having slope
$\sfrac01$ on $T_1$. Since $\Delta (\sfrac{a}b, \sfrac01) = a>1$, it follows
from \cite{Sh} and Lemma~\ref{lem3'} that $G$ is incompressible in $X(\sfrac{a}b,
*,*)$.
\end{proof}

\begin{figure}[!ht]
\centerline{\includegraphics{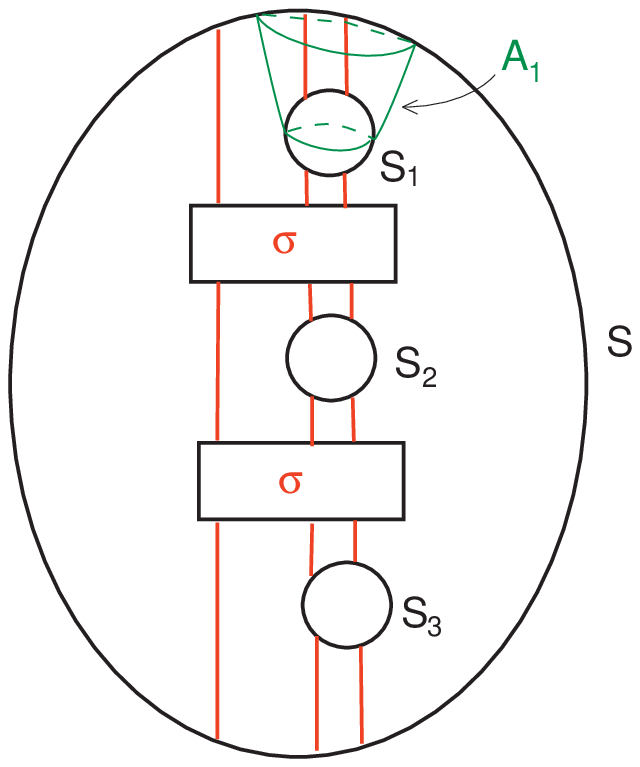}} \caption{ }\label{bgz4-fig6}
\end{figure}

\begin{lemma}\label{lem5'}
$G$ is incompressible in $X(\sfrac{a}b, *,\sfrac{a}b)$.
\end{lemma}

\begin{proof}
There is an essential annulus $A_3\subset R(\sfrac{a}b, *,*)$ with one boundary
component on $S$ and the other having slope $\sfrac01$ on $S_3$. The result now
follows as in the proof of the previous lemma.
\end{proof}

\begin{proof}[Proof of Proposition \ref{prop2}]
There is an essential disk in $\T (\sfrac{a}b\,,\sfrac01\,,\sfrac{a}b)$, meeting
$t(\sfrac{a}b,\sfrac01,\sfrac{a}b)$ in a single point; see Figure
\ref{bgz4-fig7}. Hence $G$ is compressible in
$X(\sfrac{a}b,\sfrac01,\sfrac{a}b)$. Since $\Delta (\sfrac{a}b,\sfrac01) = a>1$,
it follows from Lemma \ref{lem5'} and \cite{Wu2} that either $G$ is
incompressible in $X(\sfrac{a}b,\sfrac{a}b, \sfrac{a}b)$, or there is an
essential annulus $A\subset X(\sfrac{a}b,*, \sfrac{a}b)$ with one boundary
component on $G$ and the other having slope $\sfrac{r}s$ on $T_2$, where $\Delta
(\sfrac{r}s,\sfrac01) = \Delta(\sfrac{r}s,\sfrac{a}b) =1$. We may assume the
latter, in which case, by Dehn twisting $X(\sfrac{a}b,*,\sfrac{a}b)$ along $A$,
we have that $X(\sfrac{a}b,\sfrac{a}b,\sfrac{a}b) \cong
X(\sfrac{a}b,\sfrac01,\sfrac{a}b)$. {From} Figure \ref{bgz4-fig7} we see that
$X(\sfrac{a}b,\sfrac01,\sfrac{a}b)$ is the boundary connected sum of two copies
of $Y$, the double branched cover of the tangle shown in Figure \ref{bgz4-fig8}.

\begin{figure}[!ht]
\centerline{\includegraphics{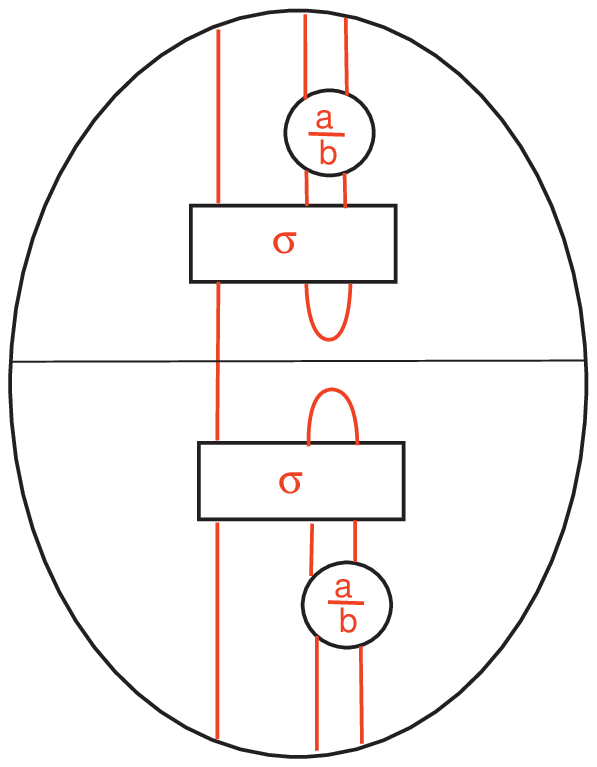}} \caption{ }\label{bgz4-fig7}
\end{figure}

The disk $D$ shown in Figure \ref{bgz4-fig8} separates the tangle into two
rational tangles $\R,\R'$ and lifts to an annulus $A\subset Y$ which separates
$Y$ into two solid tori $U$ and $U'$, the double branched covers of $\R,\R'$
respectively. Note that $A$ has winding number $a$ in $U$. Also, it is easy to
see (by Remark \ref{rem1}) that $A$ is not meridional on $U'$. Hence $Y$ is
either a Seifert fibre space with base orbifold  $D^2 (a,d)$, for some $d>1$, or
a solid torus, giving conclusions (2) and (3) respectively.
\end{proof}

\begin{figure}[!ht]
\centerline{\includegraphics{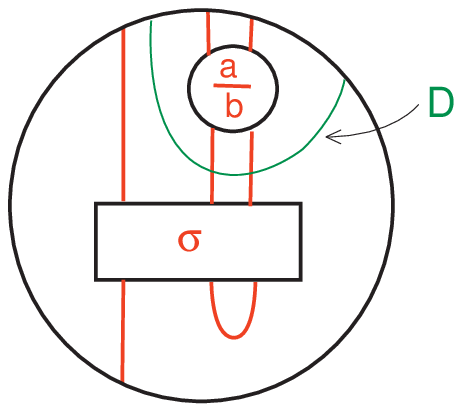}} \caption{ }\label{bgz4-fig8}
\end{figure}

Let $Z$ be the double branched cover of the tangle $(Q, s)$ shown in Figure
\ref{bgz4-fig9}. Then $\partial Z$ has one torus component and two genus two
components.

\begin{lemma}\label{lem6}
$Z(\sfrac{a}b)$ has incompressible boundary.
\end{lemma}

\begin{proof}
For $i=0,1$, there is an annulus $A_i\subset Q$, disjoint from $s$, with one
boundary component on $S_i$ and the other having slope $\sfrac01$ on $S$, as
shown in Figure \ref{bgz4-fig9}. Since $\Delta (\sfrac{a}b,\sfrac01) =a>1$, the
result follows as in the proof of Lemma \ref{lem4'}.
\end{proof}

\begin{figure}[!ht]
\centerline{\includegraphics{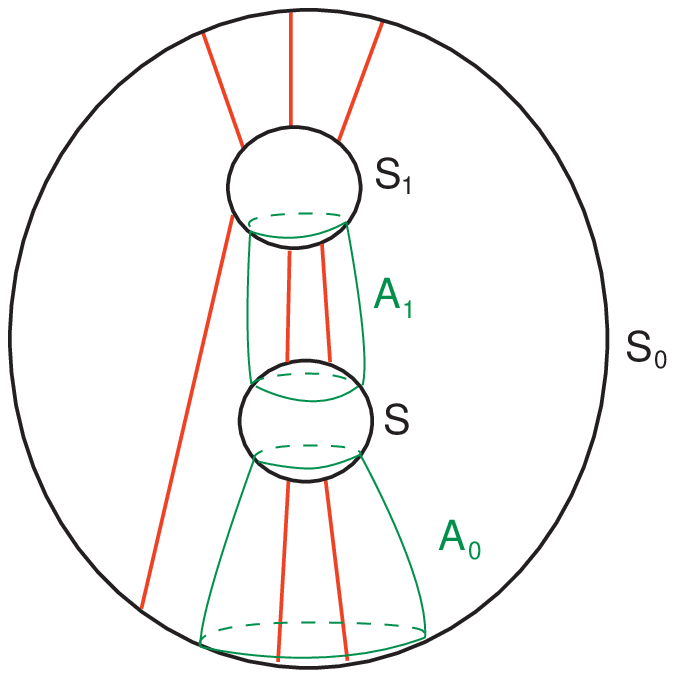}} \caption{ }\label{bgz4-fig9}
\end{figure}

Note that filling $(Q,s)$ along $S$ with the rational tangle $\R(\sfrac10)$
gives a product tangle; hence $Z\cong G\times I- \text{int }N(C)$, where $G$ is
a surface of genus two and $C$ is a simple closed curve $\subset G\times
\{\sfrac12\}$.

\begin{prop}\label{prop7}
The double branched cover $W$ of $(S^3, K)$ either \\
$(1)$ contains a separating incompressible surface of genus~2; or \\
$(2)$ contains four disjoint tori, each cutting off a manifold which is
Seifert fibred over  $D^2(a,d)$, $d>1$; or \\
$(3)$ has Heegaard genus at most 3.
\end{prop}

\begin{figure}[!ht]
\centerline{\includegraphics{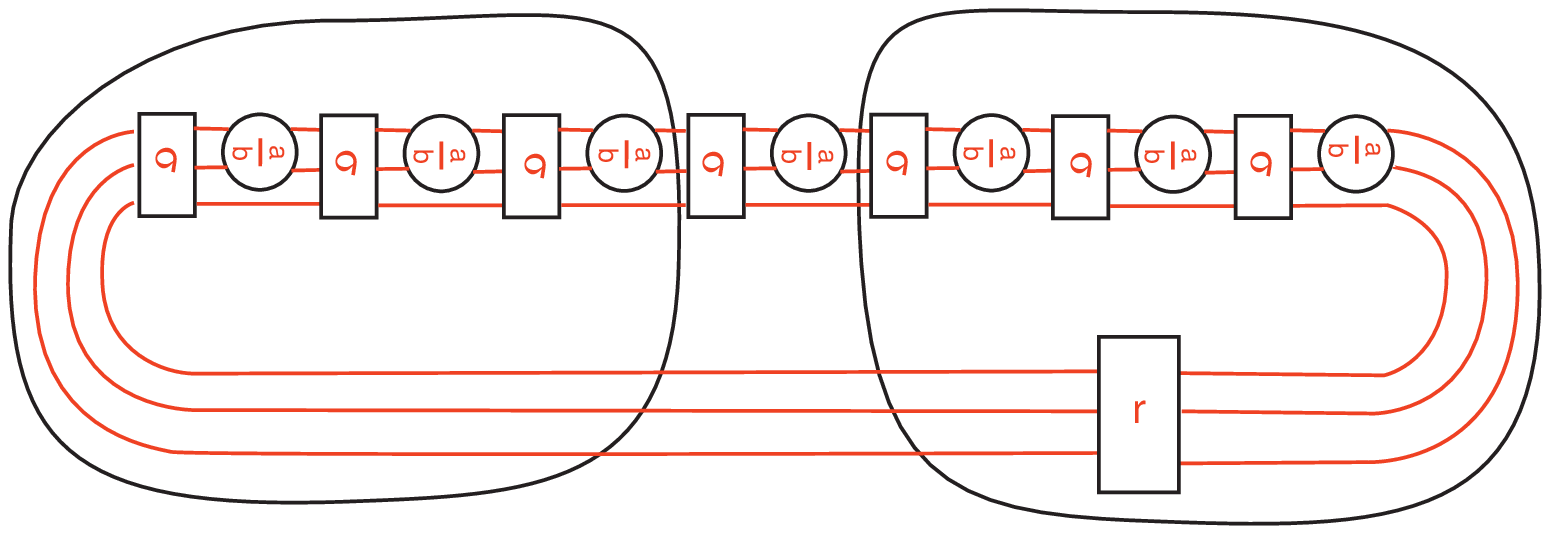}} \caption{ }\label{bgz4-fig10}
\end{figure}

\begin{proof}
{From} Figure \ref{bgz4-fig10} we see that $W\cong P\cup_G Z
(\sfrac{a}b)\cup_{G'} P'$, where $P$ and $P'$ are copies of
$X(\sfrac{a}b,\sfrac{a}b,\sfrac{a}b)$.

Case (1) of Proposition \ref{prop2}, together with Lemma \ref{lem6}, gives
conclusion~(1).

In Case (2) of Proposition~\ref{prop2}, each of $P$, $P'$ contains two disjoint
tori, each cutting off a manifold which is Seifert fibred over $D^2(a,d)$, and
we have conclusion (2).

In Case (3) of Proposition~\ref{prop2}, $P$ and $P'$ are handlebodies of genus
$2$. Also, by the remark after the  proof of Lemma \ref{lem6}, $Z(\sfrac{a}b)$
is obtained from $G\times I$ by Dehn surgery on a curve in
$G\times\{\sfrac12\}$. Hence $W$ is obtained from a closed manifold with a
Heegaard splitting of genus $2$ by a Dehn surgery on a curve in the Heegaard
surface. Since such a curve has tunnel number at most 2, $W$ has Heegaard genus
at most $3$.
\end{proof}

\begin{proof}[Proof of Theorem \ref{thm1}]
To get a contradiction, suppose $W\cong W'$.

Recall that $W'$ is the double branched cover of $(S^3, K')$
and is a Seifert fibred space with base
orbifold
$S^2(m,m,m,m,m,m,m)$.

In Case (1) of Proposition~\ref{prop7},
$W'$ would contain a separating incompressible surface of genus $2$.
This surface would have to be horizontal, and would then separate $W'$
into two twisted $I$-bundles.
Thus $W'$ would contain a non-orientable surface.
But since $W'$ is the double branched cover of a knot in $S^3$,
$H_1 (W';\zed/2)=0$, a contradiction.

In Case (2) of Proposition~\ref{prop7},
the tori in question are incompressible (otherwise $W'$ would have base orbifold
 $S^2 (a,d,r)$ for some $r\geq 1$).
Hence they are vertical in $W'$.
But since $W'$ has only 7 exceptional fibres, this is clearly impossible.

Finally, since $W'$ has base orbifold  $S^2(m,m,m,m,m,m,m)$, every
irreducible Heegaard splitting of $W'$ is either horizontal or vertical
by \cite{MSch}.
It also follows from \cite{MSch}  that when $W'$  has an irreducible horizontal Heegaard splitting,
its genus is bigger than $6$ and that any irreducible vertical  Heegaard splitting
of $W'$ has genus $6$. Hence Case (3) of Proposition \ref{prop7} is
impossible.
\end{proof}

\section{The case $\Delta(\alpha, \beta) = 5$ and $M(\alpha)$ is a lens space}    \label{lens space delta = 5}

In this section we suppose that assumptions \ref{assumptions 1} hold and show that $M(\alpha)$ cannot be a lens space, thus completing our proof of Baker's theorem \cite{Ba}. As we noted at the end of \S \ref{sec: once-punctured}, it suffices to show that the link depicted in Figure \ref{fig3}, considered as lying in a Heegaard solid torus in $L(5, 2q)$, is not isotopic to either the core of a Heegaard solid torus or the boundary of a M\"{o}bius band spine of a Heegaard solid torus.

The proof of the following lemma is straightforward.

\begin{lemma} \label{2=lens}
Let $V_1$ be a Heegaard solid torus in a lens space $L(p, q)$ and let
$K$ be either a core of $V_1$ or a $(2,k)$-cable of a core of $V_1$.
In the first case assume that $p$ is odd.
Then the double branched cover of $(L(p,q),K)$ is a lens space.
\qed
\end{lemma}

\begin{rem}
{\rm The condition that $p$ be odd in the first case is needed to guarantee
the existence of a double branched cover. Furthermore, in that case
we have $L(p,q) \cong L(p,2r) \cong L(p,2r')$, where $4rr' \equiv 1$ (mod
 $p$), and then the double branched cover is homeomorphic to either
$L(p,r)$ or $L(p,r')$.}
\end{rem}

\begin{lemma} \label{d12346}
Let $Q$  be a once-punctured torus bundle over $S^1$, with $\beta$ the boundary
slope of the fibre, and let $\gamma$ be a slope on $\partial Q$ such that
$Q(\gamma)$ is reducible.
Then $\Delta (\beta,\gamma) =1,2,3,4$ or $6$.
\end{lemma}

\begin{proof}
We consider separately three possibilities for $Q$.

(1) {\em $Q$ is hyperbolic.} Here $\Delta (\beta,\gamma) =1$ by \cite[Lemma 4.1]{BZ1}.

(2) {\em $Q$ is Seifert fibred.}  In this case the monodromy of the bundle has finite order, $d$, say, where
$d= 1,2,3,4$ or $6$. If $Q(\gamma)$ is reducible then $\gamma$ is the Seifert fibre slope, and
hence $\Delta (\beta,\gamma)=d$.

(3) {\em $Q$ is toroidal and not Seifert fibred.}
Let $T_0$ be the once-punctured torus fibre of $Q$.
Here the  monodromy of the bundle is $\pm$ the $r$th.\ power of a Dehn twist
along an essential loop $x$ in $T_0$, where $r\ne0$ and $+/-$ denotes
composition with the identity and the elliptic involution, respectively.
The free group $\pi_1 (T_0)$ has basis $\{x,y\}$ with
$[\partial T_0] = [x,y] = xyx^{-1}y^{-1}$.
Then $\pi_1(Q)$ has presentation
\begin{equation}
\langle x,y,t : t^{-1} xt = x,\ t^{-1} yt = yx^r\rangle \tag{i}
\end{equation}
\noindent or
\begin{equation}
\langle x,y,t : t^{-1} xt = (xy)x^{-1} (xy)^{-1},\ t^{-1}yt = x(x^{-r}y^{-1})
x^{-1}\rangle \tag{ii}
\end{equation}
in the $+/-$ cases mentioned above.
In both cases $\pi_1 (\partial Q) = \langle t,[x,y]\rangle$.

For the proof in this case we will use the following lemma.

\begin{lemma} \label{4t=0}
If $A*B$ is a non-trivial free product quotient of $\pi_1(Q)$, then
$4t =0\in H_1 (A*B)$.
\end{lemma}

\begin{proof}
Let $A*B$ be a quotient of $\pi_1(Q)$ with $A\ne 1\ne B$.
We adopt the convention that a word in $x,y$ and $t$ denotes the image in
$A*B$ of the corresponding element of $\pi_1(Q)$.
\smallskip

\noindent {\bf Case (i).}
Here $x$ and $t$ commute.
Hence either
\begin{itemize}
\item[(a)] $x$ and $t$ are powers of some element $z$, or
\item[(b)] $x$ and $t$ lie in a conjugate of a factor.
\end{itemize}

In subcase (a) we have $x=z^m$, $t=z^n$, say.
The second relation in the presentation (i) gives $z^{-n}yz^n = yz^{rm}$,
and therefore $y^{-1}z^n y = z^{n-rm}$.
By applying an inner automorphism of $A*B$ we may assume that $z$ is
represented by a cyclically reduced word in the factors.
It follows that $|n| = |n-rm|$, otherwise we have two cyclically reduced
words, $z^n$ and $z^{n-rm}$, of different lengths in the same conjugacy class.
Hence either $m=0$ or $y^{-1} z^n y = z^{-n}$.
If $m=0$ then $x=1$ and so $A*B$ is a quotient of
$\langle y,t : t^{-1}yt =y\rangle \cong \zed\times\zed$, a contradiction.
If $y^{-1} z^n y = z^{-n}$ then $y^{-1} ty = t^{-1}$ and so
$2t = 0 \in H_1 (A*B)$.

In subcase (b) we may assume, by applying an inner automorphism of $A*B$,
that $x,t \in A$.
Then $y^{-1} t^{-1} y = x^r t^{-1} \in A$.
But $t^{-1}\in A$, and hence $y\in A$.
Therefore $B=1$, a contradiction.
\smallskip

\noindent {\bf Case (ii).}
Let $s= txy$.
Then $\pi_1(Q)$ has the presentation
$$\langle x,y,s : s^{-1} xs = x^{-1},\ s^{-1} ys = y^{-1}x^{-r}\rangle$$
Since $x$ and $s^2$ commute, either
\begin{itemize}
\item[(a)] $x$ and $s^2$ are powers of some element $z$, or
\item[(b)] $x$ and $s^2$ lie in a conjugate of a factor.
\end{itemize}

In subcase (a), suppose $x=z^m$, $s^2 = z^n$.
The second relation in the presentation of $\pi_1(Q)$ implies
$s^{-2} ys^2 = x^r yx^r$, i.e. $z^{-n} yz^n = z^{rm} yz^{rm}$,
giving $y^{-1} z^{(n+rm)} y = z^{n-rm}$.
As in Case~(i) we may assume that $z$ is cyclically reduced, and hence
$|n+rm| = |n-rm|$, i.e. either $m=0$ or $n=0$.
If $m=0$ then $x=1$ and so $A*B$ is a quotient of the Klein bottle group
$\langle y,s :s^{-1} ys = y^{-1}\rangle$, which is easily seen to imply
$A*B\cong \zed_2 *\zed_2$.
If $n=0$ then $s^2 =1$.
Hence $2s =0\in H_1 (A*B)$.
But in $H_1(Q)$ $s= t+x+y$, $2x=0$, and $4y =0$.
Therefore $4t =0 \in H_1(A*B)$.

In subcase (b) we may assume that $x,s^2\in A$.
Hence $s\in A$.
{From} the second relation in the above presentation of $\pi_1(Q)$ we
get $(ys^{-1})^2 = x^{-r} s^{-2} \in A$.
Therefore $ys^{-1} \in A$, and hence $y\in A$.
This implies that $B=1$, a contradiction.~\qed

We now complete the proof of Lemma \ref{d12346}.

Let $\Delta = \Delta (\beta,\gamma)$.
Then $\pi_1 (Q(\gamma))$ is obtained from $\pi_1(Q)$ by adding the
relation $t^\Delta [x,y]^q =1$, for some integer $q$ coprime to $\Delta$.
It is easy to see from the presentations (i) and (ii) that $H_1(Q(\gamma))\not\cong\zed$.
Therefore $Q(\gamma)$ is a non-trivial connected sum and hence $\pi_1(Q(\gamma))$
is a non-trivial free product.
The relation $t^\Delta [x,y]^q =1$ shows
that $t$ has order $\Delta$ in $H_1 (Q(\gamma))$.
Hence by Lemma \ref{4t=0}, $\Delta$ divides $4$.
\end{proof}

Now we complete the proof that $M(\alpha)$ cannot be a lens space under the assumption that the conditions \ref{assumptions 1} hold.
Suppose otherwise.
By Lemma \ref{once-punctured cyclic case}, $M(\alpha)/\tau_\alpha \cong L(5,2q)$,
$L$ is either the core of a Heegaard solid torus in $L(5,2q)$ or
a $(2,k)$-cable of such a core, and furthermore $L(5,2q)$ has a genus~1
Heegaard splitting $V\cup V_0$ such that $L$ is isotopic to a curve in $V$
of the form shown in Figure \ref{fig3}, where $a$ and $b$ are coprime integers with
$a\ge 2$ and $\sigma$ is a 3-braid.
We will show that these conditions on $L$ lead to a contradiction.

Remove from the solid torus $V$ in Figure \ref{fig3} the interior of the 3-ball $B$
containing the $a/b$-rational tangle.
We then get a tangle $\T$ in
$Y = (V- \text{int }B)\cup V_0 = L(5,2q) \setminus \text{int }B$.
Let $X$ be the double branched cover of $(Y,\T)$.

Since $\T(a/b) = L$, by Lemma \ref{2=lens} we have
\begin{itemize}
\item $X(a/b)$ is a lens space.
\end{itemize}

Also, clearly $\T(0/1) = (\text{core of }V) \, \#\, (\text{knot in }S^3)$,
so
\begin{itemize}
\item $X(0/1) \cong L(5,r)\,\#\, N$ for some closed 3-manifold $N$.
\end{itemize}

\begin{lemma}\label{lem6.12}
$X(1/k)$ is irreducible for all $k\in\zed$.
\end{lemma}

\begin{proof}
$\T(1/k)$ is (the 3-braid $\sigma_1^k\sigma$ in $V$) $\cup V_0$.
Hence $X(1/k) = Q_k\cup \widetilde V_0$, where $Q_k$ is the double branched
cover of $(V,\sigma_1^k \sigma)$ and $\widetilde V_0$ is a solid torus.
Now $Q_k$ is a $T_0$-bundle over $S^1$, where $T_0$ is the double branched
cover of ($D^2$, 3~points), i.e. a once-punctured torus.
Let $\beta$ be the boundary slope of the fibre of $Q_k$; note that
$\beta$ projects to the meridian $\mu$ of $V$.
Let $\mu_0,\tilde\mu_0$ be the meridians of $V_0,\widetilde V_0$, respectively.
Since $\Delta (\mu,\mu_0) =5$, we have $\Delta (\beta,\tilde\mu_0)=5$.
 Hence by Lemma 8.3, $X(1/k)$ is irreducible.
\end{proof}

There is a $\zed/2$-action on $X$ with quotient $Y = L(5,2q) \setminus \text{int }B$.
It follows easily that $X$ is not a solid torus.
We consider the following three possibilities for $X$.

\noindent
(1) {\em $X$ is reducible.}
Here we must have $X\cong X'\,\#\, X(a/b)$, where $X' (a/b)\cong S^3$.
By Lemma \ref{lem6.12}, $X'(1/k) \cong S^3$ for infinitely many $k$, and hence $X'$
is a solid torus with meridian $0/1$.
Since $\Delta (a/b,\,0/1) = a>1$, this contradicts the fact that
$X'(a/b)\cong S^3$.

\noindent
(2) {\em $X$ is irreducible and not Seifert fibred.}
Since $\Delta (a/b,\, 0/1) = a>1$, the forms of $X(a/b)$ and $X(0/1)$
stated above contradict \cite{CGLS} if $N\cong S^3$ and \cite[Corollary 1.4]{BZ2} otherwise.

\noindent
(3) {\em $X$ is Seifert fibred with incompressible boundary.}

If $X$ is not the twisted $I$-bundle over the Klein bottle let $\varphi$
be the slope on $\partial X$ of the Seifert fibre in the unique Seifert
fibring of $X$.
If $X$ is the twisted $I$-bundle over the Klein bottle let $\varphi$ be the
slope of the Seifert fibre in the Seifert structure on $X$ with orbifold
$D^2 (2,2)$.
In both cases, $\varphi$ is the only slope on $\partial X$ such that
$X(\varphi)$ is a non-trivial connected sum.
Therefore, if $N\not\cong S^3$, then $\varphi =0/1$.
But $X(a/b)$ is a lens space, and so $\Delta (a/b,\, 0/1)=1$,
contradicting our assumption that $a>1$.
Hence $N\cong S^3$, and so
$\Delta (a/b,\, \varphi) = \Delta (0/1,\, \varphi) =1$.
In particular $\varphi = 1/s$ for some integer $s$.
Therefore $X(1/s)$ is reducible.
But this contradicts Lemma \ref{lem6.12}.
\end{proof}

\section{The case $\Delta (\alpha,\beta) =6$ and the involution $\tau_\alpha$ reverses the orientations of the Seifert fibres of $M(\alpha)$}
\label{sec6.3}

In this section we suppose that assumptions \ref{assumptions 1} hold and show that it is impossible for $\Delta(\alpha, \beta)$ to be $6$ and for $\tau_\alpha$ to reverse the orientations of the Seifert fibres of $M(\alpha)$. We assume otherwise in order to obtain a contradiction. Here $M(\alpha)/\tau_\alpha = L(3,q) \cong L(3,1)$.
By Lemma \ref{types}, $L$ is as shown in Figure \ref{fig3}.
By Lemma \ref{--quotient 2} parts (2) and (3), $n$ is even, $m$ is odd,
$|L| =1$, and $L_\alpha = L\cup K_\alpha$ is as shown in Figure \ref{bgz4-fig24}.
Since $L$ is a component of $L_\alpha$,  we see that $L$ is a core of some Heegaard solid torus
of $L(3,1)$.
Hence the double branched cover of $(L(3,1),L)$ is homeomorphic to $L(3,1)$.

Let $Y,\T,X$ be as in the previous section,
with $L(5,2q)$ replaced by $L(3,1)$.
Again as in that proof, here we have $X(a/b) \cong L(3,1)$ and
$X(0/1) \cong L(3,1)\,\#\, N$ for some closed 3-manifold $N$.
In the current situation we only have the following  weaker version of Lemma \ref{lem6.12}.

\begin{lemma}\label{lem6.12-weaker}
$X(1/k)$ is irreducible for infinitely many $k\in\zed$.
\end{lemma}

\begin{proof} As in the proof of Lemma \ref{lem6.12},
$\T(1/k)$ is (the 3-braid $\sigma_1^k\sigma$ in $V$) $\cup V_0$,
 and $X(1/k) = Q_k\cup \widetilde V_0$, where $Q_k$
 is the double branched
cover of $(V,\sigma_1^k \sigma)$ and $\widetilde V_0$ is a solid torus.
Now $Q_k$ is a $T_0$-bundle over $S^1$, where $T_0$ is  a once-punctured torus.
If $\rho \in B_3$, let $\tilde \rho$ denote the corresponding homeomorphism
$T_0\to T_0$.
Then $\tilde\sigma_1$ and $\tilde\sigma_2$ are Dehn twists about a pair
of curves in $T_0$ with intersection number~1.
With respect to this basis, $\tilde\rho$ defines an element of $SL_2(\zed)$.
Note that since $L$ is connected, $\sigma$ is not a power of $\sigma_1$.
The elements of $SL_2(\zed)$ corresponding to $\tilde\sigma_1^k$ and
$\tilde\sigma$ are therefore
$\left[\begin{smallmatrix} 1&k\\ 0&1\end{smallmatrix}\right]$
and
$\left[\begin{smallmatrix} a&b\\ c&d\end{smallmatrix}\right]$, say, where
$c\ne 0$.
Then the matrix corresponding to $\tilde\sigma_1^k\tilde\sigma$ has
trace $a+d+kc$, which has absolute value greater than 2 for all but at
most five values of $k$.
For such $k$ the manifold $Q_k$ is therefore hyperbolic.

Let $\beta$ be the boundary slope of the fibre  of $Q_k$; note that $\beta$
projects to the meridian $\mu$ of $V$.
Let $\mu_0,\tilde\mu_0$ be the meridians of $V_0$, $\widetilde V_0$,
respectively.
Since $\Delta (\mu,\mu_0)=3$, we have $\Delta (\beta,\tilde\mu_0)=3$.
If $Q_k$ is hyperbolic, then by \cite[Lemma 4.1]{BZ1}
$Q_k(\gamma)$ reducible implies  $\Delta (\beta,\gamma)=1$.
Therefore $X(1/k) = Q_k(\tilde\mu_0)$ is irreducible for infinitely many $k$.
\end{proof}

As in the previous section,
we have possibilities (1), (2) and (3) for $X$.
Cases~(1) and (2) are ruled out exactly as before (applying Lemma \ref{lem6.12-weaker}
instead of Lemma \ref{lem6.12}).
In case~(3) we may conclude that  both $X(a/b)$ and $X(0/1)$ are $L(3,1)$, $X(1/s)$
 is reducible for some integer $s$ and $\Delta (\beta,\tilde\mu_0)=3$.
 The proof of
Lemma \ref{d12346} shows that the monodromy of the once-punctured torus bundle
$Q_s$ has order~3.
Therefore $Q_s$ has base orbifold $D^2 (3,3)$, and so
$X(1/s)\cong Q_s (\tilde\mu_0) \cong L(3,q_1)\,\#\, L(3,q_2)$.
This implies that $X$ has base orbifold $D^2(3,3)$.
But then no two distinct fillings on $X$ can give the lens space $L(3,1)$, yielding a contradiction.

\section{The case $\Delta(\alpha, \beta) = 5$ and the involution $\tau_\alpha$ reverses the orientations of the Seifert fibres of $M(\alpha)$} \label{delta = 5}

In this section we suppose that assumptions \ref{assumptions 1} hold and show that it is impossible for $\Delta(\alpha, \beta)$ to be $5$ and for $\tau_\alpha$ to reverse the orientations of the Seifert fibres of $M(\alpha)$. We assume otherwise in order to obtain a contradiction.

As in section \ref{dist7}, we just need to show that the two knots,  $K,K'$,
shown in Figures \ref{bgz4-5ab}  and \ref{bgz4-5mn} respectively are
inequivalent in $S^3$.

\begin{figure}[!ht]
\centerline{\includegraphics{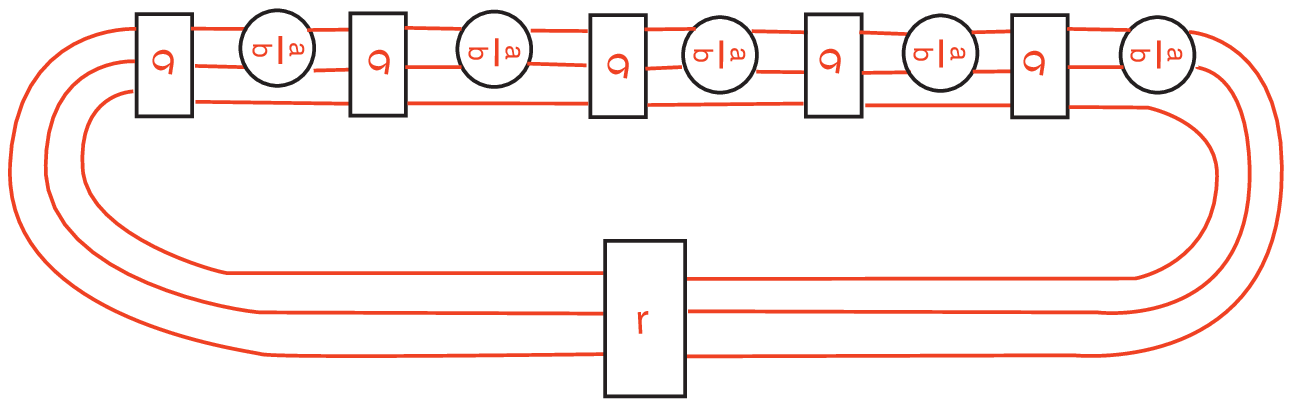}} \caption{ }\label{bgz4-5ab}
\end{figure}

\begin{figure}[!ht]
\centerline{\includegraphics{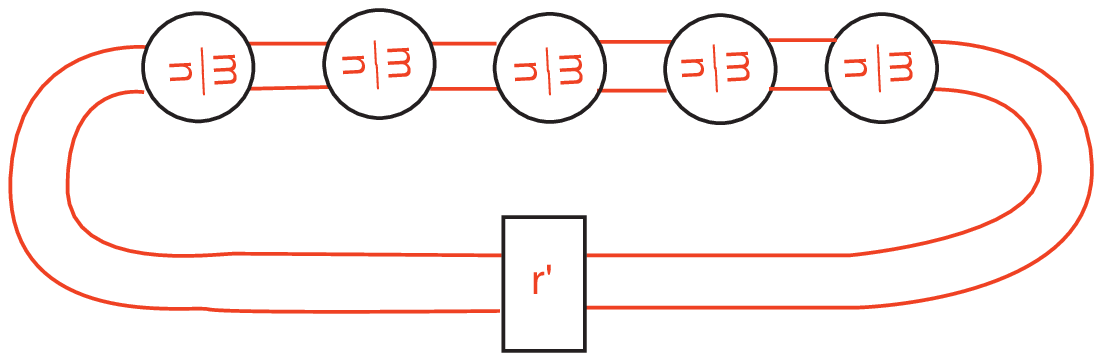}} \caption{ }\label{bgz4-5mn}
\end{figure}

\begin{thm}\label{thm8}
The knots $K$ and $K'$ are inequivalent.
\end{thm}

As in Subsection \ref{dist7}, we will show that the double branched covers $W,
W'$ of $(S^3,K), (S^3,K')$ are not homeomorphic.

Here we consider the tangle $\T = (R,t)$ shown in Figure \ref{bgz4-fig11}, with
double branched cover $X$. Let the boundary components of $R$ be $S,S_1,S_2$
(see Figure \ref{bgz4-fig11}), and the corresponding boundary components of $X$
be $G,T_1,T_2$, so that $T_1$ and $T_2$ are tori and $G$ has genus two.

\begin{figure}[!ht]
\centerline{\includegraphics{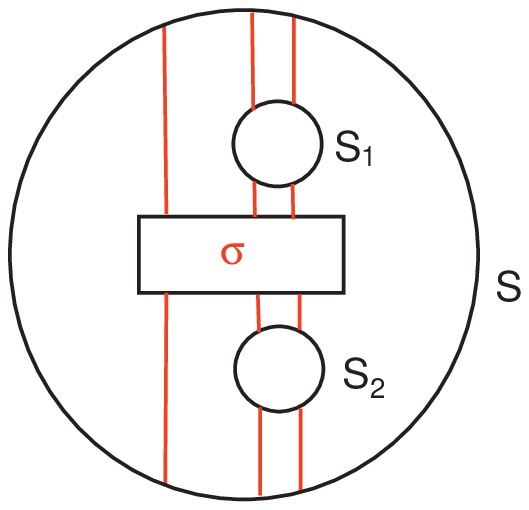}} \caption{ }\label{bgz4-fig11}
\end{figure}

\begin{lemma}\label{lem9}
If $G$ is compressible in $X$ then $\T$ is isotopic to the tangle shown in
Figure \ref{bgz4-fig12}.
\end{lemma}

\begin{figure}[!ht]
\centerline{\includegraphics{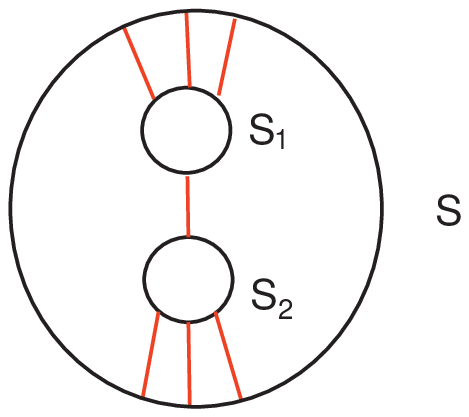}} \caption{ }\label{bgz4-fig12}
\end{figure}

\begin{proof}
Since $t\cup S_1 \cup S_2$ is connected, any essential disk $D$ in $\T$ with
$\partial D \subset S$ must meet $t$ in a single point. Hence $D$ meets the
unique strand of $t$ connecting $S_1$ and $S_2$, decomposing $\T$ into two
tangles $\T_1$ and $\T_2$. We claim that each of $\T_1$ and $\T_2$ is a product
tangle. To see this, note that deleting the strand of $t$ that joins $S_2$ to
$S$ and runs through the braid $\s$ gives the tangle shown in Figure
\ref{bgz4-fig13}. It follows that $\T_1$ is as stated. Similarly, $\T_2$ is also
a product tangle.
\end{proof}

\begin{figure}[!ht]
\centerline{\includegraphics{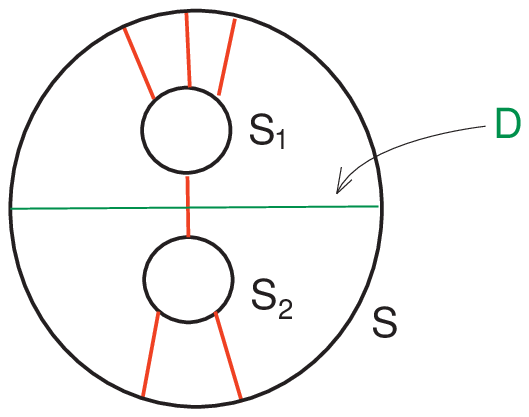}} \caption{ }\label{bgz4-fig13}
\end{figure}

\begin{cor}\label{cor10}
If $G$ is compressible in $X$ then $X(\sfrac{a}b,\sfrac{a}b)$ is a genus~2
handlebody.
\end{cor}

\begin{lemma}\label{lem11}
If $G$ is incompressible in $X$ then $G$ is incompressible in
$X(\sfrac{a}b,\sfrac{a}b)$.
\end{lemma}

\begin{proof}
This is exactly like the proof of Lemma \ref{lem5'} in section \ref{dist7},
using the annuli $A_1$ and $A_2$ shown in Figure \ref{bgz4-fig14}.
\end{proof}

\begin{figure}[!ht]
\centerline{\includegraphics{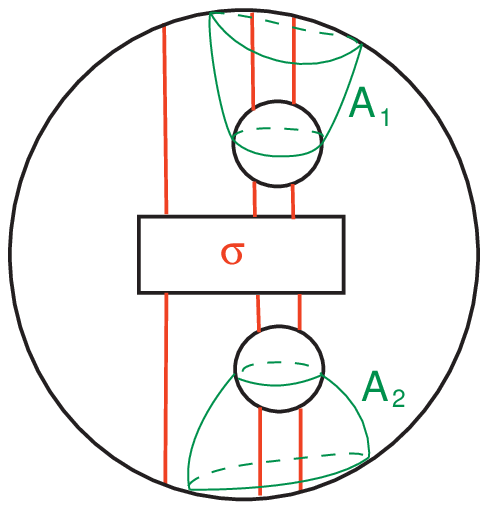}} \caption{ }\label{bgz4-fig14}
\end{figure}

\begin{prop}\label{prop12}
$W$ either \newline $(1)$ contains a separating incompressible surface of genus
$2$; or \newline $(2)$ has Heegaard genus at most $3$.
\end{prop}

\begin{proof}
From Figure \ref{bgz4-fig15} we see that $W\cong U \cup_G Z(\sfrac{a}b)
\cup_{G'} U'$, where $U$ and $U'$ are copies of $X(\sfrac{a}b,\sfrac{a}b)$.

\begin{figure}[!ht]
\centerline{\includegraphics{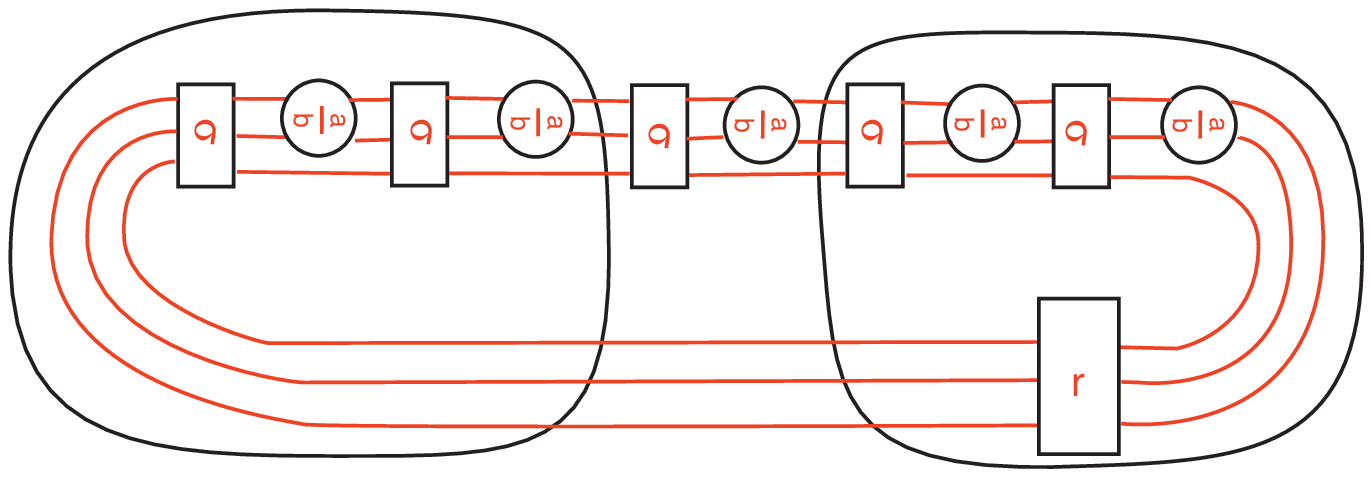}} \caption{ }\label{bgz4-fig15}
\end{figure}

If $G$ is incompressible in $X$ then we get conclusion (1) by Lemmas \ref{lem11}
and \ref{lem6}.

If $G$ is compressible in $X$ then we get conclusion  (2) by
Corollary~\ref{cor10} and the proof of part (3) of Proposition \ref{prop7}.
\end{proof}

\begin{proof}[Proof of Theorem~\ref{thm8}]
Assume $W\cong W'$. Since $W'$ is a Seifert fibre space over $S^2$ with 5
exceptional fibres, we get a contradiction to Proposition~\ref{prop12} as in the
proof of Theorem~\ref{thm1} in Cases~(1) and (3) of Proposition~\ref{prop7}.
\end{proof}

\section{A family of examples realizing $\Delta(\alpha, \beta) = 4$} \label{d=4}

We show in this section that distance $4$ between a prism manifold filling slope and a once-punctured torus slope
can be realized on infinitely many hyperbolic knot manifolds.

\begin{figure}[!ht]
\centerline{\includegraphics{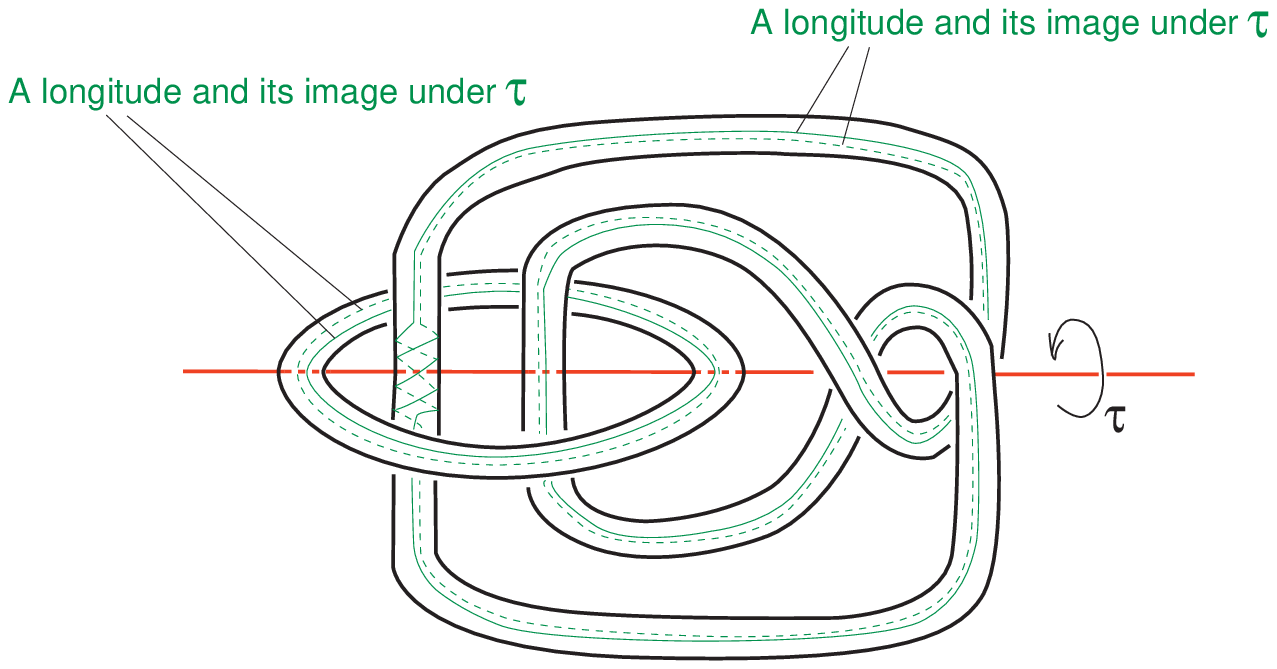}}
\caption{}\label{bgz4-whitehead}
\end{figure}

Let $Wh$ be the exterior of the Whitehead link with standard meridian-longitude
coordinates on $\partial Wh$. We use $Wh(\gamma)$ to denote the manifold of
Dehn filling one boundary component of $Wh$ with slope $\gamma$, and $Wh(\gamma,\delta)$ the
manifold of Dehn filling one boundary component with slope $\gamma$ and the other
with slope $\delta$.

\begin{thm}
For each integer $n$ with $|n|>1$, $ Wh(\frac{-2n\pm1}{n})$ is a hyperbolic
knot manifold whose $0$-slope is the boundary slope of an essential
once-punctured torus and whose $-4$-slope yields a prism manifold whose base orbifold is $S^2(2,2, |\mp2n-1|)$.
\end{thm}

\begin{proof} It is well known that $Wh(\gamma)$ is hyperbolic for each  $\gamma\notin\{ -1,
-2, -3, -4, 0, 1/0\}$. That $Wh(\gamma)$, $\gamma\ne 1/0$, contains an
essential once-punctured torus with boundary slope $0$ is obvious from the
Whitehead link diagram.

The Whitehead link admits an  involution $\tau$ as shown in Figure
\ref{bgz4-whitehead}. This involution restricts to an involution, still denoted
$\tau$, on $Wh$ and then extends to an involution $\tau_\gamma$ on $Wh(\gamma)$
and to an involution $\tau_{\gamma,\delta}$ on $Wh(\gamma,\delta)$ for all slopes $\gamma$
and $\delta$. The quotient space under $\tau$ is shown in Figure
\ref{bgz4-whitehead2}. Note that the branch set of $Wh(\gamma)/\tau_\gamma$
is obtained by removing the two $1/0$-tangles in Figure
\ref{bgz4-whitehead2} and then filling one $\gamma$-tangle. Figure
\ref{bgz4-whitehead3} shows the branch set in $Wh(-4)/\tau_{-4}$ and Figure
\ref{bgz4-whitehead4} shows the branch set in
$Wh(\frac{-2n\pm1}{n},-4)/\tau_{\frac{-2n\pm1}{n},-4}$. As the branch set in
$Wh(\frac{-2n\pm1}{n},-4)/\tau_{\frac{-2n\pm1}{n},-4}=S^3$ is a Montesinos link of type
$ (2,2, \frac{\mp 2n-1}{2})$, the double branched cover
$Wh(\frac{-2n\pm1}{n},-4)$ is a prism manifold whose base orbifold is $S^2(2,2, |\mp2n-1|)$.
\end{proof}

\begin{figure}[!ht]
\centerline{\includegraphics{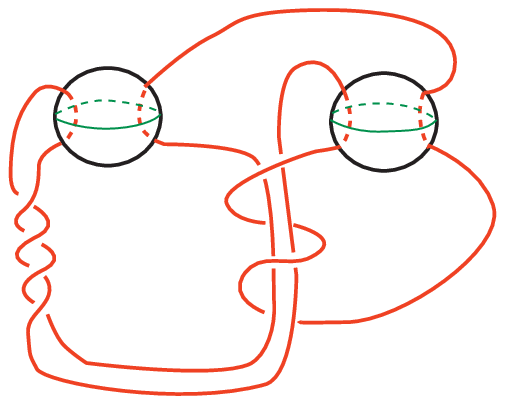}}
\caption{}\label{bgz4-whitehead2}
\end{figure}

\begin{figure}[!ht]
\centerline{\includegraphics{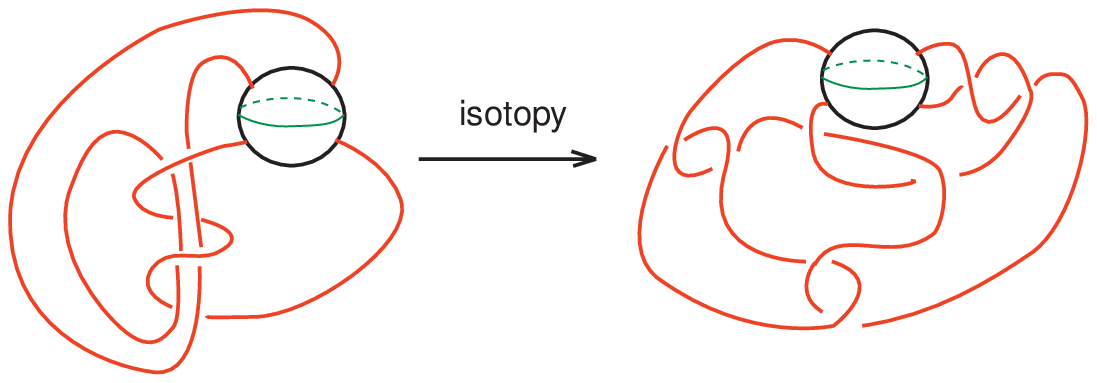}}
\caption{}\label{bgz4-whitehead3}
\end{figure}

\begin{figure}[!ht]
\centerline{\includegraphics{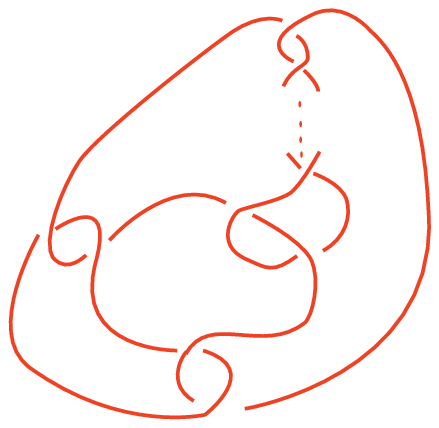}}
\caption{}\label{bgz4-whitehead4}
\end{figure}

\section{The case when $\D(\alpha,\beta)=4$ and $M(\a)$ is a prism manifold}\label{prism-section}

In this section we show
\begin{thm}\label{unique in prism case}Let $M$ be a hyperbolic knot exterior containing an essential
 once-punctured torus
with slope $\b$. If  $M(\a)$ is a prism manifold with $\D(\alpha,\beta)=4$, then
$M$ is one of the examples given in \S \ref{d=4},
that is, $\displaystyle (M; \alpha, \beta)\cong (Wh(\frac{-2n\pm1}{n}); -4, 0)$
for some integer $n$ with  $|n|>1$.
\end{thm}

Let $F$ be an essential once-punctured torus in $M$ with slope $\b$.
Choose a Klein bottle $\hat P$ in $M(\a)$ which has the minimal number of intersection
components with $\partial M$ and let $P=M\cap \hat P$. Then  $p=|\partial P|>0$ since $M$ is hyperbolic.
 The punctured Klein bottle $P$ is essential in $M$, i.e. it is incompressible
and boundary-incompressible in $M$.
The proof of this statement is essentially contained
in \cite[Proofs of Lemmas 2.1 and 2.2]{Te2} and we only need to add the
condition that $M(\a)$ is a prism manifold which is thus
irreducible and does not contain a projective plane.

As usual, the two surfaces $F$ and $P$ define
two labeled intersection graphs which we denote by $\G_F$ and $\G_P$.
Then neither $\G_F$ nor $\G_P$ contain trivial loops (\cite[Lemma 3.1]{Te2} with the same proof).
The graph $\G_F$ has a unique vertex whose valency is $4p$, and the graph $\G_P$ has $p$ vertices
each having valency $4$. Note that every edge of $\G_F$ is positive since $F$ is orientable and has only one boundary component.

\begin{lemma}\label{3properties}
(1) When $p\geq 2$, $\G_F$ has no $S$-cycle.
\newline
(2) When $p\geq 3$, $\G_F$ has no generalized $S$-cycle (See \cite{Te2} for its definition).
\newline
(3) $\G_F$  cannot have more than $\displaystyle\frac{p}{2} + 1$ mutually parallel edges.
\end{lemma}

\begin{proof}
Part (1)  is  \cite[Lemma 3.2]{Te2} with the same proof,
part (2) is  \cite[Lemma 3.3]{Te2} with a similar argument plus the fact that $M(\a)$ does not contain
projective plane, and part (3) is \cite[Lemma 6.2 (4)]{LT} with the same proof.
\end{proof}

\begin{lemma}
$p=1$.
\end{lemma}

\begin{proof}
The lemma was proved in \cite[Lemma 5.2]{Te2} when $M$ was a genus one non-cabled
knot exterior in $S^3$, in which case $p$ was an odd integer.
 In our situation, we need to extend the argument of
\cite[Lemma 5.2]{Te2} slightly, using Lemma \ref{3properties} (3) instead of
\cite[Lemma 3.4]{Te2}.

\begin{figure}[!ht]
\centerline{\includegraphics{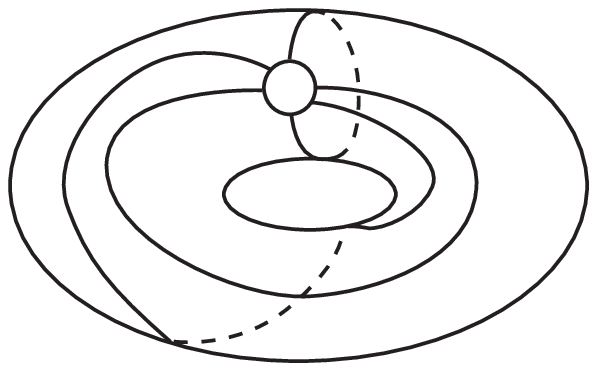}} \caption{ }\label{prism1}
\end{figure}

Suppose otherwise  that
$p\geq 2$.
The reduced graph $\overline{\G}_F$  is a subgraph of the graph shown in Figure \ref{prism1} (\cite[Lemma 5.1]{Go1}).
In particular $\overline{\G}_F$ has at most three edges. Suppose these edges of $\overline{\G}_F$
have weights $w_k$, $k=1,2,3$, some of which may possibly be zero.
Then $2(w_1+w_2+w_3)=4p$.
Let $e_1,...,e_{w_k}$ be a parallel family of consecutive edges in $\G_F$.
Reading the labels around the vertex of $\G_F$, we see that the labels of the
edges $e_1,e_2,..,e_{w_k}$ are
as illustrated in Figure \ref{prism2}.

\begin{figure}[!ht]
\centerline{\includegraphics{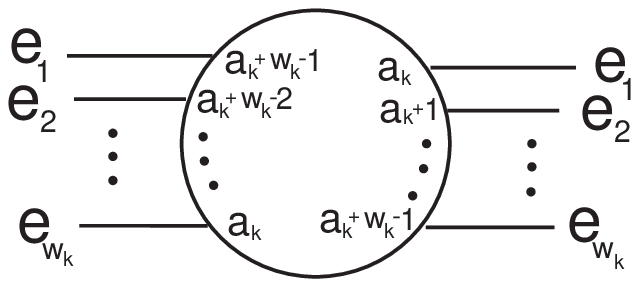}} \caption{ }\label{prism2}
\end{figure}

By Lemma \ref{3properties}, $w_k = 0$ or $1$, $1 \leq k \leq 3$. (More precisely, this follows
from Lemma \ref{3properties} (1) if $w_k$ is even, Lemma \ref{3properties} (2) if
$w_k$ is odd and $p \ge 3$,
and Lemma \ref{3properties} (3) if $w_k$ is odd and $p = 2$.) This is a contradiction.
\end{proof}

So $\G_F$ has exactly two edges and  both are
level edges (i.e. having the same label at the two endpoints of the edge).
Let $e_1, e_2$ be the two edges of $\G_F$ and
 of $\G_P$.
Note that each $e_i$  is an orientation-reversing loop in $P$ by the parity rule.

\begin{figure}[!ht]
\centerline{\includegraphics{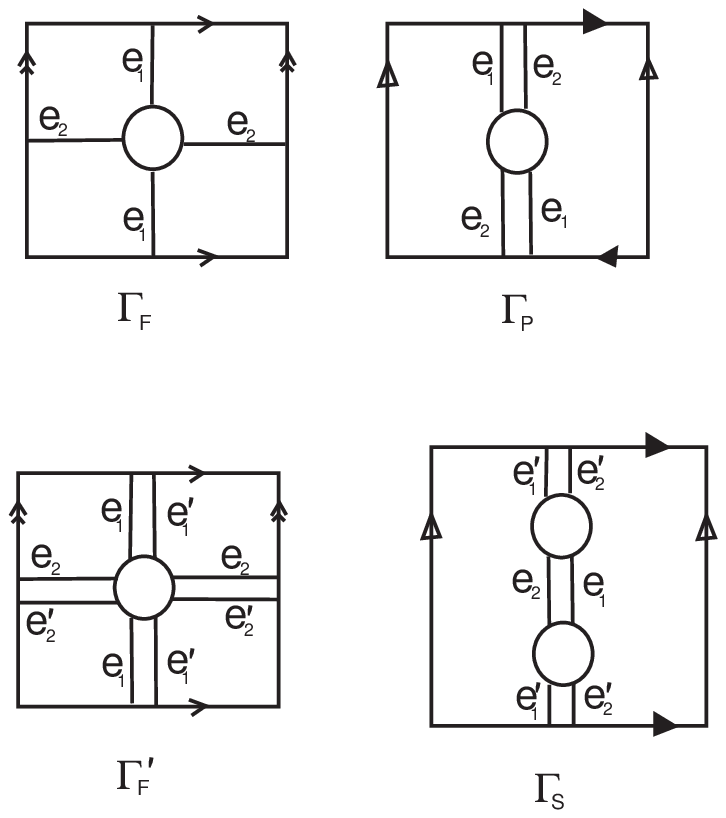}} \caption{ }\label{prism3}
\end{figure}

\begin{figure}[!ht]
\centerline{\includegraphics{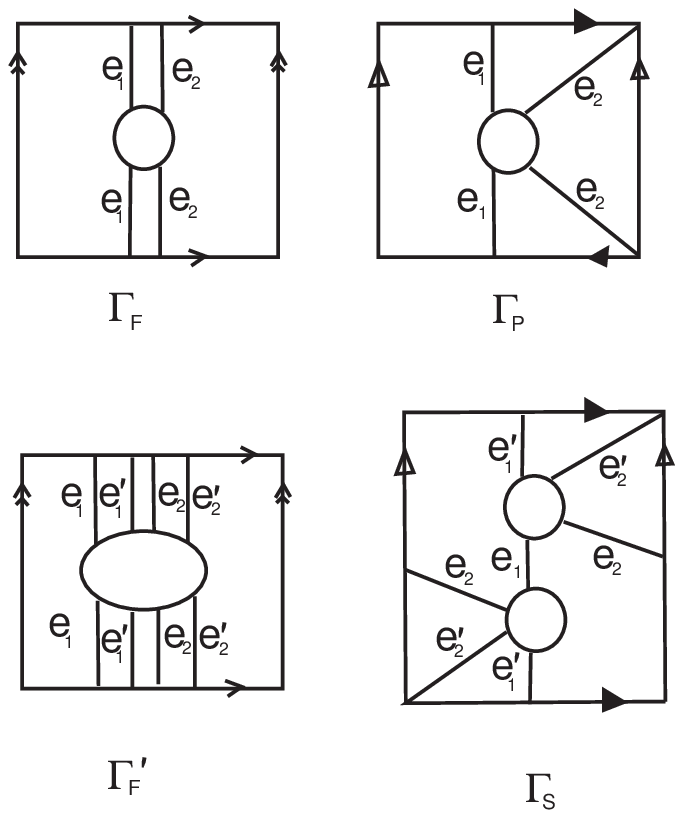}} \caption{ }\label{prism4}
\end{figure}

Since $\D(\alpha,\beta)=4$, if the endpoints of the two edges around the vertex $\partial F$ are
labeled consecutively
by $1,2,3,4$,
the labels around $\partial P$ are also consecutive.
It follows from this fact that
if the two edges
in $\G_F$ are not parallel, then the two edges in $\G_P$ must be parallel.
Also,
combining this fact with the proof of \cite[ Lemma 4.1]{Te2}, we have
that the two edges $e_1$ and $e_2$ cannot be parallel in both $\G_P$ and $\G_F$.
So there are only two possible configurations for
the pair of the graphs $\G_F$ and $\G_P$,  which we illustrate in
Figure \ref{prism3} and Figure \ref{prism4} respectively.

Let $S$ be the frontier of a thin  regular neighborhood of $P$ in $M$.
Then $S$ is a separating  twice-punctured torus in $M$.
The surfaces $F$ and $S$ define two labeled intersection graphs
$\G_F'$ and $\G_S$. Note that $\G_F'$ is obtained by doubling the edges of $\G_F$ and $\G_S$
 double covers $\G_P$.
 See Figures \ref{prism3} and \ref{prism4} for  illustrations of the graphs $\G_F, \G_P,\G_F'$ and $\G_S$.

The surface $S$ separates $M$ into two components which we denote by
$X^+$ and $X^-$, where $X^-$ is a twisted $I$-bundle over $P$.
Note that $\widehat X^-$ is a twisted $I$-bundle over the Klein bottle
$\widehat P$ and   $\widehat X^+$ is a solid torus since $M(\a)=\widehat X^-\cup \widehat X^+$ is a prism manifold.
Let $H^\e$ denote the part of the filling solid torus of $M(\a)$
contained in $\widehat X^\e$, $\e\in\{\pm\}$ and let $\partial_0 H^\e=\partial H^\e\cap \partial M$.

We first show

\begin{lemma}
The case given by Figure \ref{prism3} cannot occur.
\end{lemma}

\begin{proof}
The $4$-gon face of $\G_F'$, which we denote by $f$,  is contained in $X^+$
and its boundary edges form a Scharlemann cycle of order $4$.
From Figure \ref{prism3} we see that $\partial f$ is a non-separating curve
in the genus two surface $S\cup \p_0 H^+$ and
 $\partial f\cap S$ is contained in an essential annulus $A$ in $\widehat S$.
Let $U$ be a regular neighborhood of
$A\cup H^+\cup f$ in $\widehat X^+$.
Then $U$ is a compact $3$-manifold with $\partial U$ a torus
and the fundamental group of $U$  has the following presentation
$$<x,t; x^3txt=1>$$
where we take a fat base  point in $\widehat S$ containing $\partial S\cup e_1\cup e_2$, $x$ is
a based loop formed by a cocore arc of $\p_0 H^+$ and $t$ is represented by a core circle of $A$.
 Let $y=xt$, then
 $$\pi_1(U)=<x,y; x^2y^2=1>.$$
 So $U$ is Seifert fibred with base orbifold $D(2,2)$.
 Thus $U$ contains a Klein bottle. But $U$ is contained in the solid torus $\widehat X^+$.
 This gives a contradiction.
\end{proof}

So the case of Figure \ref{prism4} must occur.
In this case we are going to show that $M$ is obtained by
Dehn filling one boundary component of the Whitehead link exterior.

In this case, the bigon faces of $\G_F'$ between $e_1$ and $e_1'$ and between $e_2$ and $e_2'$
lie in $X^-$, and the bigon face between $e_1'$ and $e_2$, which we denote by $B$,
is contained in $X^+$.
Let $Q$ be a regular neighborhood of
$S\cup \p_0 H^+\cup B$ in $X^+$, and $\widehat Q=Q\cup H^+$.
Then it's easy to see that
$\widehat Q$ is a Seifert fibred manifold whose base orbifold is
an annulus with a single cone point of order $2$.
The boundary of $\widehat Q$ consists of two tori, one of which is the torus
$\widehat S$.
Let $T_0$ be the other component.
Note that $T_0$ is contained in the interior of $X^+$.
Since $\widehat X^+$ is a solid torus,
 $T_0$ must bound a solid torus in $\widehat X^+\setminus \widehat Q$, which we denote by
$N$.

\begin{lemma}\label{not match}The Seifert structure of $\widehat Q$ does not match
with the Seifert structure of $\widehat X^-$ whose base orbifold is
$D(2,2)$.
\end{lemma}

\begin{proof}
The $S$-cycle $\{e_1, e_1'\}$ in $\G_F'$ implies  that as a cycle in $\G_S$,
$e_1\cup e_1'$ is a fibre  of the  Seifert structure of $\widehat X^-$ whose base orbifold is
$D(2,2)$.
Similarly the $S$-cycle $\{e_1', e_2\}$ in $\G_F'$ implies that as a cycle in $\G_S$,
$e_1'\cup e_2$ is a fibre  of the  Seifert structure of $\widehat Q$.
Obviously from Figure \ref{prism4} these two cycles have different slopes
in $\widehat S$.
\end{proof}

Let $\displaystyle W=X^-\cup_S Q$.
Note that $\displaystyle M=W\cup_{T_0} N$. So we just need to show
that $W$ is the Whitehead link exterior. We use the notation $W(\partial M, \gamma)$ to denote the Dehn filling of $W$
along a slope $\gamma$ in $\partial M \subset \partial W$.

\begin{lemma}\label{4properties}
(1) $W$ is irreducible.\newline
(2) The twice-punctured torus $S$ is incompressible  in $W$.
\newline
(3) $F\cap W$ has a component which is  an essential once-punctured annulus in $W$
with the puncture lying in $\partial M$ of slope $\b$
and with the boundary of the annulus lying in $T_0$.
\newline
(4) $W(\partial M, \a)$ contains an essential torus which is $\widehat S$.
\end{lemma}

\begin{proof}
By the construction of $Q$, one can easily see that $Q$ is irreducible and $S$ is incompressible in
$Q$. Obviously $X^-$ is irreducible and $S$ is incompressible in $X^-$.
Thus $S$ is incompressible in $\displaystyle W=X^-\cup_S Q$ and $W$ is irreducible. So we get (1) and (2).

Part  (3) follows from the graph $\G_F'$ shown in Figure \ref{prism4} and the construction of $Q$.
In fact  the  exterior in $F$ of  the annulus which is the annulus face of $\G_F'$  shrunk slightly
into the interior of the face is
the required punctured annulus. It is incompressible in $W$ because it is an essential subsurface of $F$.
It is boundary incompressible in $W$ because it has only one intersection component with $\partial M$
and $M$ does not contain an essential  disk with slope $\b$.

For (4),  we just need to note that $W(\partial M, \a)=\widehat X^-\cup_{\widehat S}\widehat Q$.
\end{proof}

 \begin{lemma}
$W$ is hyperbolic.
\end{lemma}

\begin{proof} We already know that $W$ is irreducible (Lemma \ref{4properties} (1)).
Obviously $W$ cannot be  Seifert fibred since $M=W\cup N$ is hyperbolic.
So we just need to show that $W$ is atoroidal.
Suppose otherwise that $W$ contains an essential torus $T$.
Note that $T$ is separating since $M$ is hyperbolic.

Note that $Q$  (a compression body) is  of the form $T_0\times [0, 1]$
union a $1$-handle attached to $T_0\times {1}$. It is now easy to see
that any incompressible torus  in $Q$ is isotopic into $T_0\times [0,1]$, and
therefore boundary parallel.
Hence $T$ cannot be contained in  $Q$.
Obviously $X^-$ is atoroidal because it is a twisted $I$-bundle over a punctured Klein bottle.
So $T$ cannot be contained in $X^-$ either.
Therefore $T$ must intersect $S$.
As $S$ is incompressible in $W$ (Lemma \ref{4properties} (2)), we may assume that every component of
$S\cap T$ is a circle which is  essential in both $T$ and $S$.
As $S$ is separating, $T\cap S$ has even number of components.
We may further assume that each component of $T\setminus (S\cap T)$
is an essential annulus in $(X^-, S)$ or in $(Q, S)$ (using isotopy of $T$ to eliminate inessential ones), and thus can be further assumed
to be a vertical annulus in the characteristic $I$-bundle of $(X^-, S)$
or $(Q,S)$.
Note that the  characteristic $I$-bundle for  the pair $(Q,S)$ is
isotopic to a regular neighborhood of $B\cup \p_0 H^+$ in $Q$ such that
the  horizontal boundary of the  $I$-bundle
 is a twice-punctured annulus $\phi$ contained in $S$
 such that $\widehat \phi$  is an essential annulus in $\widehat S$, and the vertical
 boundary of the $I$-bundle has two components: one is $\p_0H^+$  and the other is
 the frontier of the $I$-bundle in $Q$.
So we may assume that $S\cap T$ is contained in $\phi$.

\begin{figure}[!ht]
\centerline{\includegraphics{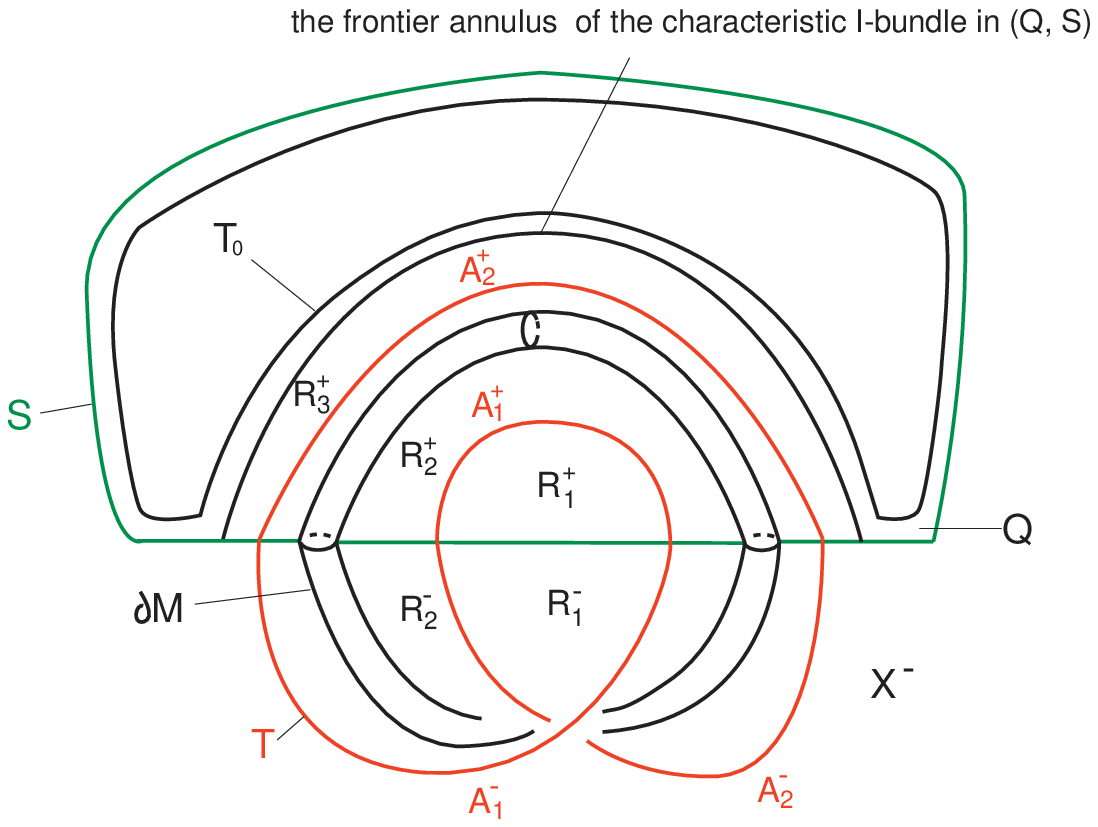}} \caption{ }\label{prism5}
\end{figure}

Let $A$ be a component of $T\setminus (T\cap S)$.
It's easy to see that
$\partial A$ is $\widehat S$-essential for otherwise $A$ would be isotopic to $\p_0 H^\e$ and
 $T$ would be parallel to $\partial M$.
Now if $A$ is contained in $Q$, its two boundary components are either isotopic in $\phi$ to the
two inner boundary components of $\phi$ respectively or bound an annulus in $\phi$ which separates
$\phi$ into two once-punctured annuli. Moreover
$A$ is a vertical annulus in the Seifert fibred structure
of $\widehat Q$. If $A$ is contained in $X^-$, it is a vertical annulus
in one of the two Seifert fibred structures of $\widehat X^-$.
So the Seifert structure of $\widehat Q$ matches a Seifert structure
of $\widehat X^-$. By Lemma \ref{not match}, the Seifert structure
of $\widehat X^-$ must be the one whose base orbifold is a M\"{o}bius band.
Thus if a component $A$ of $T\setminus (S\cap T)$ is contained in $X^-$,
it is a non-separating annulus in $X^-$. In particular if $A$ is contained in $X^-$,
$\partial A$ cannot be parallel in $S$. For otherwise the union of $A$ with
the annulus in $S$ bounded by $\partial A$  would be a Klein bottle in $W\subset M$,
giving a contradiction.

With the above information we have obtained on the components of $T\setminus (S\cap T)$
we see that the following case must occur: $T\setminus (S\cap T)$ has exactly four components,
two in $Q$ which we denote by $A_1^+$ and $A_2^+$, and two in $X^-$ which we denote by
$A_1^-$ and $A_2^-$, and they are connected as shown in Figure \ref{prism5}.
More specifically the annuli $A_1^+$ and $A_2^+$ separate $Q$ into three components $R_1^+, R_2^+, R_3^+$
such that $R_1^+$ is a solid torus in which  $A_1^+$ has winding number $2$, $R_2^+$
contains  $\p_0H^+$ and is a product $I$-bundle over a once-punctured
annulus, and $R_3$ is a regular neighborhood of $T_0$. The annuli $A_1^-$ and $A_2^-$ separate $X^-$ into two components $R_1^-, R_2^-$ such that $R_1^-$ contains $\p_0H^-$ and is a product
$I$-bundle over a once-punctured
annulus, and $R_2^-$ is a solid torus. (cf. Figure \ref{prism5}).
Moreover  $R_2^+\cup R_2^-$ is a once-punctured annulus bundle over $S^1$ with finite
order monodromy and thus is Seifert fibred. In fact one can see that the monodromy has order two.
On the other hand $R_1^+\cup R_1^-\cup R_3^+$ is  Seifert fibred over an annulus with
one cone point of order two.
Hence $W$ is a graph manifold. But $M=W\cup N$ is hyperbolic. We get a contradiction.
\end{proof}

\begin{lemma}   $W(\partial M, \beta)$ contains an essential annulus which is the cap off of
the once-punctured annulus given  in part (3) of Lemma \ref{4properties}.
\end{lemma}

\begin{proof}
Note that the punctured annulus given in part (3) of Lemma \ref{4properties}
is non-separating in $W$.
So it caps off to a non-separating annulus in $W(\partial M,\beta)$.
If this annulus is inessential in $W(\partial M,\beta)$, then it must be compressible,
from which  we  may get a compressing disk for $T_0$ in $W(\partial M,\beta)$.
That is, $\b$ becomes   a boundary-reducing Dehn filling slope on $\partial M$ for $W$.
On the other hand, $\a$ is a toroidal filling slope on $\partial M$ for $W$
by Lemma \ref{4properties} (4).
Hence by \cite{GL}, we have $\D(\alpha,\beta)\leq 2$. But this contradicts the assumption of  $\D(\alpha,\beta)=4$.
Thus the above annulus is essential in $W(\partial M, \beta)$.
\end{proof}

Now we have shown that $W$ is hyperbolic, and for $(W,\partial M)$,  $\a$ is  a
toroidal filling slope and $\b$ an annular filling slope. Furthermore
$W(\partial M,\beta)$ contains an essential annulus whose intersection with $\partial M$ has only one component.
Applying \cite[Theorem 1.1]{GW2}, we see that $W$ is the Whitehead link exterior.

So $W \cong Wh$.
By tubing off the once-punctured annulus  in $W$ (given by Lemma \ref{4properties} (3)) with an annulus in $T_0$,
we get a once-punctured torus in $(W, \partial M)$ with slope $\b$.
So $\b$ corresponds to the zero slope with respect to the standard coordinates
on $\partial Wh$. Similarly we see that $\a$ is the slope $-4$.
As $\widehat Q$ is Seifert fibred over an annulus with a single cone point,
$\widehat X^+=\widehat Q\cup_{T_0} N$ is a solid torus if and only if the filling slope on $T_0$ is distance
one from the Seifert slope of $\widehat Q$ on $T_0$.
This Seifert slope is unique.
From the examples given in \S \ref{d=4}, we see that the Seifert slope of $\widehat Q$ on $T_0$ is $-2$
and those examples are the only examples realizing Theorem \ref{once-punctured} (1).
That is, we have $\displaystyle (M; \alpha,\beta)\cong (Wh(\frac{-2n\pm1}{n}); -4,0)$
for some integer $n$ with $|n|>1$.

\section{Proof of Theorems \ref{once-punctured-exceptional} and  \ref{genusones3}} \label{sec: genus 1}

\begin{proof}[Proof of Theorem \ref{once-punctured-exceptional}]
Let $M$ be a hyperbolic knot manifold containing an essential once-punctured torus $F_\beta$ with boundary slope $\beta$. Let $\gamma$ be an exceptional slope on $\partial M$.

We may suppose that the capped-off torus $\hat F_\beta$ is incompressible in $M(\beta)$ by Proposition \ref{compresses}. Now $M(\gamma)$ is either reducible, small Seifert, or toroidal. In the first case $\Delta(\beta,\gamma) = 1$ by \cite[Lemma 4.1]{BZ1}, while in the second case Theorem \ref{once-punctured} implies that $\Delta(\beta,\gamma) \leq 5$
with equality only if $(M; \gamma, \beta) \cong (Wh(-3/2); -5, 0)$ and $M(\gamma)$ has base orbifold $S^2(2,3,3)$, and $\Delta(\beta,\gamma) = 4$ only if  $(M;\gamma,\beta)\cong (Wh(\frac{-2n\pm1}{n}); -4,0)$
 for some integer $n$ with $|n|>1$  and $M(\gamma)$ has base orbifold $S^2(2,2,|\mp 2n-1|)$.

So suppose that $M(\gamma)$ is toroidal. We then have a punctured torus $F_\gamma$ in $M$ with boundary slope $\gamma$, such that the capped-off torus $\hat F_\gamma$ in $M(\gamma)$ is incompressible. Assume that $n_\gamma$, the number of boundary components of $F_\gamma$, is minimal over all such punctured tori. Similarly, assuming for the moment only that $M(\beta)$ is toroidal, we have a punctured torus $F_\beta$ in $M$ with boundary slope $\beta$ and $n_\beta$ boundary components.
Triples $(M;F_\beta,F_\gamma)$ of this kind with $\Delta(\beta,\gamma) \ge 4$ are classified in \cite{Go1} (in the case
$\Delta(\beta,\gamma) \ge 6$) and \cite{GW} (in the case $\Delta(\beta,\gamma) = 4$ or 5). In particular, it is shown in \cite{GW} that if $M$ is a hyperbolic knot manifold with a once-punctured torus slope $\beta$ and a toroidal slope $\gamma$ with $\Delta(\beta,\gamma) = 4$, then $(M;\gamma, \beta) \cong (Wh(\delta); -4, 0)$ for some slope $\delta$ on the other boundary component of $Wh$. This proves part (3)(a) of the theorem.

The only examples with $n_\beta = 1$ and $\Delta(\beta,\gamma) \ge 5$ are $M = Wh(-5/2)$, with $\Delta(\beta,\gamma) = 7$ \cite{Go1}, and $M = M_5$ or $M_{10}$ in \cite{GW}, with $\Delta(\beta,\gamma) = 5$.
In fact the only examples with $\Delta(\beta,\gamma) = 5$ where $M(\beta)$ (say) contains a non-separating torus are $M_5, M_{10}$ and $M_{11}$ (see \cite[Lemma 23.1]{GW}). Now in \cite{MP} three examples of hyperbolic knot manifolds are given, each with a pair of toroidal fillings at distance 5, one of which contains a non-separating torus: these are $Wh(-7/2), Wh(-4/3)$ and $N(-5,5)$, described in Tables A.3, A.4 and A.9, respectively. By comparing the description in these tables of the second toroidal filling at distance 5 with that given in \cite[Lemma 22.2]{GW}, we see that $Wh(-7/2) = M_{10}$, $N(-5,5) = M_{11}$, and (hence) $Wh(-4/3) = M_5$. It is well-known that $Wh(\delta)$ contains a once-punctured essential torus of slope $0$. The determination of the slopes $\gamma, \beta$ as listed in parts (3)(b) and (3)(c) has been done by Martelli and Petronio. See \cite[Tables A.2 and A.3]{MP}.
\end{proof}

\begin{proof}[Proof of Theorem \ref{genusones3}]
Let $K \subset S^3$ be a hyperbolic knot of genus one with exterior $M_K$ and suppose $p/q$ is an exceptional filling slope on $\partial M_K$ where $q \geq 1$.

Hyperbolic genus one knots in the $3$-sphere do not admit reducible surgery slopes \cite{BZ1}, so an exceptional surgery slope is either toroidal or irreducible, atoroidal, small Seifert. If $K$ is fibred, it is necessarily the figure eight knot, and the theorem holds in this case. Assume that $K$ is not a fibred knot. Then \\
\indent \hspace{.3cm} (a) $M_K(0)$ is not fibred \cite{Ga}

\indent \hspace{.3cm} (b) $K$ admits no $L$-space surgery \cite{Ni}

\indent \hspace{.3cm} (c) $K$ is not a Eudave-Mu\~noz knot \cite{E-M}

A genus one Seifert surface for $K$ completes to an essential torus in $M_K(0)$ \cite{Ga}. Suppose that $M_K(0)$ is Seifert fibred. As its first homology group is infinite cyclic, its base orbifold must have underlying space $S^2$ and $M_K(0)$ must have non-zero Euler number. Thus it admits a non-separating, horizontal surface, which implies $M_K(0)$ fibres over the circle, contrary to (a). Thus $M_K(0)$ is not Seifert fibred, so assertion (1) of the theorem holds.

By (b), $K$ has no finite surgery slopes. Thus if $M_K(p/q)$ is small Seifert with base orbifold $S^2(a,b,c)$, then $p \ne 0$ and $(a,b,c)$ is either a Euclidean or hyperbolic triple, so $|p| \leq 3$ by Theorem \ref{once-punctured}. Consideration of $H_1(S^2(a,b,c))$ shows that $(a,b,c)$ is a hyperbolic triple. Hence assertion (2) of the theorem holds.

Theorem \ref{once-punctured} combines with (b) and assertion (2) to show that if $M_K(p/q)$ is small Seifert then $0 < |p| \leq 3$. Thus assertion (3) of the theorem holds.

Since $K$ is not a Eudave-Mu\~noz knot, each toroidal slope of $K$ is integral. It follows from \cite{Go1} and \cite{Te1} that no genus one knot in the $3$-sphere admits a toroidal filling slope of distance $5$ or more from the longitude. Such knots with toroidal slopes of distance $4$ are determined in \cite[Theorem 24.4]{GW}. In particular, all such knots are twist knots and the non-longitudinal slope is $\pm 4$. This proves assertion (4).
\end{proof}

%%%%%%%%%%%%%%%%%%%%%%%%%%%%%%%%%%%%%%%%%%%
\def\bysame{$\underline{\hskip.5truein}$}
%%%%%%%%%%%%%%%%%%%%%%%%%%%%%%%%%%%%%%%%%%%


\begin{thebibliography}{CCGLS}\label{refs}


{\scriptsize

\bibitem[Ba]{Ba} K. Baker, {\it Once-punctured tori and knots in lens spaces}, Comm. Anal. Geom. {\bf 19} (2011), 347--399.


\bibitem[BiMe]{BiMe}
J. Birman and Wm. Menasco, {\it Studying links via closed braids III}, Pac.. J. Math.  {\bf 161} (1993), 25-113.

\bibitem[BCSZ1]{BCSZ1}
S. Boyer, M. Culler, P. Shalen, and X. Zhang,
{\it Characteristic submanifold theory and Dehn filling},
Trans. Amer. Math. Soc. {\bf 357} (2005), 2389--2444.

\bibitem[BCSZ2]{BCSZ2} \bysame,         %% -------,
{\it Characteristic subvarieties, character varieties, and Dehn fillings},
Geometry \& Topology 12 (2008) 233-297.

\bibitem[BGZ1]{BGZ1}
S. Boyer, C. McA. Gordon and X. Zhang,
{\it Dehn fillings of large hyperbolic 3-manifolds},
J. Diff. Geom. {\bf 58} (2001), 263--308.

\bibitem[BGZ2]{BGZ2}
\bysame, {\it Characteristic submanifold theory and toroidal Dehn filling}, to appear in Adv. Math.

\bibitem[BGZ3]{BGZ3}
\bysame, {\it Dehn fillings of knot manifolds containing essential twice-punctured tori}, in preparation.

\bibitem[BZ1]{BZ1} S. Boyer and X. Zhang,
 {\it Reducing Dehn filling and toroidal Dehn filling},
Topology Appl. 68 (1996) 285-303.

\bibitem[BZ2]{BZ2} \bysame,
{\it On Culler-Shalen seminorms and Dehn fillings},
Ann. Math. {\bf 148} (1998), 737--801.
%\vspace{-.15cm}

\bibitem[CGLS]{CGLS}
M. Culler, C. M. Gordon, J. Luecke and P. Shalen,
{\it Dehn surgery on knots},
Ann. of Math. {\bf 125} (1987) 237--300.

\bibitem[CJR]{CJR}
M. Culler, W. Jaco and H. Rubinstein,
{\it Incompressible surfaces in once-punctured torus bundles},
Proc. Lond. Math. Soc. {\bf 45} (1982) 385--419.

\bibitem[Du]{Du}
Wm. Dunbar,
{\it Geometric orbifolds},
Rev. Mat. {\bf 1} (1988) 67--99.


\bibitem[E-M]{E-M} M. Eudave-Mu\~noz, {\it Non-hyperbolic manifolds obtained by Dehn surgery on hyperbolic knots},   Geometric topology (Athens, GA, 1993),  35--61, AMS/IP Stud. Adv. Math., 2.1, Amer. Math. Soc., Providence, RI, 1997.

\bibitem[FH]{FH} W. Floyd and A. Hatcher, {\it Incompressible surfaces in punctured-torus bundles}, Top. Appl. {\bf 13} (1982), 263--282.

\bibitem[FKP]{FKP} D. Futer, E. Kalfagianni, and J. Purcell, {\it Cusp areas of Farey manifolds and applications to
knot theory}, to appear in Int. Math. Res. Not.

\bibitem[Ga]{Ga} D. Gabai, {\it Foliations and the topology of $3$-manifolds II}, J. Diff. Geom. {\bf 26} (1987), 461--478.

\bibitem[Go1]{Go1}
C. McA. Gordon,
{\it Boundary slopes of punctured tori in $3$-manifolds},
Trans. Amer. Math. Soc. {\bf 350} (1998) 1713--1790.

\bibitem[Go2]{Go2}
\bysame,
{\it Dehn filling: a survey}, in Knot Theory, Banach Center Publications 42, Institute of Mathematics, Polish Academy of Sciences, Warsaw, 1998, 129--144.

\bibitem[GLi]{GLi}
C. McA. Gordon and R. A. Litherland,  {\it Incompressible surfaces in branched
coverings}, The Smith conjecture,  Pure Appl. Math., 112, Academic Press (1984)
139--152.

\bibitem[GL]{GL} C. McA. Gordon and J. Luecke, {\em Toroidal and boundary-reducing Dehn Fillings}, Topology Appl. 93 (1999) 77-90.

\bibitem[GW]{GW}
C. McA. Gordon and Y.-Q. Wu,
{\em Toroidal Dehn fillings on hyperbolic 3-manifolds},
Mem. Amer. Math. Soc. {\bf 194} (2008), no.909.


\bibitem[GW2]{GW2} \bysame,
 {\em Toroidal and annular Dehn fillings}, Proc. London Math. Soc.
78 (1999) 662-700.


\bibitem[HR]{HR}
C. Hodgson and H. Rubinstein, {\it Involutions and isotopies of lens spaces}, in Knot Theory and Manifolds, ed. D. Rolfsen, Lecture Notes in Mathematics {\bf 1144}, Springer-Verlag, 1983, 60--96.


\bibitem[KT]{KT} P. K.  Kim and J. L.  Tollefson, {\it  Splitting the PL
involutions of nonprime $3$-manifolds},  Michigan Math. J. 27 (1980) 259--274.

\bibitem[LM]{LM} M.~Lackenby and R.~Meyerhoff, {\em The maximal
    number of exceptional Dehn surgeries}, preprint (2008),
  arXiv:0808.1176.

\bibitem[L1]{L1}
S. Lee,
{\em Dehn fillings yielding Klein bottles},
Int. Math. Res. Not. 2006,
Art. ID 24253, 34pp.


\bibitem[L2]{L2}
\bysame,
{\em Klein bottle and toroidal Dehn fillings at distance~5},
Pacific J. Math. {\bf 247} (2010), 407--434.

\bibitem[L3]{L3}
\bysame,
{\em Lens spaces and toroidal Dehn fillings},
Math. Z. {\bf 267} (2011), 781--802.


\bibitem[LT]{LT} S. Lee and  M. Teragaito, {\em Boundary structure of hyperbolic 3-manifolds admitting
annular and toroidal fillings at large distance}, Canad. J. Math. 60 (2008) 164-188.

\bibitem[MP]{MP}
B. Martelli and C. Petronio,
{\em Dehn filling of the ``magic'' 3-manifold},
Comm. Anal. Geom. {\bf 14} (2006), 969--1026.

\bibitem[MSch]{MSch} Y. Moriah and J. Schultens,  {\it
 Irreducible Heegaard
splittings of Seifert fibred spaces are either vertical or horizontal},
Topology 37 (1998) 1089--1112.

\bibitem[Ni]{Ni} Y. Ni, {\it
Knot Floer homology detects fibred knots}, Invent. Math. 170 (2007), no. 3,
577--608.

\bibitem[Oe]{Oe}
U. Oertel,
{\em Closed incompressible surfaces in complements of star links},
Pacific J. Math. {\bf 111} (1984), 209--230.

\bibitem[Oh]{Oh}
S. Oh,
{\it Reducible and toroidal manifolds obtained by Dehn filling},
Top. Appl. {\bf 75} (1997), 93--104.

\bibitem[Sh]{Sh}
H. Short,
{\em Some closed incompressible surfaces in knot complements which survive
surgery}, in Low dimensional topology, ed. Roger Fenn, London Math. Soc. Lecture Notes 95, Cambridge University Press, 1985, 179--194.

\bibitem[Te1]{Te1}
M. Teragaito,
{\em Distance between toroidal surgeries on hyperbolic knots in the 3-sphere}, Trans. Amer. Math. Soc. {\bf 358} (2006), 1051--1075.


\bibitem[Te2]{Te2} \bysame,
{\em Creating Klein bottles by surgery on knots}, J. Knot Theory Ramifications
10 (2001) 781-794.

\bibitem[Wu1]{Wu1}
Y.-Q. Wu,
{\it Incompressibility of surfaces in surgered 3-manifolds}, Topology, 31 (1992)
271--279.

\bibitem[Wu2]{Wu2}
\bysame,
{\it Dehn fillings producing reducible manifolds and toroidal manifolds},
Topology {\bf 37} (1998), 95--108.


\bibitem[YM]{YM} S-T Yau and W.  Meeks, {\it
 The equivariant loop theorem for three-dimensional manifolds
 and a review of the existence theorem
 for minimal surfaces}, in
  The Smith conjecture,  Pure Appl. Math., 112, Academic Press (1984)
153--163.
}

\end{thebibliography}
\end{document}